%% file: main.tex
\documentclass[11pt,reqno]{article}

\input{preamble}
\usepackage{amsmath,amssymb,xfrac}
    \usepackage{enumitem} \setlist{nosep}
\usepackage{ifthen} 
\usepackage[pagebackref,breaklinks,unicode]{hyperref} 
    \renewcommand*{\backrefalt}[4]{\ifcase #1 (Not cited).\or (Cited p.~#2).\else (Cited pp.~#2).\fi} 
\usepackage{graphicx} 
\usepackage{mathrsfs}
\usepackage[T1]{fontenc}
\usepackage[margin=3cm]{geometry}
\usepackage{natbib}

\newcommand*{\ext}{\mathrm{ext}}
\newcommand*{\N}{\mathbb{N}}

\makeatletter
\newtheorem*{rep@theorem}{\rep@title}
\newcommand{\newreptheorem}[2]{%
\newenvironment{rep#1}[1]{%
 \def\rep@title{#2 \ref{##1}}%
 \begin{rep@theorem}}%
 {\end{rep@theorem}}}
\makeatother

\newreptheorem{theorem}{Theorem}
\newreptheorem{example}{Example}
\newreptheorem{fact}{Fact}

\newcommand*{\ssm}{\smallsetminus}

\definecolor{harrycomment}{rgb}{0.6,0,0.4}

\definecolor{shakedcomment}{rgb}{0, 0, 255}

\definecolor{oussamacomment}{rgb}{0,0.5,0}

\title{Quasiisometric embeddings between right-angled Artin groups: rigidity}
\hypersetup{pdftitle = Quasiisometric embeddings of RAAGs: rigidity}
\date{}
\author{Shaked Bader, Oussama Bensaid, and Harry Petyt}

\newcommand{\Addresses}{{\bigskip\footnotesize\par
\textsc{Mathematical Institute, University of Oxford, UK}\par\nopagebreak\textit{E-mail address}: 
\texttt{shaked.bader@sjc.ox.ac.uk}
\par\medskip\par
\textsc{Institut de Recherche en Mathématique et en Physique, Université Catholique de Louvain, Belgium} \par\nopagebreak\textit{E-mail address}: 
\texttt{oussama.bensaid@uclouvain.be}
\par\medskip\par
\textsc{Mathematics Institute, University of Warwick, UK}\par\nopagebreak\textit{E-mail address}: 
\texttt{harrypetyt@gmail.com}
}}

\begin{document}

\maketitle

\begin{abstract}
We prove a range of rigidity results for quasiisometric embeddings between right-angled Artin groups. 
Among other things, we:

\begin{itemize}
\item   classify all self--quasiisometric-embeddings of RAAGs defined on cycles;

\item   show that no $n$--dimensional RAAG is a universal receiver for quasiisometric embeddings of $n$--dimensional RAAGs;

\item   prove that if the direct product $F_2^n\times A_{C_5}^m$ can be quasiisometrically embedded in a RAAG $A_\Gamma$ of the same dimension, then the defining graph of $F_2^n\times A_{C_5}^m$ must appear as an induced subgraph of $\Gamma$.

\end{itemize}
We also establish a strong rigidity theorem for quasiisometrically embedded images of 2--flats in RAAGs defined by triangle-free graphs that are not stars, generalising a theorem of Bestvina--Kleiner--Sageev.

\end{abstract}

{\hypersetup{hidelinks}\setcounter{tocdepth}{1}\tableofcontents\setcounter{tocdepth}{2}}

\section{Introduction} 

A group is a \emph{right-angled Artin group} if it has a generating set where the only relations are that some generators commute. If there are no such relations then the group is free, whereas if every two generators commute then the group is free-abelian. Between these extremes there is a large variety of examples, with perhaps the most obvious being direct products of free groups, or free products of free-abelian groups. In fact, each right-angled Artin group can be encoded by a simplicial graph, and the correspondence is bijective \cite{droms:isomorphisms}.

Because of the simplicity of the relations in right-angled Artin groups, they have some very nice algebraic and combinatorial properties \cite{hermillermeier:algorithms,servatius:automorphisms,davisjanuszkiewicz:right,hsuwise:onlinear}. See \cite{charney:introduction} for a survey. However, as is perhaps to be expected from a class of groups that is naturally indexed by the set of all simplicial graphs, their simple definition belies a rich spectrum of behaviours.

This is perhaps best illustrated by their subgroup structure. Indeed, the highly influential class of \emph{compact special} groups of Haglund--Wise can be characterised as the \emph{convex-cocompact} subgroups of right-angled Artin groups \cite{haglundwise:special}. Such groups have a powerful associated theory, particularly in the case where they are hyperbolic \cite{wise:structure}. Another famous example is the result of Bestvina--Brady, who used subgroups of right-angled Artin groups to show that the Whitehead and Eilenberg--Ganea conjectures are incompatible \cite{bestvinabrady:morse}.

In this article, we are interested in the large-scale geometry of right-angled Artin groups, and in how such groups relate to one another.

\subsection*{Large-scale geometry of groups}

The study of groups through their large-scale geometry is one of the central themes in geometric group theory, with the most ambitious goal being to classify groups based on their geometry. A prominent instance of this is Gromov's program to classify groups up to quasiisometry \cite{gromov:asymptotic}, which he proposed in the wake of important results such as: his polynomial growth theorem, which implies quasiisometric rigidity for free-abelian groups \cite{gromov:groups,bass:degree,guivarch:croissance}; Mostow rigidity and Tukia's theorem for hyperbolic manifolds \cite{mostow:onremarkable,marden:geometry,prasad:strong,tukia:onquasiconformal,tukia:homeomorphic}; and the Stallings--Dunwoody theorem for virtually free groups \cite{stallings:ontorsionfree,dunwoody:accessibility}. 

There is by now a very large body of work on quasiisometric classification and rigidity, of which we mention a few aspects.

For symmetric spaces of non-compact type, quasiisometry determines the space up to a rescaling of the irreducible factors \cite{kleinerleeb:rigidity}; see also \cite{tukia:homeomorphic,pansu:metriques,chow:groups} for the rank-one cases. The same holds for higher-rank Euclidean buildings \cite{kleinerleeb:rigidity}. Mapping class groups enjoy a similar rigidity phenomenon: apart from a few low-complexity exceptions, quasiisometry determines the underlying surface \cite{behrstockkleinerminskymosher:geometry}. Other examples include irreducible nonuniform lattices in semisimple Lie groups other than $\mathrm{Isom}(\mathbb H^2)$, some large-type Artin groups, and solvable Baumslag--Solitar groups $BS(1,n)$, for which quasiisometry within the family is equivalent to commensurability \cite{schwartz:quasiisometry,eskin:quasiisometric,huangosajda:quasieuclidean,huangosajdavaskou:rigidity,farbmosher:quasiisometric}.

A natural variant of the question of when two spaces are quasiisometric is that of when one can be quasiisometrically embedded in the other. This is in general a much more difficult question, in line with the difference between asking when two groups are isomorphic and asking whether one is a subgroup of another. At the most basic level, one loses access to arguments of the form ``prove something for a quasiisometry, and then use that same property for its quasiinverse''.

Nevertheless, in some settings with especially strong geometry one can still obtain analogous rigidity results. In higher-rank symmetric spaces, Fisher--Whyte \cite{fisherwhyte:quasiisometric} proved that, under natural compatibility assumptions, quasiisometric embeddings remain rigid, in the sense that maximal flats are mapped within finite Hausdorff distance of maximal flats. For mapping class groups, Bowditch \cite{bowditch:large:mapping} showed that, for compact orientable surfaces of equal complexity at least $4$, the existence of a quasiisometric embedding between the corresponding mapping class groups forces the surfaces to be homeomorphic, and moreover the embedding is at finite distance from an isometry. Related rigidity phenomena for quasiisometric embeddings were also established for nonuniform lattices in higher-rank semisimple Lie groups \cite{fishernguyen:quasiisometric} and for solvable Baumslag--Solitar groups \cite{nairne:embeddings}.

Our goal in the present paper is to obtain rigidity results for quasiisometric embeddings between right-angled Artin groups.

\subsection*{Geometry of right-angled Artin groups}

Quasiisometries between right-angled Artin groups are still not completely understood, though there have been important contributions from a number of authors. As an early sign that the picture is likely to be complicated, note that there is a marked difference between the two extreme cases of free groups and free-abelian groups: any two nonabelian free groups are quasiisometric, but free abelian groups of different ranks are not quasiisometric.

Moving away from the extreme cases, there seems to be a subtle line between these two behaviours. On the ``free'' end, Behrstock--Neumann showed that if $\Gamma$ and $\Lambda$ are finite trees of diameter at least three, then the associated right-angled Artin groups $A_\Gamma$ and $A_\Lambda$ are quasiisometric \cite{behrstockneumann:quasiisometric}. Together with Januszkiewicz, they later found higher-dimensional examples \cite{behrstockjanuszkiewiczneumann:quasiisometric}.

On the other side, Bestvina--Kleiner--Sageev introduced the family of \emph{atomic} graphs, and showed that if two right-angled Artin groups associated with atomic graphs are quasiisometric, then they are isomorphic \cite{bestvinakleinersageev:asymptotic}. This was later vastly generalised by Huang, who showed the same result for all right-angled Artin groups with finite outer automorphism group \cite{huang:quasiisometric:1} and in various cases where infinite outer automorphism groups are permitted \cite{huang:quasiisometric:2}. Even more remarkably, Huang, building on work with Kleiner \cite{huangkleiner:groups}, showed that for graphs $\Gamma$ that have no induced 4--cycles, are ``star-rigid'', and have finite $\Out(A_\Gamma)$, the group $A_\Gamma$ is absolutely quasiisometrically rigid, in the sense that \emph{any} group quasiisometric to $A_\Gamma$ is commensurable with it \cite{huang:commensurability}.

The most general results classifying when two right-angled Artin groups are quasiisometric are those of Margolis, which are stated in terms of \emph{JSJ trees of cylinders} \cite{margolis:quasiisometry}.

At this point, we note that all of the rigidity results mentioned above rely in an essential way on the existence of a quasiinverse. Indeed, each of them in some way involves using a ``double embedding'' argument, where one shows a property of a quasiisometry, notes that it holds for the quasiinverse, and then uses that the composition is coarsely the identity.

To the authors' knowledge, the only general result about quasiisometric embeddings between right-angled Artin groups is that of Rull \cite{rull:embedding}. Elaborating on the ``Alice's diary'' construction introduced in \cite{buyalodranishnikovschroeder:embedding}, she proved that if $\Gamma$ is an $n$--colourable finite graph, then $A_\Gamma$ can be quasiisometrically embedded in a product of $n$ finite-rank free groups.

In particular, there do not seem to be any nontrivial general results obstructing the existence of quasiisometric embeddings. As in \cite{fisherwhyte:quasiisometric,bowditch:large:mapping}, in order to have any control, we must restrict our attention to pairs of right-angled Artin groups of the same dimension (note that this is automatic when considering quasiisometries). If $\Gamma$ is a graph, then the dimension of the associated right-angled Artin group is equal to the clique number of $\Gamma$.


\subsection*{Main results}
To simplify the discussion, all right-angled Artin groups in these statements will be assumed to be finitely generated, although many of our results do not rely on that hypothesis.

\subsubsection*{Controlling flats}

One of the main results of \cite{bestvinakleinersageev:asymptotic} states that every \emph{quasiisometry} between 2--dimensional right-angled Artin groups maps each 2--flat within a finite Hausdorff distance of a 2--flat \cite[Thm~1.1]{bestvinakleinersageev:asymptotic}. The following result is a substantial generalisation of this in several directions at once; see \cref{thm:qi_embedding_every_2flats_2D_RAAGs}.

\begin{mthm} \label{mthm:flats}
Let $\Gamma$ be a triangle-free simplicial graph such that no connected component is a star. Let $Y$ be a finite-dimensional CAT(0) cube complex of asymptotic rank two. If $f\colon  A_\Gamma \to Y$ is a quasiisometric embedding, then for every $2$--flat $F \subset A_\Gamma$, the image $f(F)$ lies at finite Hausdorff distance from some $2$--flat in $Y$.
\end{mthm}

Firstly, \cref{mthm:flats} is about quasiisometric embeddings rather than quasiisometries. This increase in generality necessitates the no-star assumption, as observed in \cite[\S2]{bowditch:quasiisometric}; see \cref{rem:so_star_neccesary}.
Secondly, we only assume that $Y$ is a CAT(0) cube complex, rather than another right-angled Artin group. In particular, it need not admit any cobounded group action. 
Additionally, we do not require $Y$ to have the same dimension as $A_\Gamma$. \emph{Asymptotic rank} was introduced for metric spaces in \cite{wenger:asymptotic}, and for a CAT(0) cube complex it is equal to the supremal $n$ such that the cube complex contains arbitrarily large isometrically embedded boxes $[0,m]^n$ \cite[Prop.~2.18]{baderbensaidpetyt:from}. Comparing this notion with the cubical dimension, note that every hyperbolic space of finite \emph{asymptotic dimension} is quasiisometric to a finite-dimensional CAT(0) cube complex of asymptotic rank one \cite[Thm~4.8]{petyt:mapping}.

The proof of \cite[Thm~1.1]{bestvinakleinersageev:asymptotic} relies on the existence of a quasiinverse to show that the quasiisometry induces an isomorphism of the ``quarter-plane complex''. In the case of quasiisometric embeddings, this is no longer true. We build on results from \cite{baderbensaidpetyt:from} to gain control over the images of ``most'' \emph{singular} geodesics (those that lie in the 1--skeleton) of 2--flats, and use this to control the images of the 2--flats themselves.
This elaborates on \cref{thm:weakly_FB_cliques}, which gives a condition under which a quasiisometric embedding maps top-dimensional \emph{standard flats} (cosets of parabolic abelian subgroups) within finite Hausdorff distance of flats. The latter result is not restricted to the two-dimensional setting, and generalises results given in \cite{baderbensaidpetyt:from}.

\subsubsection*{Unconcealability results}

Our next family of results show that for certain graphs $\Gamma$, the existence of a quasiisometric embedding of $A_\Gamma$ in a right-angled Artin group $A_\Lambda$ of the same dimension cannot be hidden more deeply than in $\Lambda$. The strongest version of this is the following, which gives a necessary and sufficient condition (see \cref{cor:embedding_prod_F2_and_5cycles}). Let $S^0$ denote the graph with two vertices and no edges.

\begin{mthm} \label{mcor:unconcealability}
Suppose that $\Gamma$ is a square, a pentagon, or, more generally, a join of copies of the pentagon and $S^0$. If $A_\Gamma$ quasi-isometrically embeds in a right-angled Artin group $A_\Lambda$ of the same dimension, then $\Gamma$ is an induced subgraph of $\Lambda$.
\end{mthm}

In the case where $\Gamma$ is a square, Genevois \cite{genevois:polynomial} has recently announced a result extending \cref{mcor:unconcealability} to arbitrary $\Lambda$, using different methods. 

The starting point for proving \cref{mcor:unconcealability} is the results of \cite{baderbensaidpetyt:from}, which give geometric branching conditions under which a quasiisometric embedding between CAT(0) cube complexes of the same rank maps flats within finite Hausdorff distance of flats.
In \cref{sec:paper1_for_raags} we phrase conditions on $\Gamma$ that imply those geometric conditions for standard flats. The results then in particular imply that the \emph{extension graph} of $\Gamma$ (see \cref{sec:extension_graphs}) can be seen in the boundary of the codomain; see \cref{thm:fully_branching_raag_to_sing_boundary}, and we proceed from there.
These conditions are closed under products and are satisfied by both cycles and joins of two-point graphs (for which the corresponding right-angled Artin group is a product of nonabelian free groups).


\cref{mcor:unconcealability} can be compared with \cite{kambites:oncommuting} and \cite[Cor.~8.1]{kimkoberda:embedability}, which respectively show that when the right-angled Artin group on the square or on the pentagon is a subgroup of some $A_\Lambda$, $\Lambda$ contains an induced square or pentagon, respectively.

\cite[Cor.~8.1]{kimkoberda:embedability} shows moreover that if $\Gamma$ is triangle-free and $A_{C_n}$ is a subgroup of $A_\Gamma$ for some $n\ge5$, then $C_m$ is an induced subgraph of $\Gamma$ for some $m\in\{5,\dots,n\}$.
We prove a more general geometric analogue of these results; see \cref{thm:embedding_products_in_RAAGs_general}. Let $K_{n\times2}$ denote the complete $n$--partite graph whose parts have two vertices: it is the join of $n$ copies of $S^0$ and the associated right-angled Artin group is $A_{K_{n\times2}}=(F_2)^n$.

\begin{mthm} \label{mthm:unconcealable}
Let $p,q \geq 0$, $n_1,\dots,n_q\ge4$, and let $\Gamma = K_{p\times 2} * C_{n_1} *\cdots *C_{n_q}$. If $A_\Gamma$ quasiisometrically embeds into a $(p+2q)$--dimensional right-angled Artin group $A_\Lambda$, then there exist $m_i \in \{4,\dots,n_i\}$ such that $K_{p\times 2} * C_{m_1} *\cdots *C_{m_q}$ is an induced subgraph of $\Lambda$. Moreover, if $n_i$ is odd, then we can take $m_i$ to be odd.
\end{mthm}

See \cref{rem:need_5} for a discussion of the optimality of this statement. 
\cref{mthm:unconcealable} has the following corollary.
We use the convention that the girth of a forest is $\infty$.

\begin{mcor} \label{mcor:girth}
If $\Gamma$ has strictly smaller girth than $\Lambda$, then $A_\Gamma$ cannot be quasiisometrically embedded in $A_\Lambda$.
\end{mcor}

\begin{proof}
If $\Gamma$ has girth three, then $A_\Gamma$ has dimension greater than two, so cannot be quasiisometrically embedded in $A_\Lambda$. Otherwise, both $A_\Gamma$ and $A_\Lambda$ are 2--dimensional. Let $C_n\subset\Gamma$ be a shortest cycle. We have that $A_{C_n}< A_\Gamma$ is a parabolic subgroup. If there were a quasiisometric embedding $A_\Gamma\to A_\Lambda$, then we would get a quasiisometric embedding $A_{C_n}\to A_\Lambda$. But then \cref{mthm:unconcealable} would imply that the girth of $\Lambda$ is at most $n$, a contradiction.
\end{proof}

Thanks to Droms' classification of coherent right-angled Artin groups \cite{droms:graph}, we also obtain the following corollary, which appears as \cref{cor:coherence}.

\begin{mcor}
Let $A_\Gamma$ be a 2--dimensional right-angled Artin group. If some incoherent right-angled Artin group quasiisometrically embeds in $A_\Gamma$, then $A_\Gamma$ is incoherent.
\end{mcor}

\subsubsection*{No universal receiver}

The class of right-angled Artin groups satisfying the branching conditions of \cref{thm:fully_branching_raag_to_sing_boundary} is rich enough to show that there is no universal receiver for quasiisometric embeddings between $n$--dimensional right-angled Artin groups; see \cref{sec:universal}.

\begin{mthm} \label{mthm:universal}
For each $n>1$, there is no $n$--dimensional right-angled Artin group into which every $n$--dimensional right-angled Artin group quasiisometrically embeds.
\end{mthm}

This provides a geometric analogue of \cite[Thm~1.16]{kimkoberda:embedability}, which states that there is no 2--dimensional right-angled Artin group that contains all 2--dimensional right-angled Artin groups as subgroups. It is also a generalisation of that result, because whenever one 2--dimensional right-angled Artin group is a subgroup of another, it can be found as an undistorted subgroup \cite[Cor.~1.15]{kimkoberda:embedability}. 

In fact, we prove a slightly more general statement, applying in particular to compact special cube complexes: if $X$ is a compact nonpositively curved cube complex with no self-intersecting hyperplanes, then $\widetilde X$ is not a universal receiver for quasiisometric embeddings of finitely generated right-angled Artin groups of dimension $\dim X$; see \cref{thm:no_univ_cat0_cube_complex}.

The proof of \cref{mthm:universal} uses a chromatic number obstruction --- we show that if the chromatic number of $\Gamma$ is greater than that of $\Lambda$ then $A_\Gamma$ cannot be quasiisometrically embedded into $A_{\Lambda}$. The algebraic statement \cite[Thm~1.16]{kimkoberda:embedability} was proved in a similar fashion. The idea of using the chromatic number of the defining graph to obstruct group embeddings has also been used in \cite{kimkoberda:obstruction}, to obstruct embeddings of right-angled Artin groups into mapping class groups. In \cite[Sec.~5]{button:groups}, Button showed that for a graph $\Gamma$ with chromatic number $c$,  the associated right-angled Artin group acts properly on a product of $c$ trees.

\subsubsection*{The QI group and QIE monoid}

It is natural to try to understand what forms quasiisometric embedding between right-angled Artin groups can take. In general this is a very difficult question, and \cite{baderbensaidpetyt:quasiisometric:flexibility} shows that exotic quasiisometric embeddings exist even between right-angled Artin groups defined on cycles. However, in the case of \emph{self}--quasiisometric-embeddings one has access to more restrictions, and we are able to give a complete description of all quasiisometric embeddings $A_{C_n}\to A_{C_n}$ for each $n>4$.


\begin{mthm} \label{mthm:self_qie}
For each $n \geq 5$, every quasiisometric embedding $A_{C_n}\to A_{C_n}$ is, up to bounded distance, a composition of a left translation, an automorphism of $C_n$, and a map that preserves the syllable structure of reduced words and changes only the exponents. Moreover, such a map is a quasiisometry if and only if the associated exponent-change maps are onto.
\end{mthm}

We refer to \cref{thm:qie_of_cycle_to_itself} for a more precise statement, which includes additional information about the exponent-change maps and also the converse direction of the characterisation. 

\cref{mthm:self_qie} can be compared with results on mapping class groups. There, Bowditch proved that every self--quasiisometric-embedding of a mapping class group of complexity at least four is at finite distance from a left translation \cite{bowditch:large:mapping}. It is also reminiscent of the Laurence--Servatius theorem describing generators of $\Out(A_\Gamma)$ \cite{laurence:generating,servatius:automorphisms}. 

Cycle graphs $C_n$ for $n > 4$ form the simplest examples of star-rigid, square-free graphs whose right-angled Artin groups have a finite outer automorphism group.
As described previously, Huang proved that 
such right-angled Artin groups are absolutely quasiisometrically rigid \cite{huang:commensurability}. Unlike the mapping class group case, where rigidity is obtained by showing that $\mathrm{QI}(\operatorname{MCG}(S))=\operatorname{MCG}(S)$ \cite{behrstockkleinerminskymosher:geometry}, Huang has to contend with much larger quasiisometry groups.
Indeed, \cite[Thm~1.10]{bestvinakleinersageev:asymptotic} states that if $\Gamma$ is atomic then $\mathrm{Aut}(A_\Gamma)\hookrightarrow\mathrm{Comm}(A_\Gamma)\hookrightarrow\mathrm{QI}(A_\Gamma)$ and the image of each embedding has infinite index in the respective codomain. This makes Huang's result more surprising.

The condition that $\Gamma$ is atomic in \cite[Thm~1.10]{bestvinakleinersageev:asymptotic} is quite restrictive: it in particular requires $A_\Gamma$ to be two-dimensional.
We extend the result to general non-abelian right-angled Artin groups. More strongly, we use the machinery developed in the proof of \cref{mthm:self_qie} to show that $\mathrm{QI}(A_\Gamma)$ is uncountable. The following summarises the results of \cref{sec:qi_big}.

\begin{mthm} \label{mthm:QI_uncountable}
    Let $\Gamma$ be a graph. 
    The natural homomorphisms 
    \[\mathrm{Aut}(A_\Gamma)\to \mathrm{Comm}(A_\Gamma)\quad \text{and}\quad\mathrm{Comm}(A_\Gamma)\to \mathrm{QI}(A_\Gamma)\] 
    are injective.    
    If $\Gamma$ is not complete then their images have infinite index, and $\mathrm{QI}(A_\Gamma)$ is uncountable.
    In particular, if $\Gamma$ is finite then the image of $\mathrm{Comm}(A_\Gamma)$ in $\mathrm{QI}(A_\Gamma)$ has uncountable index.
\end{mthm}

The tension between rigidity and the large quasiisometry group is demystified in \cref{sec:self_qie}, where we show that though the quasiisometry group is big, for a star-rigid graph $\Gamma$ such that $\mathrm{Out}(A_\Gamma)$ is finite, self-quasiisometries have a very rigid structure, similar to the one exhibited by quasiisometric embeddings in \cref{mthm:self_qie}. 
Note that we allow $\Gamma$ to have induced 4--cycles.

\begin{mthm} \label{mthm:self_qi}
Let $\Gamma$ be a non-trivial star-rigid graph such that $\Out(A_\Gamma)$ is finite. Every quasiisometry $A_\Gamma\to A_\Gamma$ is, up to bounded distance, a composition of a left translation, a graph automorphism, and a map which preserves the syllable structure of reduced words and changes only the exponents.
\end{mthm}

See \cref{thm:qi_of_finite_out_to_itself} for a more precise statement that in particular has enough information about the exponent-change maps to state the converse direction.

The basic strategy in the proof of \cref{mthm:self_qi} is as follows. Using \cite{huang:quasiisometric:1}, we induce an automorphism of the extension graph $\Gamma^\ext$ of $\Gamma$ (see \cref{sec:extension_graphs}). Using this, we replace the original quasiisometric embedding by a better-behaved one at finite distance from it. We then use the combinatorics of $\Gamma^\ext$ to understand the structure of this perturbed embedding.

The strategy for \cref{mthm:self_qie} follows the same lines, but requires an additional input in order to convert the self--quasiisometric-embedding of the group into an endomorphism of the extension graph. This input is provided by \cref{mthm:ext_map} below.

\subsubsection*{Stable quasiisometric embeddings}

As mentioned above, \cref{thm:fully_branching_raag_to_sing_boundary} is an interpretation of results in \cite{baderbensaidpetyt:from} for right-angled Artin groups. It is a special case of a more general result showing that quasiisometric embeddings between same-dimensional CAT(0) cube complexes map singular geodesics that satisfy certain branching conditions within finite Hausdorff distance of singular geodesics in the codomain.

In the setting of right-angled Artin groups, we have access to the stronger notion of standard geodesics. 
Under an additional hypothesis on the codomain right-angled Artin group, we show that something stronger than \cref{thm:fully_branching_raag_to_sing_boundary} is true: quasiisometric embeddings must map standard geodesics at finite Hausdorff distance from \emph{standard} geodesics; see \cref{thm:squarefree_stable}. We refer to quasiisometric embeddings with this property as \emph{stable}, inspired by terminology from \cite{huang:quasiisometric:1}. In \cref{prop:stable_implies_induced}, we show that stable quasiisometric embeddings induce maps of extension graphs (\cref{def:extension_graph}), which together with \cref{thm:squarefree_stable} implies the following; see \cref{cor:squarefree_stable}. Let $S_n$ denote the graph that is the suspension of the disjoint union of a vertex and an $(n-1)$--clique.

\begin{mthm} \label{mthm:ext_map}
Let $\Gamma$ and $\Lambda$ be graphs with clique number $n$. Suppose that $\Gamma$ satisfies the conditions of \cref{thm:fully_branching_raag_to_sing_boundary} and that $S_n$ is not a subgraph of $\Lambda$. Every quasiisometric embedding $A_\Gamma\to A_\Lambda$ is stable and so induces a graph embedding $\Gamma^\ext\to\Lambda^\ext$.
\end{mthm}

Observe that the condition on $\Lambda$ is automatically satisfied if it has no 4--cycle subgraphs, and this is precisely what the condition says when $n=2$. We remark that, in contrast to \cite[Thm~1.3]{kimkoberda:embedability}, the existence of a graph embedding $\Gamma^\ext\to\Lambda^\ext$ does not imply that $A_\Gamma$ is a subgroup of $A_\Lambda$. This even fails when $\Gamma$ and $\Lambda$ are cycle graphs; see \cite{baderbensaidpetyt:quasiisometric:flexibility}, which contains a complete classification of when $A_{C_m}$ can be quasiisometrically embedded in $A_{C_n}$.

When the target is a $2$--dimensional CAT(0) cube complex $Y$, we can still obtain an embedding of $\Gamma^\ext$ into the \emph{singular boundary graph} of $Y$, which is a subset of the Tits boundary that is sensitive to the cubical structure of $Y$; see \Cref{def:singular_boundary_graph}. The following (\cref{cor:2dim_boundary_map}) is a special case of \cref{cor:peripheral_boundary_map}, which is stated not just in dimension two.

\begin{mthm} \label{mthm:2dim_boundary}
Let $\Gamma$ be a triangle-free graph such that no connected component is a star, and let $Y$ be a CAT(0) square complex. If there is a quasiisometric embedding $A_\Gamma\to Y$, then there is a combinatorial embedding $\Gamma^\ext\to\partial_{\mathrm{sing}}Y$.
\end{mthm}

In the particular case where $Y=X_\Lambda$ comes from a right-angled Artin group $A_\Lambda$, it should be noted that finding a map to $\partial_{\mathrm{sing}}Y$ is weaker than finding a map to $\Lambda^\ext$ as in \cref{mthm:ext_map}, because $\Lambda^\ext$ also carries group-theoretic information.

\medskip

\subsection*{Outline of the article}

\begin{itemize}
\item   \cref{sec:prelims} covers background results on CAT(0) cube complexes, right-angled Artin groups, and extension graphs.
\item   In \cref{sec:paper1_for_raags}, we interpret some of the main results of \cite{baderbensaidpetyt:from} in the setting of right-angled Artin groups. These results underpin everything that follows.
\item   The goal of \cref{sec:universal} is to prove \cref{mthm:universal}, showing that there is no universal receiver for quasiisometric embeddings of $n$--dimensional right-angled Artin groups.
\item   Our results on unconcealability are proved in \cref{sec:unconcealable}.
\item   \cref{sec:qie_stable} is where we introduce stable quasiisometric embeddings and prove that they induce embeddings of extension graphs. We also prove \cref{thm:squarefree_stable}, which gives conditions under which all quasiisometric embeddings are stable.
\item   \cref{sec:self_qie} is dedicated to the description of quasiisometric embeddings from a right-angled Artin group to itself.
\item   Finally, \cref{sec:beyond_FB} introduces the weakened branching condition on cliques and studies the quasiisometric images of corresponding flats.
\end{itemize}

\subsection*{Acknowledgements}
OB acknowledges support from the FWO and F.R.S.-FNRS under the Excellence of Science (EOS) programme (project ID 40007542). SB and HP thank the Isaac Newton Institute for their hospitality during the programme \emph{Operators, Graphs, Groups}, where some of the work on this paper took place (EPSRC grant EP/Z000580/1).

We are very grateful to Sam Kim for several very useful suggestions, in particular in relation to \cref{thm:no_univ_cat0_cube_complex}. We also thank Jack Button for useful comments on an earlier version.

Claude was used to check for typos in the first version of this paper.

\section{Preliminaries} \label{sec:prelims}

Let $X$ be a metric space and $A\subset X$. For $r\ge0$, the \emph{$r$--neighbourhood} of $A$ is $A^{+r}=\{x\in X\,:\,\dist(x,A)\le r\}$. The \emph{Hausdorff distance} between subsets $A,B\subset X$ is 
$$
d_{\mathrm{Haus}}(A,B)=\inf\{r\geq 0 \,:\, A\subseteq B^{+r}\text{ and }B\subseteq A^{+r}\}.
$$

The \emph{clique number} $\omega_\Gamma$ of a graph $\Gamma$ is the supremal cardinality of cliques of $\Gamma$. If $\omega_\Gamma\le2$, then we refer to $\Gamma$ as being \emph{triangle-free}. All graphs considered in this paper will be simplicial. We write $K_{n\times2}$ for the complete $n$--partite graph whose parts all have cardinality two. 

\subsection{CAT(0) spaces}\label{sec:CAT(0)_spaces}

We shall use only a limited amount of CAT(0) geometry in this paper. We refer the reader to \cite{ballmann:lectures} and \cite{bridsonhaefliger:metric} for background on the subject. The main class of CAT(0) spaces we are concerned with in this paper are CAT(0) cube complexes (see \cref{subsec:ccc}) and more specifically, Salvetti complexes. 

\begin{definition}[Flats]\label{def:flats_in_CAT(0)}
Let $(X,d)$ be a CAT(0) space. A \emph{$k$--flat} in $X$ is the image of an isometric embedding $(\mathbb R^k,\|\cdot\|_2)\to (X,d)$.
\end{definition}

\begin{definition}[Parallels] \label{def:parallel_set}
    Let $X$ be a complete CAT(0) space, and let $F,F' \subset X$ be flats. We say that $F$ and $F'$ are \emph{parallel} if they are at finite Hausdorff distance. Equivalently, the distance from $F$ to $F'$ and the distance from $F'$ to $F$ are constant. 

    The \emph{parallel set} of a flat $F$ is the union of all flats parallel to $F$, and is denoted by $P(F)$. It is a closed convex subset of $X$, and it admits a canonical splitting as a metric product $P(F)=F\times Y$ for some complete CAT(0) space $Y$, see \cite[Section~2.3.3]{kleinerleeb:rigidity}.
\end{definition}




The following simple lemma shows that a geodesic lying near a closed, convex subset has a parallel inside that subset. We refer to \cite[Lem.~4.2]{baderbensaidpetyt:from} for a proof.

\begin{lemma}\label{lem:parallel_geod_in_convex}
Let $X$ be a complete CAT(0) space, let $A \subseteq X$ be a closed, convex subset, and let $\gamma \subseteq X$ be a geodesic. If $\gamma \subseteq A^{+D}$, for some $D \geq 0$, then there exists a parallel geodesic $\gamma' \subseteq A$ such that $d_{\mathrm{Haus}}(\gamma,\gamma') \leq D$.
\end{lemma}

\subsection{CAT(0) cube complexes} \label{subsec:ccc}

We refer the reader to \cite{wise:structure,bowditch:median:book,genevois:algebraic} for basic facts about CAT(0) cube complexes. 

A \emph{cube complex} is a cell complex obtained by gluing Euclidean unit cubes $[0,1]^n$ along isometric faces. A CAT(0) cube complex is a simply connected cube complex such that the link of each vertex is a flag simplicial complex. The \emph{dimension} of a CAT(0) cube complex is the supremal dimension of the cubes involved in its construction. 

A key combinatorial feature of CAT(0) cube complexes is their \emph{hyperplanes}. For each edge $e$ of $X$, the hyperplane \emph{dual} to $e$ can be defined as an equivalence class of edges: set $e_1\sim e_2$ whenever they are opposite edges in a square and extend transitively. Each hyperplane $h$ separates $X$ into two \emph{halfspaces}. We say that $h$ \emph{separates} two points of $X$ if they lie in different halfspaces. Two hyperplanes \emph{cross} if all four possible intersections of halfspaces are nonempty.

If $X$ is a CAT(0) cube complex, then the length metric $d_2$ obtained by equipping each cube with the $\ell^2$ metric makes $X$ into a CAT(0) space. If one instead equips the cubes with the $\ell^1$ metric, then one obtains a metric $d_1$ on $X$ that makes $X$ into a \emph{median metric space}. For vertices $x,y\in X$, the distance $d_1(x,y)$ is equal to the number of hyperplanes that separate $x$ from $y$. If $X$ is finite-dimensional, then $d_1$ and $d_2$ are bilipschitz equivalent. In this paper, we shall almost always consider CAT(0) cube complexes as being equipped with the metric $d_1$. The exception is when we are discussing flats.

\begin{definition}[Singular] \label{def:ccc_singular}
Let $X$ be a CAT(0) cube complex. A $k$--flat $F\subset X$ is \emph{singular} if it is contained in the $k$--skeleton of $X$. A \emph{singular geodesic} is a singular $1$--flat. A \emph{singular geodesic ray} is a subset of the 1--skeleton that is the image of an isometric embedding $[0,\infty)\to(X,d_2)$. 
\end{definition}


A \emph{2--orthant} is the image of an isometric embedding of $[0,\infty)\times[0,\infty)$. Following \cite[Def.~11.1]{baderbensaidpetyt:from}, we introduce the following definition. 

\begin{definition}[Singular boundary graph] \label{def:singular_boundary_graph}
Let $X$ be a CAT(0) cube complex. The \emph{singular boundary graph} of $X$, denoted $\partial_{\mathrm{sing}}X$, is the following graph. The vertices of $\partial_{\mathrm{sing}}X$ are the equivalence classes of singular geodesic rays in $X$, where geodesic rays are deemed equivalent if they lie at finite Hausdorff distance from one another. Two vertices of $\partial_{\mathrm{sing}}X$ are joined by an edge whenever they admit singular geodesic representatives that span a 2--orthant of $X$.
\end{definition}

\begin{example}\label{ex:stairwase_and_sing_bound_product_trees}
If $X$ is the ``staircase'' square complex, obtained from a 2--orthant by deleting all squares containing a point $(x,y)$ with $y>x$, then $\partial_{\mathrm{sing}}X$ is a single point. 

Let $\alpha_2>\alpha_1>0$. If $X$ is the subcomplex of $[0,\infty)^2$ consisting of all squares whose points $(x,y)$ all satisfy $\alpha_1 x\le y\le\alpha_2x$, then $\partial_{\mathrm{sing}}X=\varnothing$.
\end{example}

See \cite[\S11]{baderbensaidpetyt:from} for more discussion of the singular boundary graph and a comparison to other notions of boundary for CAT(0) cube complexes.


\subsection{Right-angled Artin groups} \label{sec:raags}

A \emph{right-angled Artin group} (RAAG) is a group $A$ with a generating set $S$ such that the only relations are that some elements of $S$ commute. This can be encoded by the (simplicial) graph $\Gamma$ that has vertex set $S$ and an edge between vertices whose corresponding generators commute. We write $A=A_\Gamma$, and refer to $S$ as the \emph{standard} generating set of $A_\Gamma$. 

\begin{remark}
In the body of the paper, we do not require right-angled Artin groups to be finitely generated unless specified otherwise.
\end{remark}

\bsh{Notation} \label{sh:vertex_notation}
If we are given a vertex $v\in\Gamma$, then we shall sometimes write $s_v\in S$ for the corresponding standard generator. Similarly, given an element $s\in S$, we shall sometimes write $v_s$ for the corresponding vertex of $\Gamma$.
\esh

\bsh{The space $X_\Gamma$} \label{sh:the_space_X_Gamma}
The Cayley graph of $A_\Gamma$ with respect to the generating set $S$ is a median graph, which is the 1--skeleton of a CAT(0) cube complex $X_\Gamma$. The quotient of $X_\Gamma$ by the action of $A_\Gamma$ is known as the \emph{Salvetti complex} of $A_\Gamma$. 

If $X_\Gamma$ is finite-dimensional, then it is quasiisometric to its 1--skeleton.
The dimension of $X_\Gamma$ is equal to the clique number of $\Gamma$, so as long as $\Gamma$ has finite clique number, which it always will in this paper, $X_\Gamma$ is quasiisometric to the standard Cayley graph of $A_\Gamma$. Hence for most purposes we will not distinguish the two, and just write $A_\Gamma$. For instance, the \emph{dimension} of $A_\Gamma$ is the dimension of $X_\Gamma$. But when thinking of the space as a CAT(0) object we will consider $X_\Gamma$. 
\esh

When making metric statements about an infinitely generated right-angled Artin group, we are referring to the metric on the standard Cayley graph.


Given $g\in A_\Gamma$, write $|g|=|g|_\Gamma$ for the word length with respect to $S$. A \emph{minimal representative} of $g\in G$ is a word $\gamma$ over $S$ that represents $g$ and has length $|\gamma|=|g|$. 

\begin{proposition}[\cite{hermillermeier:algorithms}] \label{prop:shuffling}
The minimal representatives of $g\in A_\Gamma$ can be obtained from an arbitrary representative by repeatedly shuffling commuting pairs of elements and cancelling inverse pairs.
\end{proposition}

A \emph{prefix} of $g$ is an element $h\in A_\Gamma$ such that $|g|=|h|+|h^{-1}g|$. In other words, $h$ lies on an $\ell^1$-geodesic from 1 to $g$. Note that if $\gamma$ is a word representing $g$, then not all prefixes of $g$ are represented by initial substrings of $\gamma$. We will never talk about ``prefixes of words'': we shall always refer to such things as initial substrings. Suffixes are defined similarly.

\begin{definition}[Syllable-length]
The \emph{syllable-length} of an element $g\in A_\Gamma$, denoted $\syl g$, is the minimal $k$ such that $g$ can be written as $g=s_1^{n_1}\dots s_k^{n_k}$, where $s_i\in S$. A word in $S$ is said to be \emph{syllable-reduced} if it has the form $s_1^{n_1}\dots s_k^{n_k}$ and $k$ is the syllable-length of the element of $A_\Gamma$ that it represents.
\end{definition}

\begin{definition}[Standard]\label{def:standard}
Each edge of $X_\Gamma$ is labelled by some standard generator in $S$. A geodesic in the 1--skeleton of $X_\Gamma$ is \emph{standard} if all its edges have the same label. We denote the standard geodesic labelled by $s$ and going through $g$ by $g\sgen s$. Note that $g\sgen s=gs\sgen s$. Indeed, its vertex set is just the $g$--coset of $\sgen s$.

For $k>1$, a $k$--flat in $X_\Gamma$ is \emph{standard} if it is singular and contains $k$ parallelism classes of standard geodesics. Standard flats correspond to cliques in $\Gamma$.
\end{definition}


\begin{definition} \label{def:duplication_graph}
For a graph $\Gamma$, we define the \emph{duplication} of $\Gamma$ to be the graph $d\Gamma$ with vertex set $\Gamma\times\{0,1\}$ and an edge from $(v,\eps_1)$ to $(w,\eps_2)$ whenever $\{v,w\}\in E(\Gamma)$. In other words, it is the lexicographic product of $\Gamma$ with an edge.
\end{definition}

For a RAAG $A_\Gamma$, the link of each vertex of $X_\Gamma$ is the flag complex on $d\Gamma$. 

The following appears as \cite[Ex.~1.4]{bestvinakleinersageev:asymptotic}.

\bsh{Doubling} \label{sh:doubling}
Let $\Gamma$ be a graph and let $v\in\Gamma$. The \emph{star} of $v$, denoted $\star(v)$ is the induced subgraph of $\Gamma$ whose vertices are $v$ and all its neighbours. The \emph{double} of $\Gamma$ along $v$ is the graph $\Gamma'$ obtained from two copies of $\Gamma$ by gluing them along $\star(v)$. It turns out that $A_{\Gamma'}$ is isomorphic to the index-two subgroup $\sgen{\{s_v^2\}\cup\{s_w,s_vs_ws_v^{-1}\mid v\neq w\in \Gamma\}}=\ker(\alpha_v)$ of $A_\Gamma$, where $\alpha_v:A_\Gamma\to\Z_2$ is defined by $\alpha(s_v)=1$ and $\alpha_v(s_w)=0$ for all vertices $w\ne v$.

One can check that the inclusion map of $\ker(\alpha_v)$ into $A_\Gamma$ is a 2--quasiisometric embedding whose image is 2--coarsely dense.
\esh

We will need the following lemma concerning hyperplanes crossed by asymptotic singular geodesic rays and the labels of parallel singular geodesics; it is an adaptation of \cite[Cor.~2.6]{bestvinakleinersageev:asymptotic}.


\begin{lemma}\label{lem:parallel_sing_geod_same_labels}
Let $X$ be a finite-dimensional CAT(0) cube complex, and let $\Gamma$ be a simplicial graph with finite clique number.
\begin{enumerate}
    \item If $\alpha,\alpha' \subset X$ are asymptotic singular geodesic rays, then $\alpha$ and $\alpha'$ contain singular subrays $\beta \subset \alpha$ and $\beta' \subset \alpha'$ crossing the same hyperplanes.
    
    \item If $\alpha,\alpha' \subset X_\Gamma$ are asymptotic singular geodesic rays, then $\alpha$ and $\alpha'$ contain singular subrays $\beta \subset \alpha$ and $\beta' \subset \alpha'$ with the same set of labels.
    
    \item If $\alpha,\alpha' \subset X_\Gamma$ are parallel singular geodesics, then they have the same set of labels.
\end{enumerate}
\end{lemma}

\begin{proof}
    $(1)$ By \cite[Lem.~2.10]{huang:top}, the distance between $\alpha$ and $\alpha'$ is attained on singular subrays $\beta \subset \alpha$ and $\beta' \subset \alpha'$. In particular, $\beta$ and $\beta'$ are parallel, in the sense that the distance from one to the other is constant. Since $X$ is a CAT(0) cube complex, and $\beta$ is a singular geodesic ray that is a subcomplex, its parallel set $P(\beta)$ is well defined; see \cite[Rem.~3.5]{huang:quasiisometric:1}. It splits as $P(\beta)=\beta \times C$, for some CAT(0) cube complex $C$. In this splitting, $\beta$ corresponds to $\beta \times \{v\}$ for some vertex $v \in C$, and $\beta'$ corresponds to $\beta \times \{v'\}$ for some vertex $v' \in C$. Take a combinatorial path $v=v_0,\dots,v_k=v'$ in $C$. For each $i$, set $\beta_i:=\beta \times \{v_i\}$. The singular rays $\beta_i$ and $\beta_{i+1}$ bound the flat strip subcomplex $\beta \times [v_i,v_{i+1}]$. Hence every edge of $\beta_i$ is opposite, in a square of this strip, to an edge of $\beta_{i+1}$. Since opposite edges of a square are dual to the same hyperplane, $\beta_i$ and $\beta_{i+1}$ cross the same hyperplanes. By induction, $\beta=\beta_0$ and $\beta'=\beta_k$ cross the same hyperplanes.

    $(2)$ By $(1)$, the rays $\alpha$ and $\alpha'$ contain singular subrays $\beta \subset \alpha$ and $\beta' \subset \alpha'$ crossing the same hyperplanes. In $X_\Gamma$, every hyperplane is dual to edges with the same label. It follows that $\beta$ and $\beta'$ have the same set of labels.

    $(3)$ The proof is the same as in $(1)$, working directly with the parallel set $P(\alpha)$ instead of passing to subrays.
\end{proof}

\subsection{Extension graphs}\label{sec:extension_graphs}

A useful combinatorial object for studying a right-angled Artin group $A_\Gamma$ is the associated \emph{extension graph}, which was introduced by Kim--Koberda \cite{kimkoberda:embedability}. From a coarse perspective, it is a quasitree that can be thought of as obtained by collapsing all standard 2--flats \cite{kimkoberda:geometry}. In this article we are more interested in the fine structure of extension graphs.

The following is equivalent to \cite[Def.~1.2]{kimkoberda:embedability}.

\begin{definition}[Extension graph] \label{def:extension_graph}
For each $g\in A_\Gamma$, let $g\cdot\Gamma$ be an isomorphic copy of~$\Gamma$, with vertices $\{g\cdot v\mid v\in V(\Gamma)\}$. The \emph{extension graph} $\Gamma^\ext$ of $\Gamma$ is obtained from the disjoint union $\bigsqcup_{g\in A_\Gamma}g\cdot\Gamma$ as follows:
\begin{itemize}
\item   first identify vertices $g\cdot v\in g\cdot\Gamma$ and $h\cdot w\in h\cdot\Gamma$ whenever $gs_vg^{-1}=hs_wh^{-1}\in A_\Gamma$;
\item   then identify each pair of edges with the same endpoints.
\end{itemize}
\end{definition}

For example, for $s$ a standard generator of $A_\Gamma$, the subgraphs $g\cdot\Gamma$ and $gs^n\cdot\Gamma$ of $\Gamma^\ext$ are glued along $\str(g\cdot v)$ for all $n\ne0$, because $gs_vg^{-1}=(gs)s_v{(gs)^{-1}}$ exactly when $s_v$ commutes with $s$. If we identify the vertex $g\cdot v$ of $\Gamma^\ext$ with the standard geodesic $g\sgen {s_v}$ of $A_\Gamma$, then we see that $g\cdot v=h\cdot v$ exactly when $g\sgen {s_v}$ and $h\sgen {s_v}$ are parallel.

Observe that in \cref{def:extension_graph}, vertices $g\cdot v$ and $h\cdot w$ can only be identified if $v=w$. The group $A_\Gamma$ has a natural left action on $\Gamma^\ext$, given by $g(h\cdot v)=gh\cdot v$. 

\begin{remark} \label{rem:extension_graph_definition}
The extension graph $\Gamma^\ext$ can equivalently be defined as follows. It has a vertex for each parallelism class of standard geodesics in $X_\Gamma$. Two parallelism classes are joined by an edge if they admit representatives that span a 2--flat of $X_\Gamma$.
\end{remark}

\begin{definition}[Types and labels]
Let $S=\{s_i\}$ be the standard generating set of $A_\Gamma$.
We say that a vertex $v\in\Gamma^\ext$ is of \emph{type} $i$ or $s_i$ if there exists $g\in A_\Gamma$ such that $v=g\cdot v_i$, where $s_i$ is the standard generator of $A_\Gamma$ corresponding to the vertex $v_i$ of $\Gamma$. We write $\type(v)=s_i$. A \emph{label} of $v$ is any such group element $g$. The \emph{minimal label} is the element $g\in A_\Gamma$ such that $|g|$ is minimal among labels of $v$. 
\end{definition}

\begin{remark}\label{rem:extension_graph_embeds_singular_boundary} 
The extension graph $\Gamma^\ext$, and more generally its duplication $d(\Gamma^\ext)$, embed naturally as induced subgraphs of $\partial_{\mathrm{sing}}X_\Gamma$. Indeed, by \Cref{rem:extension_graph_definition}, each standard geodesic determines two vertices of $\partial_{\mathrm{sing}}X_\Gamma$, corresponding respectively to the rays defined by positive and negative powers. Restricting to rays defined by positive powers yields an embedding of $\Gamma^\ext$, while restricting to rays defined by negative powers yields another one. Together, these two embeddings yield an embedding of $d(\Gamma^\ext)$.
\end{remark}

\section{Branching theorems for RAAGs} \label{sec:paper1_for_raags}

The main results of \cite{baderbensaidpetyt:from} concern quasiisometric embeddings between CAT(0) cube complexes. They give geometric branching conditions on flats under which the images of those flats are Hausdorff-close to flats in the codomain, and deduce the existence of induced maps between appropriately defined boundaries. In this section we interpret some of those definitions and results for right-angled Artin groups.

\cref{lem:finite_index} will compare \cref{def:fully_branching_in_graphs} with \cite[Defs.~9.1, 9.3, 9.5]{baderbensaidpetyt:from} in the case of standard geodesics in universal covers of Salvetti complexes; see also \cite[Rem.~9.4]{baderbensaidpetyt:from}. \cref{def:fully_branching_in_graphs} is slightly more general than just a direct translation, but we will show that the results from \cite{baderbensaidpetyt:from} still hold under its conditions; see \cref{thm:fully_branching_raag_rephrased,thm:fully_branching_raag_to_sing_boundary}. 

\begin{definition}\label{def:fully_branching_in_graphs}
Let $\Gamma$ be a (simplicial) graph with clique number $n \geq 2$.
\begin{enumerate}
\item A clique is said to be \emph{branching} if it is contained in an $n$--clique and is equal to the intersection of some maximal cliques.
\item A vertex $v$ is called \emph{branch-complemented} if it is contained in an $n$--clique $K$ such that both $\{v\}$ and $K\ssm\{v\}$ are branching. We call the clique $K\ssm\{v\}$ a \emph{complementary clique} for $v$.
\item A subgraph of $\Gamma$ is \emph{directionally branch-complemented} if each of its vertices is branch-complemented.
\end{enumerate}
\end{definition}

If $\Gamma$ is triangle-free, then a vertex is branching precisely when it is not a leaf, and branch-complemented if and only if it is branching and has a branching neighbour. Thus $\Gamma$ is directionally branch-complemented if and only if it has no leaves. See \Cref{fig:fully_branching_vertices_cliques_3D} for examples of directionally branch-complemented subgraphs when the clique number is $3$.

\begin{figure}[htbp]
    \centering
\begin{tikzpicture}[baseline=-.1cm, scale=1]
\def\h{0.8660254}
\draw[thick]
(-1,0) -- (-0.5,\h) -- (0.5,\h) -- (1,0) -- (0.5,-\h) -- (-0.5,-\h) -- (-1,0)
(1,0) -- (2,0)
(-1,0) -- (-2,0) -- (-3,0);
\fill[red] (-1,0) circle(.09);
\fill[red] (-0.5,\h) circle(.09);
\fill[red] (0.5,\h) circle(.09);
\fill[red] (1,0) circle(.09);
\fill[red] (0.5,-\h) circle(.09);
\fill[red] (-0.5,-\h) circle(.09);
\fill (2,0) circle(.09);
\fill[red] (-2,0) circle(.09);
\fill (-3,0) circle(.09);
\end{tikzpicture}
\par\vspace{0.2cm}
\begin{tikzpicture}[baseline=-.1cm, scale=1.3]
\fill[fill=red!30]
(-.4,-.35) to (.4,-.35) to (0,.35);
\draw[thick]
(-.4,-.35) to (.4,-.35) to (0,.35) to (-.4,-.35);
\draw[thick]
(0,.35) to (-.2,.7) to (.2,.7) to (0,.35)
(-.2,.7) to (.2,.7) to (0,1.05) to (-.2,.7)
(-.4,-.35) to (-.8,-.35) to (-.6,-.7) to (-.4,-.35)
(-.8,-.35) to (-.6,-.7) to (-1,-.7) to (-.8,-.35)
(.4,-.35) to (.8,-.35) to (.6,-.7) to (.4,-.35)
(.8,-.35) to (.6,-.7) to (1,-.7) to (.8,-.35);
\fill[red] (-.4,-.35) circle(.08);
\fill[red] (.4,-.35) circle(.08);
\fill[red] (0,.35) circle(.08);
\fill (-.2,.7) circle(.08);
\fill (.2,.7) circle(.08);
\fill (0,1.05) circle(.08);
\fill (-.8,-.35) circle(.08);
\fill (-.6,-.7) circle(.08);
\fill (-1,-.7) circle(.08);
\fill (.8,-.35) circle(.08);
\fill (.6,-.7) circle(.08);
\fill (1,-.7) circle(.08);
\end{tikzpicture}
\qquad
\begin{tikzpicture}[baseline=-.1cm, scale=1.25]
\fill[fill=red!30]
(-.4,.35) to (0,-.35) to (.8,-.35) to (.4,.35);
\draw[thick]
(-.8,-.35) -- (.8,-.35)
(-.4,.35) -- (.4,.35);
\draw[thick]
(-.8,-.35) to (-.4,.35) to (0,-.35) to (.4,.35) to (.8,-.35)
(-.4,.35) to (0,1.05) to (.4,.35)
(.8,-.35) to (1.6,-.35) ;
\fill (-.8,-.35) circle(.08);
\fill[red] (-.4,.35) circle(.08);
\fill[red] (0,-.35) circle(.08);
\fill[red] (.4,.35) circle(.08);
\fill[red] (.8,-.35) circle(.08);
\fill (0,1.05) circle(.08);
\fill (1.6,-.35) circle(.08);
\end{tikzpicture}
\qquad
\begin{tikzpicture}[baseline=-.1cm, scale=1.3]
\fill[fill=red!30]
(-.4,-.35) to (0,.35) to (.4,-.35);
\draw[thick]
(-.8,.35) -- (.8,.35)
(-1.2,-.35) -- (1.2,-.35);
\draw[thick]
(-1.2,-.35) to (-.8,.35) to (-.4,-.35) to (0,.35) to (.4,-.35) to
(.8,.35) to (1.2,-.35);
\fill (-1.2,-.35) circle(.08);
\fill (-.8,.35) circle(.08);
\fill[red] (-.4,-.35) circle(.08);
\fill[red] (0,.35) circle(.08);
\fill[red] (.4,-.35) circle(.08);
\fill (.8,.35) circle(.08);
\fill (1.2,-.35) circle(.08);
\end{tikzpicture}
\qquad
\begin{tikzpicture}[baseline=-.1cm, scale=1.3]
\fill[fill=red!30]
(-.4,-.35) to (0,.35) to (.4,-.35);
\draw[thick]
(-.8,.35) -- (.8,.35)
(-1.2,-.35) -- (1.2,-.35);
\draw[thick]
(-1.2,-.35) to (-.8,.35) to (-.4,-.35) to (0,.35) to (.4,-.35) to (.8,.35)
(1.2,-.35) to (.4,-.35);
\fill (-1.2,-.35) circle(.08);
\fill (-.8,.35) circle(.08);
\fill[red] (-.4,-.35) circle(.08);
\fill[red] (0,.35) circle(.08);
\fill[red] (.4,-.35) circle(.08);
\fill (.8,.35) circle(.08);
\fill (1.2,-.35) circle(.08);
\end{tikzpicture}
\par\vspace{0.2cm}
\begin{tikzpicture}[baseline=-.1cm, scale=0.6]
\draw[thick]
(0,0) -- (0,2) -- (2,2) -- (2,0) -- (0,0) 
(0,0) -- (.9,.9)
(1.1,1.1) -- (2,2)
(2,0) -- (0,2)
(-0.8,-0.8) -- (0,0)
(0,2) -- (-0.8,-0.8) -- (2,0)
(2,2) -- (3.5,2) -- (3.5,3.5) -- (2,3.5) -- (2,2)
(2,2) -- (2.65,2.65)
(2.85,2.85) -- (3.5,3.5)
(3.5,2) -- (2,3.5);

\fill (0,0) circle(.18);
\fill (2,0) circle(.18);
\fill (0,2) circle(.18);
\fill[red] (2,2) circle(.18);
\fill (-0.8,-0.8) circle(.18);
\fill (3.5,3.5) circle(.18);
\fill (2,3.5) circle(.18);
\fill (3.5,2) circle(.18);
\end{tikzpicture}
\hspace{1.3cm}
\begin{tikzpicture}[baseline=-.1cm, scale=0.6]
\draw[thick]
(0,0) -- (0,2) -- (2,2) -- (2,0) -- (0,0) 
(0,0) -- (.9,.9)
(1.1,1.1) -- (2,2)
(2,0) -- (0,2)
(-0.8,-0.8) -- (0,0)
(0,2) -- (-0.8,-0.8) -- (2,0)
(0,2) -- (1,3.2) -- (2,2) -- (3.2,1) -- (2,0);

\fill (0,0) circle(.18);
\fill (2,0) circle(.18);
\fill (0,2) circle(.18);
\fill[red] (2,2) circle(.18);
\fill (1,3.2) circle(.18);
\fill (3.2,1) circle(.18);
\fill (-0.8,-0.8) circle(.18);
\end{tikzpicture}
\hspace{1.3cm}
\begin{tikzpicture}[baseline=-.1cm, scale=0.6]
\draw[thick]
(0,0) -- (0,2) -- (2,2) -- (2,0) -- (0,0) 
(0,0) -- (.9,.9)
(1.1,1.1) -- (2,2)
(2,0) -- (0,2)
(-0.8,-0.8) -- (0,0)
(0,2) -- (-0.8,-0.8) -- (2,0)
(2,2) -- (4,2);

\fill (4,2) circle(.18);
\fill (0,0) circle(.18);
\fill (2,0) circle(.18);
\fill (0,2) circle(.18);
\fill[red] (2,2) circle(.18);
\fill (-0.8,-0.8) circle(.18);
\end{tikzpicture}
\caption{Red vertices are branch-complemented. Every subgraph whose vertices are all red is directionally branch-complemented. The examples in row $n$ have clique number $n+1$.}
\label{fig:fully_branching_vertices_cliques_3D}
\end{figure}
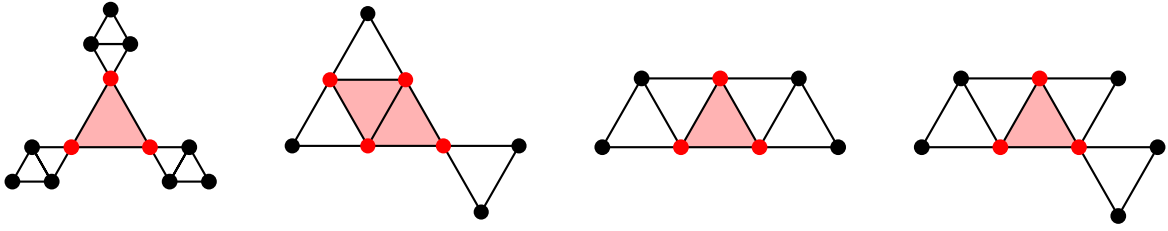

\begin{remark}\label{rem:fully_branching_not_branching}
It should be emphasized that being directionally branch-complemented imposes no additional condition on the edges of a subgraph: the property only depends on its vertices. Thus, although a directionally branch-complemented vertex is necessarily branching, a directionally branch-complemented clique of size at least $2$ need not be branching. For example, in the leftmost graph in the second row of \cref{fig:fully_branching_vertices_cliques_3D}, none of the three directionally branch-complemented edges are branching. More strongly, the third graph in \Cref{fig:fully_branching_graphs_subgraphs_3D} provides examples of directionally branch-complemented edges that are not branching, and are not even contained in any $3$--clique.
\end{remark}

When $\Gamma$ is a directionally branch-complemented subgraph of itself, we simply refer to $\Gamma$ as being directionally branch-complemented. We also refer to the associated RAAG $A_\Gamma$ as being directionally branch-complemented.

\begin{example}\label{ex:fully_branching_examples}
Here are some examples of directionally branch-complemented graphs.
\begin{itemize}
\item   As mentioned above, if $\Gamma$ is triangle-free and has no leaves then it is directionally branch-complemented. This applies in particular to all \emph{atomic} graphs as defined in \cite{bestvinakleinersageev:asymptotic}.

\item   More generally, suppose $\omega_\Gamma=n\ge2$ and every vertex of $\Gamma$ is contained in some $n$--clique. If every $(n-1)$--clique of $\Gamma$ is branching, then $\Gamma$ is directionally branch-complemented. This applies to the $1$--skeleton of any triangulation of a closed manifold.     

\item   If $\Gamma$ is a join of $n$ edgeless graphs of size at least two (so that $A_\Gamma$ is a product of nonabelian free groups), then $\Gamma$ satisfies the previous item, so is directionally branch-complemented.

\item If $\Gamma$ is a directionally branch-complemented graph, then so is its suspension $\widehat{\Gamma}$. Note that $A_{\widehat{\Gamma}} = A_\Gamma \times F_2$.
    
\item If $\Gamma_1$ and $\Gamma_2$ are directionally branch-complemented graphs, then so is their join $\Gamma_1 * \Gamma_2$. Note that $A_{\Gamma_1 * \Gamma_2} = A_{\Gamma_1} \times A_{\Gamma_2}$.
    
\item Let $\Gamma$ be a directionally branch-complemented graph of clique number $n$, let $K\subset \Gamma$ be a branching $(n-1)$--clique, and let $u \in \Gamma$ be a vertex that is not adjacent to any vertex of $K$. Let $\Gamma'$ be the graph obtained from $\Gamma$ by adding a vertex $v\notin \Gamma$, adjacent to every vertex of $K\cup \{u\}$. Then $\Gamma'$ is again directionally branch-complemented. See the bottom-left example in \Cref{fig:fully_branching_graphs_subgraphs_3D}.

\item If $\Gamma_1$ and $\Gamma_2$ are directionally branch-complemented graphs of clique number $n$, then any graph obtained from their disjoint union by adding edges between them, while keeping clique number $n$, is again directionally branch-complemented, see \Cref{fig:fully_branching_graphs_subgraphs_3D}.
\end{itemize}
\end{example}

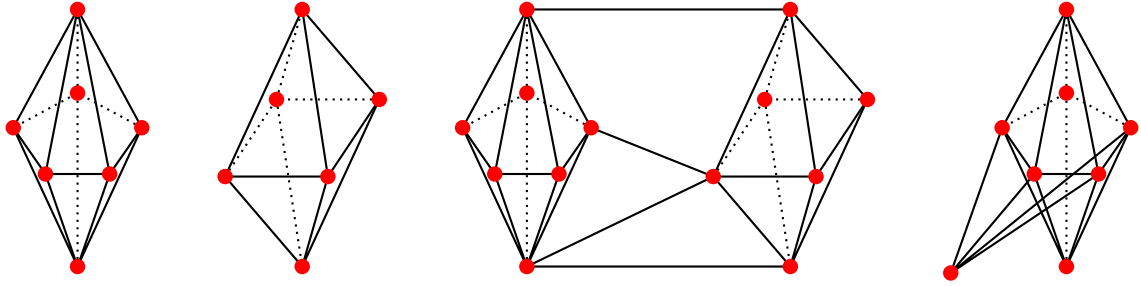
\begin{figure}[htbp]
    \centering
\begin{tikzpicture}[baseline=-.1cm, scale=1.4]
\coordinate (t) at (-.85,1);
\coordinate (b) at (-.85,-1);

\coordinate (v1) at (-1.35,.08);
\coordinate (v2) at (-1.10,-.28);
\coordinate (v3) at (-.60,-.28);
\coordinate (v4) at (-.35,.08);
\coordinate (v5) at (-.85,.35);

\draw[thick]
(v1) -- (v2) -- (v3) -- (v4);
\draw[thick,dotted]
(v4) -- (v5) -- (v1);

\draw[thick]
(t) -- (v1) (t) -- (v2) (t) -- (v3) (t) -- (v4)
(b) -- (v1) (b) -- (v2) (b) -- (v3) (b) -- (v4);
\draw[thick,dotted]
(t) -- (v5)
(b) -- (v5);

\fill[red] (v1) circle(.06);
\fill[red] (v2) circle(.06);
\fill[red] (v3) circle(.06);
\fill[red] (v4) circle(.06);
\fill[red] (v5) circle(.06);
\fill[red] (t) circle(.06);
\fill[red] (b) circle(.06);
\end{tikzpicture}
\hspace{1.5cm}
\begin{tikzpicture}[baseline=-.1cm, scale=1.4]
\draw[thick]
(-1.4,-.3) -- (-.6,-.3) to (-.2,.3);
\draw[thick, dotted]
(-1.4,-.3) -- (-1,.3) to (-.2,.3);
\draw[thick]
(-1.4,-.3) to (-.8,1)
(-.6,-.3) to (-.8,1)
(-.2,.3) to (-.8,1)
(-1.4,-.3) to (-.8,-1)
(-.6,-.3) to (-.8,-1)
(-.2,.3) to (-.8,-1);
\draw[thick, dotted]
(-1,.3) to (-.8,1)
(-1,.3) to (-.8,-1);
\fill[red] (-1.4,-.3) circle(.06);
\fill[red] (-.6,-.3) circle(.06);
\fill[red] (-.2,.3) circle(.06);
\fill[red] (-1,.3) circle(.06);
\fill[red] (-.8,1) circle(.06);
\fill[red] (-.8,-1) circle(.06);
\end{tikzpicture}
\hspace{1.5cm}
\begin{tikzpicture}[baseline=-.1cm, scale=1.4]

\coordinate (tL) at (-.85,1);
\coordinate (bL) at (-.85,-1);

\coordinate (v1) at (-1.35,.08);
\coordinate (v2) at (-1.10,-.28);
\coordinate (v3) at (-.60,-.28);
\coordinate (v4) at (-.35,.08);
\coordinate (v5) at (-.85,.35);

\coordinate (a1) at (.6,-.3);
\coordinate (a2) at (1.4,-.3);
\coordinate (a3) at (1.8,.3);
\coordinate (a4) at (1,.3);
\coordinate (tR) at (1.2,1);
\coordinate (bR) at (1.2,-1);

\draw[thick]
(v1) -- (v2) -- (v3) -- (v4);
\draw[thick,dotted]
(v4) -- (v5) -- (v1);

\draw[thick]
(tL) -- (v1) (tL) -- (v2) (tL) -- (v3) (tL) -- (v4)
(bL) -- (v1) (bL) -- (v2) (bL) -- (v3) (bL) -- (v4);
\draw[thick,dotted]
(tL) -- (v5)
(bL) -- (v5);

\draw[thick]
(a1) -- (a2) -- (a3);
\draw[thick,dotted]
(a1) -- (a4) -- (a3);

\draw[thick]
(a1) -- (tR)
(a2) -- (tR)
(a3) -- (tR)
(a1) -- (bR)
(a2) -- (bR)
(a3) -- (bR);
\draw[thick,dotted]
(a4) -- (tR)
(a4) -- (bR);

\draw[thick]
(tL) -- (tR)
(bL) -- (bR)
(bL) -- (a1)
(v4) -- (a1);

\fill[red] (v1) circle(.06);
\fill[red] (v2) circle(.06);
\fill[red] (v3) circle(.06);
\fill[red] (v4) circle(.06);
\fill[red] (v5) circle(.06);
\fill[red] (tL) circle(.06);
\fill[red] (bL) circle(.06);

\fill[red] (a1) circle(.06);
\fill[red] (a2) circle(.06);
\fill[red] (a3) circle(.06);
\fill[red] (a4) circle(.06);
\fill[red] (tR) circle(.06);
\fill[red] (bR) circle(.06);
\end{tikzpicture}
\par\vspace{.5cm}
\hspace{-2cm}


\begin{tikzpicture}[baseline=-.1cm, scale=1.4]
\coordinate (t) at (-.85,1);
\coordinate (b) at (-.85,-1);

\coordinate (v1) at (-1.35,.08);
\coordinate (v2) at (-1.10,-.28);
\coordinate (v3) at (-.60,-.28);
\coordinate (v4) at (-.35,.08);
\coordinate (v5) at (-.85,.35);

\coordinate (u) at (-1.75,-1.05);

\draw[thick]
(v1) -- (v2) -- (v3) -- (v4);
\draw[thick,dotted]
(v4) -- (v5) -- (v1);
\draw[thick]
(t) -- (v1) (t) -- (v2) (t) -- (v3) (t) -- (v4)
(b) -- (v1) (b) -- (v2) (b) -- (v3) ;
\draw[thick]
(b) -- (v4);
\draw[thick,dotted]
(t) -- (v5)
(b) -- (v5);

\draw[thick]
(u) -- (v4)
(u) -- (v1)
(u) -- (v2);

\fill[red] (v1) circle(.06);
\fill[red] (v2) circle(.06);
\fill[red] (v3) circle(.06);
\fill[red] (v4) circle(.06);
\fill[red] (v5) circle(.06);
\fill[red] (t) circle(.06);
\fill[red] (b) circle(.06);
\fill[red] (u) circle(.06);
\end{tikzpicture}
\hspace{2cm}
\begin{tikzpicture}[baseline=-.1cm, scale=1.4]
\draw[thick]
(-1.4,-.3) -- (-.6,-.3) to (-.2,.3);
\draw[thick, dotted]
(-1.4,-.3) -- (-1,.3) to (-.2,.3);
\draw[thick]
(-1.4,-.3) to (-.8,1)
(-.6,-.3) to (-.8,1)
(-.2,.3) to (-.8,1);
\draw[thick, dotted]
(-1,.3) to (-.8,1);
\draw[thick,dotted]
(-1,.3) -- (-1.45,-1);
\draw[thick]
(-1.4,-.3) -- (-1.45,-1) -- (-.2,-1) -- (-.6,-.3);
\draw[thick]
(-.2,.3) -- (-.2,-1);

\fill[red] (-1.4,-.3) circle(.06);
\fill[red] (-.6,-.3) circle(.06);
\fill[red] (-.2,.3) circle(.06);
\fill[red] (-1,.3) circle(.06);
\fill[red] (-.8,1) circle(.06);
\fill[red] (-1.45,-1) circle(.06);
\fill[red] (-.2,-1) circle(.06);
\end{tikzpicture}
    \caption{Some directionally branch-complemented graphs.}
    \label{fig:fully_branching_graphs_subgraphs_3D}
\end{figure}

\begin{remark}[Comparison] \label{rem:dbc_vs_out}
Let us compare the condition that $\Gamma$ is directionally branch-complemented with the condition that $\Out(A_\Gamma)$ is finite. According to \cite{laurence:generating,servatius:automorphisms}, the latter is equivalent to requiring that $\Gamma$ has no separating vertex star and no domination pair: distinct vertices $v,w$ such that $\lk(w)\subseteq \star(v)$. Therefore, in the triangle-free case, every graph $\Gamma$ satisfying the ``finite Out'' condition is directionally branch-complemented. The converse, however, already fails for the square $C_4$, which is directionally branch-complemented but has infinite outer automorphism group. 

More generally, the finite Out condition is not stable under taking joins, since domination pairs are often created in the process. By contrast, the join of two directionally branch-complemented graphs is again directionally branch-complemented; see \Cref{ex:fully_branching_examples}. 

In higher dimensions, the two notions become farther apart: it ceases to be true that RAAGS with finitely many outer automorphisms are directionally branch-complemented. For example, the triangular prism graph, that is, the Cartesian product of a triangle with an edge, is not directionally branch-complemented (there are no complementary cliques for vertices) but its associated RAAG has finitely many outer automorphisms. 

We now compare the directionally branch-complemented condition above to branching conditions appearing in work on hierarchically hyperbolic groups, interpreted for RAAGs. Let $n=\omega_\Gamma$. The conditions of \cite[\S5]{behrstockhagensisto:quasiflats} amount to asking that: every vertex that lies in some $n$--clique is the intersection of two $n$--cliques; and if $v,w\in\Gamma$ are adjacent vertices contained in $n$--cliques, then the edge $vw$ is the intersection of two $n$--cliques. These conditions were later refined in \cite{mangioni:random}, but the flavour is similar and in particular one still requires the condition on edges of $\Gamma$.
\end{remark}

The following generalises the last item in \cref{ex:fully_branching_examples} to subgraphs. 

\begin{example}
Let $\Gamma_1$ and $\Gamma_2$ be graphs with clique number $n$, and let $\Lambda_1\subset \Gamma_1$ and $\Lambda_2\subset \Gamma_2$ be directionally branch-complemented subgraphs. Let $\Gamma$ be obtained from the disjoint union of $\Gamma_1$ and $\Gamma_2$ by adding arbitrary edges between vertices of $\Lambda_1$ and vertices of $\Lambda_2$, without increasing the clique number above $n$. The subgraph of $\Gamma$ induced by $\Lambda_1\cup \Lambda_2$ is again directionally branch-complemented.  
\end{example}

We now compare \cref{def:fully_branching_in_graphs} to the definitions of \cite[\S9]{baderbensaidpetyt:from}.

\begin{lemma} \label{lem:finite_index}
    Let $\Gamma$ be a graph with clique number $n$.
    
    \begin{enumerate}
    \item Let $K\subseteq \Gamma$ be a directionally branch-complemented $k$--clique and let $F\subset X_\Gamma$ be a standard flat corresponding to $K$. 
    There is a graph $\Lambda$, a $k$--clique $K'\subseteq \Lambda$, a monomorphism $\iota\colon A_{\Lambda}\to A_{\Gamma}$ with $[A_\Gamma:\iota(A_\Lambda)]\le2^{n^2}$, and a standard flat $F'\subseteq X_{\Lambda}$ such that: $\iota(F)$ lies at Hausdorff distance at most $2^{n^2}$ from $F$, and $F'$ is directionally branch-complemented in the sense of \cite[Def.~9.5]{baderbensaidpetyt:from}. 

    \item Let $K\subseteq \Gamma$ be a clique consisting of vertices that lie in directionally branch-complemented $n$--cliques and let $F\subset X_\Gamma$ be a standard flat corresponding to $K$. 
    There exist $\Lambda$, $K'\subset\Lambda$, $\iota:A_\Lambda\to A_\Gamma$, and $F'\subset X_\Lambda$ as in the first item, except $F'$ is strongly directionally branch-complemented in the sense of \cite[Def.~9.9]{baderbensaidpetyt:from} and $2^{n^2}$ is replaced by $2^{n^2(n+1)}$.

    \item If $K\subseteq \Gamma$ is a branching clique and $F\subset X_\Gamma$ is a standard flat corresponding to $K$, then we can find $\Lambda$, $\iota$, and $F'$ as in the first item, except $F'$ is branching in the sense of \cite[Def.~9.1]{baderbensaidpetyt:from} (namely an intersection of $n$--flats), and $2^{n^2}$ is replaced by $2^n$.
%
    \end{enumerate}
\end{lemma}

The idea for the first item is as follows. Each vertex $v$ of $K$ lies in an $n$--clique and is the intersection of some maximal cliques. By doubling $\Gamma$ along vertices of those maximal cliques (see Item~\ref{sh:doubling}), we can produce a graph $\Lambda$ containing $K$ such that $v$ is an intersection of $n$--cliques in $\Lambda$. The subtlety is that $\Lambda$ will have multiple vertices that represent $A_\Gamma$-cosets of $\sgen{s_v}$ for $v\in K$, and only some of them will be given the stronger branching property by the sequence of doubles that we took. We therefore need to choose our sequence carefully in order to ensure that the cosets of $\sgen{s_v}$ lying in $F$ come from vertices of $\Lambda$ whose branching was upgraded.

\begin{proof}[Proof of \cref{lem:finite_index}]
For the first item, let $K=\{v_1,\dots,v_k\}$, and let $F$ be a standard flat associated with $K$. Since each $v_j$ is branching, by definition we have $v_j=\bigcap_{i=1}^{m_j}C^i_j$, where the $C^i_j$ are maximal cliques, and we can take $m_j\leq n$. There is also an $n$--clique $C_j$ containing $v_j$. For fixed $j$, the fact that the $C^i_j$ are maximal cliques whose intersection is $v_j$ implies that for each $j'\ne j$ there must be some $u_{j',j}\in\bigcup_{i=1}^{m_j}C^i_j$ such that $v_{j'}\notin\star(u_{j',j})$. In particular, $u_{j',j}\notin K$. Let $U=\{u_{j',j}\,:\,j'\ne j\}$, and enumerate it as $U=\{u_1,\dots,u_l\}$. We have $l\le n^2$.

Set $\Gamma_0=\Gamma, K_0=K$ and $F_0=F$. For $i\in\{1,\dots ,l\}$ we will produce $\Gamma_i$ as a double of $\Gamma_{i-1}$, with the corresponding inclusion $\iota_i\colon A_{\Gamma_i}\to A_{\Gamma_{i-1}}$ from Item~\ref{sh:doubling}. We will choose $K_i\subseteq \Gamma_i$ a copy of $K$ inside one of the two copies of $\Gamma_{i-1}$ and $F_i$ a standard flat corresponding to $K_i$ such that $\iota_i(F_i)$ is at Hausdorff distance at most 2 from $F_{i-1}$.

This is done inductively. Assume we have constructed $\Gamma_i$, $K_i$, and $F_i$. 
The graph $\Gamma_i$ contains $2^i$ copies of $\Gamma$, and $K_i$ is in one of them. Let $u_{i+1}'\in\Gamma_i$ be the copy of $u_{i+1}$ inside the same copy of $\Gamma$ as $K_i$. We let $\Gamma_{i+1}$ be the double of $\Gamma_i$ along $u_{i+1}'$.
It consists of two copies of $\Gamma_i$ glued along the star of $u_{i+1}'$. 
Let $\iota_{i+1}:A_{\Gamma_{i+1}}\to A_{\Gamma_i}$ be the corresponding inclusion from  Item~\ref{sh:doubling}. It is a 2--quasiisometry. Since $u_{i+1}'\notin K_i$, the graph $\Gamma_{i+1}$ has two subgraphs that are naturally associated with $K_i$, one for each copy of $\Gamma_i$. The standard flat $F_i$ in $X_{\Gamma_i}$, which corresponds to the $n$--clique $K_i$, is at Hausdorff distance at most 2 from $\iota(F_{i+1})$ for some standard flat $F_{i+1}\subset X_{\Gamma_{i+1}}$ corresponding to one of the two copies of $K_i$ inside $\Gamma_{i+1}$. Precisely which one that is can be read from a choice of coset representative for $F_i$ inside $A_{\Gamma_i}$ together with which standard generator of $A_{\Gamma_i}$ corresponds to the vertex $u_{i+1}'$. Let $K_{i+1}$ be that copy of $K_i$. 

From this inductive procedure we end up with a graph $\Gamma_l$, an $n$--clique $K_l\subset\Gamma_l$, a standard flat $F_l\subset X_{\Gamma_l}$, and a $2^l$--quasiisometry $\iota=\iota_1\circ\dots\circ\iota_l:A_{\Gamma_l}\to A_\Gamma$ such that $\iota(F_l)$ is at Hausdorff distance at most $2^l$ from $F$. By construction, each vertex of $K_l$ is an intersection of $n$--cliques inside $\Gamma_l$: the vertex coming from $v_j$ is an intersection of $n$--cliques of $\Gamma_l$ coming from $C_j$. Recalling that $l\le n^2$ gives the first statement. 

To prove the second statement we simply repeat the above argument. Above, we took an arbitrary $n$--clique $C_j$ containing the vertex $v_j\in K$ and showed that by doubling at most $n^2$ times we could make it so that each $v_j$ was the intersection of copies of $C_j$. This time, we can assume that each $C_j$ is itself directionally branch-complemented, so we can apply the above procedure for each $C_j$, and after at most $n^3$ doubles we have copies $C_j'$ of $C_j$ each of whose vertices is an intersection of $n$--cliques. Applying the argument for the first statement with $C'_j$ in place of $C_j$, after at most a further $n^2$ doubles we obtain a graph $\Lambda$ and a clique $L'\subset\Lambda$ corresponding to $L$ such that every vertex of $L'$ is an intersection of $n$--cliques whose vertices are themselves intersections of $n$--cliques. This establishes the second statement.

The proof of the final statement is similar but simpler as at most $n$ doubles are needed.
\end{proof}

The condition in the second item of \cref{lem:finite_index} may sound a bit technical, but it applies in many situations. Indeed, by unpicking the definitions we see the following.

\begin{remark} \label{rem:FB_implies_strong_FB} 
Let $\Gamma$ be a graph with clique number $n$. The statement that $\Gamma$ is directionally branch-complemented precisely means that every vertex is branch-complemented. In particular, every vertex of $\Gamma$ lies in a directionally branch-complemented $n$--clique. Hence every clique in $\Gamma$ has the property that each of its vertices lies in a directionally branch-complemented $n$--clique.
\end{remark}

We now state the consequences of the main results of \cite{baderbensaidpetyt:from} that are needed in this paper. In particular, we are restricting to the domain being the universal cover of a Salvetti complex, only considering standard flats, and forcing the codomain to be $n$--dimensional (rather than merely having asymptotic rank $n$). 

\begin{theorem}\label{thm:fully_branching_raag_rephrased}
Let $\Gamma$ be a graph with clique number $n$, and let $Y$ be an $n$--dimensional CAT(0) cube complex. For every $q$ there exists $D=D(n,q)$ such that the following holds for every $q$--quasiisometric embedding $f\colon X_\Gamma \to Y$.
    

Let $K \subset \Gamma$ be a clique consisting of vertices that lie in directionally branch-complemented $n$--cliques. If $F \subset X_\Gamma$ is a standard flat associated with $L$, then $f(F)$ lies at Hausdorff distance at most $D$ from a singular flat in $Y$. 
\end{theorem}

\begin{proof}
By \cref{lem:finite_index}, there exists a graph $\Lambda$, a clique $K'\subseteq \Lambda$, a $2^{n^3}$--quasiisometry $\iota\colon A_{\Lambda}\to A_{\Gamma}$, and a standard flat $F'\subseteq X_{\Lambda}$ such that $\iota(F')$ lies at Hausdorff distance at most $2^{n^3}$ from $F$ and $F'$ is strongly directionally branch-complemented. Applying \cite[Thm~10.3]{baderbensaidpetyt:from} to $f\circ\iota$ shows that $f\circ\iota(F')$ lies at Hausdorff distance at most $D'=D'(n,q)$ from a singular flat in $Y$. But $F$ and $\iota(F')$ lie at Hausdorff distance at most $2^{n^3}$ and $f$ is  quasiisometric embedding, so the result follows.
\end{proof}

It should be noted that not every geodesic in $X_\Gamma$ that is branching or branch-complemented in the sense of \cite{baderbensaidpetyt:from} is standard. The results of \cite{baderbensaidpetyt:from} therefore hold for more flats than we are considering here.

We shall use the following observation in \cref{sec:beyond_FB}. Recall that the asymptotic rank of a CAT(0) cube complex is the supremal $n$ such that it contains isometrically embedded boxes $[0,m]^n$ for all $m$. 

\begin{remark} \label{rem:rephrased_asrk}
If in \cref{thm:fully_branching_raag_rephrased} we only assume that $Y$ is a finite-dimensional CAT(0) cube complex of asymptotic rank $n$, then with the same proof, but using \cite[Thm~10.7]{baderbensaidpetyt:from} instead of \cite[Thm~10.3]{baderbensaidpetyt:from}, we obtain almost the same conclusion as in \cref{thm:fully_branching_raag_rephrased}. The only difference is that $D$ now depends on $Y$ and the flat in $Y$ can only be taken to be a ``semisingular'' flat.
\end{remark}

Analogously to \cite[Cor.~11.4]{baderbensaidpetyt:from}, it follows from \cref{thm:fully_branching_raag_rephrased} that if $\Gamma$ is directionally branch-complemented then $f$ induces a map between some appropriately defined boundaries. In view of \cref{rem:extension_graph_embeds_singular_boundary}, this restricts to give the following.

\begin{theorem}\label{thm:fully_branching_raag_to_sing_boundary}
Let $\Gamma$ be a directionally branch-complemented graph with clique number $n$. If $Y$ is an $n$--dimensional CAT(0) cube complex, then every quasiisometric embedding $f\colon X_\Gamma\to Y$ induces a graph embedding $\Gamma^\ext \to \partial_{\mathrm{sing}}Y$.
\end{theorem}

\begin{proof}
Recall from \cref{rem:FB_implies_strong_FB} that every clique in $\Gamma$ has the property that all of its vertices lie in directionally branch-complemented $n$--cliques. Hence, given a vertex $g\cdot v\in\Gamma^\ext$, with $v\in\Gamma$, \cref{thm:fully_branching_raag_rephrased} provides a singular geodesic $\gamma_v\subset Y$ at finite Hausdorff distance from $f(g\sgen{s_v})$. Define $f_*(g\cdot v)\in\partial_{\mathrm{sing}}Y$ to be a point represented by $\gamma_v$. The map $f_*$ is injective 
because distinct vertices of $\Gamma^\ext$ correspond to geodesics that are not at finite Hausdorff distance from one another.

Similarly, if $g\cdot\{v,w\}$ is an edge $\Gamma^{\ext}$, then \cref{thm:fully_branching_raag_rephrased} applies to the standard flat $g\sgen{s_v,s_w}$ and provides a singular flat at finite Hausdorff distance from $f(g\sgen{s_v,s_w})$. This implies that $f_*(g\cdot v)$ and $f_*(g\cdot w)$ are adjacent in $\partial_{\mathrm{sing}}Y$.
\end{proof}

\section{Universality with respect to quasiisometric embeddings} \label{sec:universal}

We say that a metric space $Y$ is a \emph{universal receiver} for quasiisometric embeddings of spaces in a family $\mathcal F$ if every $X\in \mathcal F$ admits a quasiisometric embedding into $Y$.

In this section, we use \cref{thm:fully_branching_raag_to_sing_boundary} to show that for each $n>1$, the fundamental group of a compact $n$--dimensional non-positively curved cube complex with no self-intersecting hyperplanes is not a universal receiver for quasiisometric embeddings of finitely generated $n$--dimensional RAAGs. In particular, no finitely generated RAAG of dimension $n$ is a universal receiver for quasiisometric embeddings of finitely generated $n$--dimensional RAAGs. Since every 2--dimensional RAAG subgroup of a 2--dimensional RAAG can be found as a quasiisometrically embedded subgroup (\cite[Cor.~1.15]{kimkoberda:embedability}), this generalises \cite[Thm~1.16]{kimkoberda:embedability}, where it is shown that no 2--dimensional RAAG is a universal receiver for subgroup embeddings of 2--dimensional RAAGs.

Our strategy is to find directionally branch-complemented RAAGs of fixed dimension whose defining graphs have arbitrarily large chromatic number, and to compare this chromatic number with the chromatic number of the singular boundary of the target. We write $\chi(\Gamma)$ for the chromatic number of a graph $\Gamma$. We refer to \cref{def:singular_boundary_graph} for the definition of singular boundary graph.

\begin{proposition}\label{prop:finite_chrom_nb_not_univ_receiver}
Let $Y$ be an $n$--dimensional CAT(0) cube complex, with $n \geq 2$. 
If $\chi(\partial_{\mathrm{sing}} Y) < \infty$, then $Y$ is not a universal receiver for quasiisometric embeddings of finitely generated $n$--dimensional RAAGs.
\end{proposition}

\begin{proof}
    Consider a family $\{\Gamma_m\,:\,m\in \mathbb{N}\}$ of triangle-free graphs without leaves whose chromatic numbers satisfy $\chi(\Gamma_m)\ge m$. For example, $\{\Gamma_m\}$ could be the family of Burling graphs \cite{burling:oncolouring}. For triangle-free graphs, having no leaves is equivalent to being directionally branch-complemented.
    
    Let $S^k\Gamma_m$ be the $k$--fold suspension of $\Gamma_m$. As taking a suspension increases the clique and chromatic numbers by 1, the family $\{S^k\Gamma_m\}$ is a set of directionally branch-complemented graphs (\cref{ex:fully_branching_examples}) with clique number $k+2$ such that $\chi(S^k\Gamma_m)\ge m+k$.

    For $n\ge2$, if $Y$ is an $n$--dimensional CAT(0) cube complex such that $A_{S^{n-2}\Gamma_m}$ can be quasiisometrically embedded in $Y$, then \cref{thm:fully_branching_raag_to_sing_boundary} tells us that $S^{n-2}\Gamma_m$ embeds in $\partial_{\mathrm{sing}}Y$. Hence $\chi(\partial_{\mathrm{sing}}Y)\ge \chi(S^{n-2}\Gamma_m)\ge m+n-2$. The result follows.
\end{proof}

Recall that the \emph{crossing graph} $\Delta(X)$ of a CAT(0) cube complex $X$ is the graph whose vertices are the hyperplanes of $X$, with an edge between two vertices whenever the corresponding hyperplanes cross.

\begin{lemma}\label{lem:chrom_sing_bound_less_chrom_crossing_graph}
If $X$ is a finite dimensional CAT(0) cube complex, then
$$
\chi(\partial_{\mathrm{sing}} X) \leq \chi(\Delta(X)).
$$
Moreover, for every graph $\Gamma$ with finite clique number, we have
$\chi(\partial_{\mathrm{sing}} X_\Gamma)=\chi(\Gamma)$.
\end{lemma}

\begin{proof}
    Let $c$ be a colouring of $\Delta(X)$. We use it to colour $\partial_{\mathrm{sing}}X$. Let $x\in \partial_{\mathrm{sing}}X$, and let $\gamma$ be a singular geodesic ray representing $x$. Define $C_x$ to be the set of colours appearing infinitely many times among the hyperplanes crossed by $\gamma$. By \Cref{lem:parallel_sing_geod_same_labels}, this is well defined and does not depend on the representative $\gamma$. Choose one element in every non-empty subset of the set of colours, and colour $x$ by the chosen element of $C_x$. 
    
    We claim that this defines a colouring of $\partial_{\mathrm{sing}}X$. Let $x,y$ be adjacent vertices of $\partial_{\mathrm{sing}}X$. By \cite[Lemma~4.6]{baderbensaidpetyt:from}, there is a $2$--orthant bounded by singular geodesic rays $\gamma_x$ and $\gamma_y$ representing $x$ and $y$, respectively. Therefore, the hyperplanes crossed by $\gamma_x$ and $\gamma_y$ intersect, and no colour can appear infinitely many times along both. Hence $C_x\cap C_y=\varnothing$, and $x$ and $y$ have distinct colours. 

    Now, let $\Gamma$ be a graph with finite clique number, and let $c$ be a colouring of $\Gamma$. Since all edges dual to a given hyperplane in $X_\Gamma$ have the same label, we colour this hyperplane by the colour of the corresponding label. Since the labels of edges of singular geodesic rays bounding an orthant are adjacent in $\Gamma$, the same argument as before implies that adjacent vertices in $\partial_{\mathrm{sing}}X_\Gamma$ have distinct colours. Therefore, $\chi(\partial_{\mathrm{sing}} X_\Gamma) \leq \chi(\Gamma)$. The other inequality holds because $\Gamma$ embeds in $\partial_{\mathrm{sing}} X_\Gamma$ (\cref{rem:extension_graph_embeds_singular_boundary}).     
\end{proof}

Note that the previous proof uses the axiom of choice to choose a colour in
each $C_x$. In the case of finitely many colours, which is enough for the
applications to finitely generated RAAGs below, one may simply choose the
smallest colour in $C_x$.

\begin{proposition}\label{prop:no_self_intersec_finite_chrom_sing_boundary}
    Let $X$ be a compact non-positively curved cube complex. If no hyperplane of $X$ self-intersects, then
    $$
    \chi(\partial_{\mathrm{sing}} \tilde{X}) < \infty.
    $$
\end{proposition}

\begin{proof}
    Since $X$ is compact, it has finitely many hyperplanes, and so $\chi(\Delta(X))<\infty$. We claim that $\chi(\Delta(\tilde X))\leq \chi(\Delta(X))$, which concludes by \cref{lem:chrom_sing_bound_less_chrom_crossing_graph}. Indeed, any colouring of $\Delta(X)$ induces a colouring of the hyperplanes of $\tilde X$ by colouring a hyperplane of $\tilde X$ with the colour of its projection to $X$. If two hyperplanes of $\tilde X$ cross, then they have distinct projections, otherwise some hyperplane of $X$ would self-intersect. Moreover, their projections cross, and hence have different colours. Therefore, $ \chi(\partial_{\mathrm{sing}} \tilde{X}) \leq  \chi(\Delta(\tilde X)) \leq \chi(\Delta(X)) <\infty$.
\end{proof}

Putting everything together, we obtain the main result of this section.

\begin{theorem}\label{thm:no_univ_cat0_cube_complex} 
Let $Y$ be a compact, $n$--dimensional, non-positively curved cube complex, with $n\geq 2$. If no hyperplane of $Y$ self-intersects, then $\pi_1(Y)$ is not a universal receiver for quasiisometric embeddings of finitely generated $n$--dimensional RAAGs. In particular, there is no finitely generated $n$--dimensional RAAG that is a universal receiver for quasiisometric embeddings of finitely generated $n$--dimensional RAAGs. 
\end{theorem}

\begin{proof}
    By \Cref{prop:no_self_intersec_finite_chrom_sing_boundary}, we have $\chi(\partial_{\mathrm{sing}}\tilde Y)<\infty$. Hence, by \cref{prop:finite_chrom_nb_not_univ_receiver}, $\tilde Y$ is not a universal receiver for quasiisometric embeddings of finitely generated $n$--dimensional RAAGs. Since $\pi_1(Y)$ is quasiisometric to $\tilde Y$, the same holds for $\pi_1(Y)$. The last assertion follows because hyperplanes in Salvetti complexes do not self-intersect.
\end{proof}

A natural question then is to ask whether that is the only obstruction: if we fix a chromatic number for the defining graph, is there a universal RAAG with respect to quasiisometric embeddings for this family? This is exactly the main theorem of \cite{rull:embedding}, which shows that $(F_2)^n$ is the universal RAAG for the class of RAAGs with chromatic number $n$.

\section{Geometric unconcealability} \label{sec:unconcealable}

In this section, we show that certain RAAGs cannot be concealed by a quasiisometric embedding: if such a RAAG quasiisometrically embeds into another RAAG $A_\Lambda$ of the same dimension, then this is already visible in the defining graph $\Lambda$. On the subgroup level, Kim--Koberda proved that cycle RAAGs are unconcealable \cite[Thm.~1.11]{kimkoberda:embedability}, recovering a theorem of Kambites for $F_2 \times F_2$. More precisely, they proved the following.

\begin{theorem}[{\cite[Cor.~3.9]{kambites:oncommuting}, \cite[Cor.~8.1]{kimkoberda:embedability}}] \label{thm:kk_irrepressible}
If $F_2\times F_2<A_\Lambda$ then $\Lambda$ contains an induced square. If $\Lambda$ is a finite, triangle-free graph such that $A_{C_n}<A_\Lambda$ for some $n\ge5$, then $\Lambda$ contains an induced $m$--cycle for some $m\in\{5,\dots,n\}$. 
\end{theorem}

We prove \cref{mthm:unconcealable}, which is a geometric analogue of this result that works in every dimension. 

To prove the theorem, we will show that if $F_2^p\times A_{C_{n_1}}\times\dots\times A_{C_{n_q}}$ quasiisometrically embeds in $A_{\Lambda}$ of the same dimension, then we can find $K_{p\times 2}* C_{n_1}*\dots *C_{n_q}$ as a geometrically embedded subgraph of $\Lambda^{\ext}$ and so $K_{p\times 2}* C_{m_1}*\dots *C_{m_q}$ as an induced subgraph of $\Lambda^{\ext}$. The proof then concludes by studying the extension graph, in \cref{prop:pulling_join_from_ext}.
Recall that $A_{\Gamma_1 * \Gamma_2} \cong A_{\Gamma_1} \times A_{\Gamma_2}$, and that $K_{n\times m}$ denotes the complete $n$--partite graph in which each part has size $m$.

\subsection{Extension graph data associated to singular geodesics}

Let $\Gamma$ be a simplicial graph. Note that every edge $e \subset X_\Gamma$ is contained in a unique standard geodesic, which we denote by $\gamma_e$. If $\alpha \subset X_\Gamma$ is a singular geodesic, define 
$$\operatorname{Ext}(\alpha) \subset V(\Gamma^\ext)$$ 
to be the set of vertices represented by the standard geodesics $\gamma_e$, where $e$ ranges over the edges of $\alpha$. Note that $|\operatorname{Ext}(\alpha)| =1$ if and only if $\alpha$ is a standard geodesic.

\begin{proposition}\label{prop:properties_of_EXT_of_singgeod}
Let $\Gamma$ be a simplicial graph, and let $\alpha \subset X_\Gamma$ be a singular geodesic.
\begin{enumerate}
    \item If $\alpha'$ is a singular geodesic parallel to $\alpha$, then $\operatorname{Ext}(\alpha)=\operatorname{Ext}(\alpha')$.
    \item If $\beta \subset X_\Gamma$ is a singular geodesic such that some parallel of $\alpha$ and some parallel of $\beta$ span a $2$--flat, then every vertex of $\operatorname{Ext}(\alpha)$ is adjacent to every vertex of $\operatorname{Ext}(\beta)$.
    \item The vertices of $\operatorname{Ext}(\alpha)$ are pairwise non-adjacent.
\end{enumerate}
\end{proposition}

\begin{proof}
    (1) Since $\alpha'$ is parallel to $\alpha$, it is contained in the parallel set of $\alpha$. The parallel set of $\alpha$ splits canonically as $P_\alpha=\alpha \times C$ for some CAT(0) cube complex $C$, see \cite[Lem.~3.4]{huang:quasiisometric:1}. By denoting $\alpha=\alpha\times\{c\}$ and $\alpha'=\alpha\times\{c'\}$, let $e\subset\alpha$ be an edge, and let $\sigma$ be a shortest combinatorial path in $C$ from $c$ to $c'$. Then $e\times \sigma \subset P_\alpha$ is a strip made of a chain of squares. So, moving $e$ successively along these squares produces an edge $e'\subset\alpha'$ corresponding to $e$. In particular, $e$ and $e'$ have the same label. Moreover, at each step, the standard geodesic through the current copy of $e$ is replaced by a parallel standard geodesic. Hence $\gamma_e$ and $\gamma_{e'}$ are parallel, so they determine the same vertex of $\Gamma^\ext$. Therefore $\operatorname{Ext}(\alpha)\subseteq \operatorname{Ext}(\alpha')$, and the reverse inclusion follows symmetrically.

    (2) Let $e \subset \alpha$ and $f \subset \beta$ be edges. Choose parallels $\alpha'$ and $\beta'$ of $\alpha$ and $\beta$, respectively, such that $\alpha'$ and $\beta'$ span a $2$--flat $F$, and let $e' \subset \alpha'$ and $f' \subset \beta'$ be the corresponding edges, as in (1). Let $e''$, resp.\ $f''$, be an edge of the square $e' \times f' \subset F$ parallel to $e'$, resp.\ $f'$. Since the singular geodesic in $F$ containing $e''$ is parallel to $\alpha'$, item (1) implies that $\gamma_{e''}$ and $\gamma_{e'}$ represent the same vertex of $\Gamma^\ext$. Similarly, $\gamma_{f''}$ and $\gamma_{f'}$ represent the same vertex of $\Gamma^\ext$. Since $e''$ and $f''$ are adjacent sides of a square, the standard geodesics $\gamma_{e''}$ and $\gamma_{f''}$ span a standard $2$--flat. Hence they are adjacent in $\Gamma^\ext$. Therefore, $\gamma_{e'}$ and $\gamma_{f'}$ are adjacent, and again by item (1) we get that $\gamma_e$ and $\gamma_f$ are adjacent as well.

    (3) Suppose that two vertices of $\operatorname{Ext}(\alpha)$ are adjacent. Then there exist edges $e,f\subset \alpha$ such that $\gamma_e$ and $\gamma_f$ admit parallel representatives spanning a standard $2$--flat. Let $H_e$ and $H_f$ be the hyperplanes dual to $e$ and $f$ respectively. Parallel singular geodesics cross the same hyperplanes. Moreover, if two singular geodesics span a singular $2$--flat, then every hyperplane dual to an edge of one of these geodesics crosses every hyperplane dual to an edge of the other. Hence $H_e$ crosses $H_f$. In particular, $\alpha$ crosses two crossing hyperplanes. This is impossible, since a CAT(0) geodesic contained in the $1$--skeleton of a CAT(0) cube complex cannot cross two crossing hyperplanes.
\end{proof}

\begin{remark}\label{rem:non_commuting_vertex_types_in_EXT}
If $\alpha \subset X_\Gamma$ is a non-standard geodesic, then there exist $u,v \in \operatorname{Ext}(\alpha)$ whose vertex types do not commute. Indeed, $\alpha$ contains two consecutive edges $e,e'$ with distinct labels $a,a'$. Since $\alpha$ is geodesic, the labels $a$ and $a'$ do not commute. The vertices of $\Gamma^\ext$ represented by the standard geodesics $\gamma_e$ and $\gamma_{e'}$ then lie in $\operatorname{Ext}(\alpha)$ and have vertex types $a$ and $a'$.
\end{remark}

\begin{proposition}\label{prop:EXT_of_images_of_SFB_geodesics}
    Let $\alpha$ and $\beta$ be geodesics spanning $\mathbb{R}^2$, let $\Lambda$ be a finite simplicial graph, and let $f\colon  \mathbb{R}^2 \to X_\Lambda$ be a quasiisometric embedding. Assume that there exists $D \geq 0$ such that, for every geodesic $\gamma$ parallel to $\alpha$ or $\beta$, the image $f(\gamma)$ lies at Hausdorff distance at most $D$ from a singular geodesic in $X_\Lambda$. Then, for any singular geodesics $\alpha'$ and $\beta'$ lying at finite Hausdorff distance from $f(\alpha)$ and $f(\beta)$, respectively, the following hold:
\begin{itemize}
    \item $\operatorname{Ext}(\alpha') \cap \operatorname{Ext}(\beta') = \varnothing$;
    \item every vertex of $\operatorname{Ext}(\alpha')$ is adjacent to every vertex of $\operatorname{Ext}(\beta')$.
\end{itemize}
\end{proposition}

\begin{proof}
    By \cite[Prop.~4.3]{baderbensaidpetyt:from}, there exists a singular $2$--flat $F \subset X_\Lambda$ such that $f(\mathbb R^2)$ lies at finite Hausdorff distance from $F$. Write $F=\ell_1\times \ell_2$, where $\ell_1,\ell_2$ are singular geodesics. Now let $\alpha' \subset X_\Lambda$ be a singular geodesic at finite Hausdorff distance from $f(\alpha)$. Since $f(\alpha)$ is contained in a bounded neighbourhood of $F$, \cref{lem:parallel_geod_in_convex} implies that $\alpha'$ admits a parallel singular geodesic contained in $F$. Therefore, $\alpha'$ is parallel to either $\ell_1$ or $\ell_2$. The same holds for $\beta'$. Moreover, $\alpha'$ and $\beta'$ cannot be parallel to the same factor. Indeed, otherwise $\alpha'$ and $\beta'$ would be parallel, hence at finite Hausdorff distance from each other. It would imply that $f(\alpha)$ and $f(\beta)$ are at finite Hausdorff distance from each other, contradicting the fact that $\alpha$ and $\beta$ span $\mathbb R^2$. We may therefore assume that $\alpha'$ is parallel to $\ell_1$ and $\beta'$ is parallel to $\ell_2$. By item (1) of \Cref{prop:properties_of_EXT_of_singgeod}, $\operatorname{Ext}(\alpha')=\operatorname{Ext}(\ell_1)$ and $\operatorname{Ext}(\beta')=\operatorname{Ext}(\ell_2)$. Since $\ell_1$ and $\ell_2$ span a $2$--flat, item (2) of \Cref{prop:properties_of_EXT_of_singgeod} implies that every vertex of $\operatorname{Ext}(\ell_1)$ is adjacent to every vertex of $\operatorname{Ext}(\ell_2)$. Therefore, every vertex of $\operatorname{Ext}(\alpha')$ is adjacent to every vertex of $\operatorname{Ext}(\beta')$. Finally, if $\operatorname{Ext}(\alpha') \cap \operatorname{Ext}(\beta')$ were non-empty, then any vertex in the intersection would be adjacent to itself, a contradiction.
\end{proof}

\subsection{Pulling cycles down from the extension graph}

In this subsection we prove \cref{thm:embedding_products_in_RAAGs_general}. The main step is to generalise \cite[Lem.~3.9(3)]{kimkoberda:embedability}, which says that if $C_n$ is an induced subgraph of $\Lambda^\ext$ for some $n \geq 4$, then there exists $4 \leq m \leq n$ such that $C_m$ is an induced subgraph of $\Lambda$. We extend this to joins of cycles. See also \cite[Corollary~8.1]{kimkoberda:embedability} for the subgroup-theoretic counterpart.

\begin{proposition}\label{prop:pulling_join_from_ext}
Let $q \geq 1$, let $n_1,\dots,n_q \geq 4$, and let $\Lambda$ be a finite simplicial graph. Let $\Sigma$ be a (possibly empty) finite graph such that $\Sigma * C_{n_1} * \cdots * C_{n_q}$ is an induced subgraph of $\Lambda^\ext$. Then there exist integers $m_1,\dots,m_q$ such that $4 \leq m_i \leq n_i$, and an induced subgraph
$$
C_{m_1} * \cdots * C_{m_q} \subset \Lambda
$$
such that every vertex of $C_{m_1} * \cdots * C_{m_q}$ is adjacent in $\Lambda$ to every vertex type appearing in $\Sigma$.

Moreover, if $\Sigma * C_{n_1} * \cdots * C_{n_q}$ has the same clique number as $\Lambda^\ext$ and $n_i$ is odd, then we can choose $m_i$ to be odd.
\end{proposition}

\begin{proof}
    By \cite[Lemma 3.1]{kimkoberda:embedability} (see also \cite{kimkoberdalee:finite}), the extension graph $\Lambda^\ext$ can be obtained as an increasing union of a sequence of induced subgraphs
    $$
    \Lambda=\Lambda_0 \subset \Lambda_1 \subset \Lambda_2 \subset \cdots,
    $$
    where each $\Lambda_{N+1}$ is obtained from $\Lambda_N$ by doubling along the star of a vertex. In particular, every finite subgraph of $\Lambda^\ext$ is contained in some $\Lambda_N$. We prove the proposition by induction on the smallest $N$ such that $ \Sigma *\bigast_{i=1}^q C_{n_i} \subseteq \Lambda_N$.

    If $N=0$, we may take $m_i:=n_i$ for every $i$ and the statement is trivial.
    

    Suppose now that the statement is true for $\Lambda_N$, and write 
    \[\Lambda_{N+1}=\Gamma \cup_{\str_{\Gamma}(v)=\phi(\str_{\Gamma}(v))} \Gamma',\]
    where $\Gamma=\Lambda_N$, $v\in \Gamma$, and $\phi\colon \Gamma\to \Gamma'$ is an isomorphism. Let $S=\star_\Gamma(v)=\star_{\Gamma'}(v)\subset\Lambda_{N+1}$. We denote $C_{n_i}$ by $C_i$ for every $i$. We distinguish two cases, either each $C_i$ is contained in $\Gamma$ or in $\Gamma'$, or not.

\textbf{Case 1. Each $C_i$ is contained in $\Gamma$ or $\Gamma'$.}
Assume first that each $C_i$ is contained in a single copy, i.e.\ either in $\Gamma$ or in $\Gamma'$. 
Note that there are no edges between $\Gamma-S$ and $\Gamma'-S$, hence either $\bigast_{i=1}^q C_i$ is fully contained in $\Gamma$ or in $\Gamma'$. Without loss of generality, we assume the former.

If some vertex $w \in \Sigma$ lies in $\Gamma' - S$, then, since $w$ is adjacent to every vertex of $\bigast_{i=1}^qC_i$, they all must lie in $S$. 
Let $\Sigma_0$ be the graph obtained by replacing every vertex in $w\in \Sigma\cap \Gamma'$ by its double $\phi^{-1}(w)\in \Gamma$. 
The graph $\Sigma_0$ has the same set of vertex types as $\Sigma$, and is adjacent to every vertex in $\bigast_{i=1}^qC_i$. 

Note that the clique number of $\Sigma_0*\bigast_{i=1}^qC_i$ is at least the clique number of $\Sigma*\bigast_{i=1}^qC_i$, as the map we defined preserves adjacencies.  Hence, if $\Sigma * \bigast_{i=1}^qC_i$ has the same clique number as $\Lambda^\ext$, then $\Sigma_0 * \bigast_{i=1}^qC_i$ has the same clique number as $\Lambda^{\ext}$. Thus, the conclusion follows from the induction hypothesis on $\Sigma_0 * \bigast_{i=1}^qC_i$.


\textbf{Case 2. There exists $i$ such that $C_i$ is not contained in $\Gamma$ or $\Gamma'$.}
Assume now that one factor, say $C_1$, is not contained in a single copy. First, note that $v \notin C_1$. Indeed, if $v \in C_1$, then the two neighbours of $v$ in $C_1$ must lie in $S$, by definition of $S$. Since $C_1$ is induced, these are the only two vertices of $C_1$ lying in $S$. It follows that the rest of the cycle is a path joining these two vertices with all internal vertices outside $S$, and therefore $C_1$ is contained in a single copy, a contradiction. 

As well as not containing $v$, we have that $C_1$ cannot lie inside $S$, because $S$ is contained in both $\Gamma$ and $\Gamma'$, and in particular in a single copy. Since there are no edges between $\Gamma - S$ and $\Gamma' - S$, and since every vertex of $\Sigma * C_2 * \cdots * C_q$ is adjacent to every vertex of $C_1$, it follows that $\Sigma * C_2 * \cdots * C_q$ is contained in $S$. Moreover it cannot contain $v$, by definition of $S$ and the fact that $C_1$ is not contained in $S$. In particular, every vertex of $\Sigma * C_2 * \cdots * C_q$ is adjacent to $v$. 

Choose a subpath $P=(x_0,x_1,\dots,x_k)$ of $C_1$ such that $x_0,x_k \in S$ and $x_1,\dots,x_{k-1} \notin S$. Since $C_1$ is not contained in a single copy, we can choose $P$ such that $2\leq k \leq n_1-2$. Moreover, $P$ is contained in a single copy, say $P \subset \Gamma$ (up to an isomorphism of $\Lambda_{N+1}$ interchanging $\Gamma$ and $\Gamma'$ via $\phi$). Since $C_1$ is induced and $k<n_1-1$, the endpoints $x_0$ and $x_k$ are not adjacent. 


We claim that the cycle 
    $$
    C_1' := (x_0,x_1,\dots,x_k,v)
    $$
    is an induced cycle contained in $\Gamma$. Indeed, the only edges between $x_0,\dots,x_k$ are the edges of the path $P$, since $C_1$ is induced. The vertex $v$ is adjacent to both $x_0$ and $x_k$, because they lie in $S$. Moreover, $v$ is not adjacent to any of the vertices $x_1,\dots,x_{k-1}$, since they lie outside $S$. 
    Finally, $x_0$ and $x_k$ are not adjacent, as discussed before. Hence $C_1'$ is an induced cycle of length $k+2$ and $4\leq k+2\leq n_1$. 
    
    We note that every vertex of $\Sigma * C_2 * \cdots * C_q$ is adjacent to every vertex of $P$ by hypothesis, and is adjacent to $v$ because $\Sigma * C_2 * \cdots * C_q \subset S$. Therefore, 
    $ 
    \Sigma*C_1' * C_2 * \cdots * C_q
    $
    is an induced subgraph of $\Gamma=\Lambda_N$. 

    For the moreover part, we will show that if $\Sigma*\bigast_{i=1}^qC_i$ has the same clique number as $\Lambda^{\ext}$ and $n_1$ is odd, then we can choose $P$ such that $k$ is odd and so $C_1'$ is of odd length. It is clear that $    \Sigma*C_1' * C_2 * \cdots * C_q$ has the same clique number as $ 
    \Sigma*C_1 * C_2 * \cdots * C_q
    $, so that will finish our proof, by the induction hypothesis.
    


    Assume that $\Sigma*\bigast_{i=1}^qC_i$ has the same clique number as $\Lambda^\ext$ and that $n_1$ is odd. Set $A:=\Sigma * C_2 * \cdots * C_q$. Since $A * C_1$ is induced and has the same clique number as $\Lambda$, and since $\omega_{C_1}=2$, we have $\omega_A=\omega_\Lambda-2$. We claim that $C_1$ has no edge with both endpoints in $S$. Indeed, suppose that $x,y$ are adjacent vertices of $C_1$ lying in $S$. Choose a clique $K \subset A$ of cardinality $\omega_\Lambda-2$. Since every vertex of $K$ is adjacent to $x$, $y$, and $v$, while $v$ is also adjacent to $x$ and $y$, it follows that $K \cup \{v,x,y\}$ is a clique of size $\omega_\Lambda+1$, a contradiction. Therefore $C_1$ has no edge with both endpoints in $S$. Hence the subpaths of $C_1$ with endpoints in $S$ and interior outside $S$ partition the edge set of $C_1$. Since $n_1$ is odd, the sum of their lengths is odd, so at least one of them has odd length. We choose $P=(x_0,x_1,\dots,x_k)$ to be such a path, and as explained $C_1'=(x_0,x_1,\dots,x_k,v)$ is of odd length and we finish by the induction hypothesis applied to $ 
    \Sigma*C_1' * C_2 * \cdots * C_q
    $. 
\end{proof}




We will use the following observation: 

\begin{lemma}\label{lem:singular_tripod_RAAG}
    Let $T$ be an infinite tripod, $\Lambda$ a simplicial graph, and $f\colon  T \to X_\Lambda$ a quasiisometric embedding. Suppose that $f$ maps every geodesic at finite Hausdorff distance from a singular geodesic of $X_\Lambda$. Then there exists a geodesic $\gamma \subset T$ such that $f(\gamma)$ is at finite Hausdorff distance from a singular, non-standard geodesic.
\end{lemma}
\begin{proof}
    Write $T := \gamma_1 \cup \gamma_2$, where $\gamma_1,\gamma_2$ are geodesics sharing a geodesic ray. We claim that $f(\gamma_1)$ and $f(\gamma_2)$ cannot both be at finite Hausdorff distance from standard geodesics. Indeed, suppose that $f(\gamma_1)$ and $f(\gamma_2)$ are respectively at finite Hausdorff distance from standard geodesics $\beta_1$ and $\beta_2$. Since $\gamma_1$ and $\gamma_2$ share a geodesic ray, the geodesics $\beta_1$ and $\beta_2$ have asymptotic subrays. However, standard geodesics with asymptotic rays are parallel. Indeed, by \Cref{lem:parallel_sing_geod_same_labels}, these asymptotic rays contain subrays with the same labels. Since $\beta_1$ and $\beta_2$ are standard, it follows that they are labelled by the same vertex $a$. Since they have asymptotic rays, there exist $D\geq 0$, and $g_1,g_2 \in A_\Lambda$ such that for all $n \geq 0$, $d(g_1 a^n, g_2 a^n) \leq D$, i.e.\ $a^{-n}g_1^{-1}g_2a^n$ lies in a ball of radius $D$. Since this ball is finite, two of these elements are equal, so $g_1^{-1}g_2$ commutes with $a^k$ for some $k \ne0$, hence with $a$. Therefore $\beta_1$ and $\beta_2$ are parallel. This implies that $f(\gamma_1)$ and $f(\gamma_2)$ are at finite Hausdorff distance, which contradicts the fact that $\gamma_1$ and $\gamma_2$ are not at finite Hausdorff distance. We conclude that  $f(\gamma_1)$ and $f(\gamma_2)$ cannot both be at finite Hausdorff distance from standard geodesics. 
\end{proof}

After introducing one more bit of notation, we will be ready to prove the main theorem of the section. 

\begin{definition}\label{def:floor_subgraph_ext_graph}
If $G$ is a finite subgraph of $\Lambda^\ext$, define $\lfloor G \rfloor$ to be the subgraph of $\Lambda$ induced by the set of vertex types appearing in $G$.
\end{definition}

Note that the RAAG associated to $K_{p\times 2}*C_{n_1}*\dots*C_{n_q}$ is $(p+2q)$--dimensional when every $n_i\geq 4$.

\begin{theorem}\label{thm:embedding_products_in_RAAGs_general}
Let $p,q \geq 0$, let $n_1,\dots,n_q\ge4$, and let $\Gamma = K_{p\times 2} * C_{n_1} *\cdots *C_{n_q} $. If $A_\Lambda$ is a $(p+2q)$--dimensional RAAG in which $A_\Gamma$ quasiisometrically embeds, then there are integers $m_1,\dots,m_q$ with $4 \leq m_i\leq n_i$ such that $K_{p\times 2} * C_{m_1} *\cdots *C_{m_q}$ is an induced subgraph of $\Lambda$. Moreover, if $n_i$ is odd, then we can take $m_i$ to be odd.
\end{theorem}

\begin{proof}
Let $\Gamma=K_{p\times 2}* C_{n_1} * \cdots * C_{n_q}$. 
 write $X_\Gamma =  X_{\Gamma_1} \times \cdots \times X_{\Gamma_p} \times X_{C_{n_1}} \times \cdots \times X_{C_{n_q}}$, where each $X_{\Gamma_i}$ is isometric to the 4--regular tree.

 First, note that $\Gamma$ is directionally branch-complemented, see \Cref{ex:fully_branching_examples}. 
Moreover, every singular geodesic in $X_{\Gamma_1}\times \cdots \times X_{\Gamma_p}$ is branch-complemented in the sense of \cite[Def.~9.3]{baderbensaidpetyt:from}. Since $\Lambda$ has the same clique number as $\Gamma$, it follows from \cite[Thm~10.3]{baderbensaidpetyt:from} that there exists a constant $D \geq 0$ such that every branch-complemented singular geodesic in $X_\Gamma$ is mapped by $f$ at Hausdorff distance at most $D$ from a singular geodesic of $Y$.

By \Cref{lem:singular_tripod_RAAG}, for each $i\in\{1,\dots,p\}$, there exists a singular geodesic $\alpha_i \subset X_{\Gamma_i}$, and a non-standard singular geodesic $\alpha_i' \subset X_{\Lambda}$, such that
    $$d_{\mathrm{Haus}} (f(\alpha_i), \alpha_i') \leq D.$$
    Since $\alpha_i'$ is non-standard, there exist $u_i, u_i' \in \operatorname{Ext}(\alpha_i')$ whose vertex types do not commute, see \Cref{rem:non_commuting_vertex_types_in_EXT}. Set $U_i := \{u_i,u_i'\}$. 

    For every $j \in \{1, \dots, q\}$, let $\beta_{1,j}, \dots, \beta_{n_j,j}$ be the standard geodesics based at the identity in $X_{C_{n_j}} \subseteq X_\Gamma$, and let $\beta'_{1,j}, \dots, \beta'_{n_j,j} \subset X_\Lambda$ be singular geodesics such that, for all $k \in \{1, \dots, n_j\}$, 
    $$d_{\mathrm{Haus}} (f(\beta_{k,j}),\beta'_{k,j})\leq D.$$
    For every $j \in \{1, \dots, q\}$ and every $k \in \{1, \dots, n_j\}$, take $v_{k,j} \in \operatorname{Ext}(\beta'_{k,j})$, and set $V_j := \{ v_{1,j}, \dots, v_{n_j,j}\}$. 

    It follows from \Cref{prop:EXT_of_images_of_SFB_geodesics} that the sets $U_1, \dots, U_p, V_1, \dots, V_q$ are pairwise disjoint, and that every vertex in one of these sets is adjacent to every vertex in the others. Therefore, they form a join 
    $$U_1 * \cdots *U_p *V_1 * \cdots * V_q\,\subset\,\Lambda^\ext.$$
    Moreover, by the second item of \Cref{prop:EXT_of_images_of_SFB_geodesics}, for each $j \in \{1, \dots, q\}$, the vertices of $V_j = \{ v_{1,j}, \dots, v_{n_j,j}\}$ commute with each other in a cyclic order. 
    Let $\bar V_j$ denote the subgraph of $\Lambda^\ext$ induced by the vertices of $V_j$. If $\bar V_j$ is not an induced cycle, replace it by a smaller induced one $C_j' \subset \bar V_j$. Such a cycle necessarily has length at least $4$, since otherwise the clique number of $U_1 * \cdots *U_p * \bar V_1 * \cdots * \bar V_q$ would exceed the clique number of $\Lambda^\ext$, which equals the clique number of $\Lambda$. So each $C_j'$ has length between 4 and $n_j$, inclusive.
    Moreover, it is easy to see that if $n_j$ is odd then $C_j'$ can be taken with odd length. Note that  
    $$U_1 * \cdots *U_p *C'_1 * \cdots * C'_q$$
    is also an induced subgraph of $\Lambda^\ext$, by construction. Set $\Sigma := U_1 * \cdots * U_p$. By the choice of $u_i,u_i' \in U_i$ for each $i$, the corresponding vertex types are not adjacent in $\lfloor \Sigma \rfloor$; see \Cref{def:floor_subgraph_ext_graph}. Moreover, since every vertex of each factor of $\Sigma$ is adjacent to every vertex of the other factors, the adjacency of the corresponding vertices in $\lfloor \Sigma \rfloor$ is preserved by the centralizer theorem for RAAGs \cite{servatius:automorphisms}. Therefore, $\Sigma$ is isomorphic to $\lfloor \Sigma \rfloor$. The conclusion follows now from \Cref{prop:pulling_join_from_ext}.
\end{proof}

If we restrict \Cref{thm:embedding_products_in_RAAGs_general} to products of copies $F_2$ with copies of $A_{C_5}$, we obtain the following.

\begin{corollary} \label{cor:embedding_prod_F2_and_5cycles}
    Let $p,q \geq 0$, let $\Gamma = K_{p\times 2} * \bigast_{i=1}^q C_5$, and let $A_\Lambda$ be a $p+2q$--dimensional RAAG. If $A_\Gamma$ quasiisometrically embeds in $A_\Lambda$, then $\Gamma$ is an induced subgraph of $\Lambda$.
\end{corollary}

\begin{remark} \label{rem:need_5}
The control on the $m_i$ in \cref{thm:embedding_products_in_RAAGs_general} is in a sense optimal. Indeed, 
\cite[Thm~1.1]{rull:embedding} shows that $A_{C_n}$ quasiisometrically embeds in $A_{C_4}$ for all even $n$, we cannot expect the lower bound on $m_i$ to be 5 when $n_i$ is even. Moreover, \cite[Thm~1.12, Cor.~1.15]{kimkoberda:embedability} shows that every $A_{C_n}$ with $n\ge5$ is an undistorted subgroup of $A_{C_5}$, so we cannot improve on the number 5. This is related to the observation that a square complex can coarsely contain cyclic unions of six orthants whilst having maximum vertex degree five, as shown by the configuration in \cref{fig:no_6_slide}. 

Note also that for the statement in \cref{thm:embedding_products_in_RAAGs_general} that we can take $m_i$ to be odd when $n_i$ is odd, it is essential that $\Lambda$ has the same clique number as $\Gamma$. Indeed, \cite{rull:embedding} also shows that $A_{C_n}$ quasiisometrically embeds in $(F_2)^3=A_{K_{3\times 2}}$ for all $n$.
\end{remark}

\begin{figure}[ht]
\includegraphics[height=4cm]{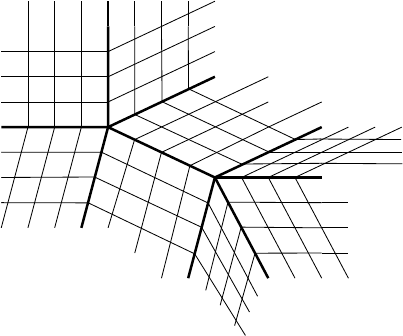}\centering
\caption{A CAT(0) square complex with maximal degree five whose Tits boundary is a cycle of length six.} \label{fig:no_6_slide}
\end{figure}

\cref{thm:embedding_products_in_RAAGs_general} in particular shows that the only RAAGs that quasiisometrically embed in RAAGs on forests are RAAGs on forests.
Droms showed that a 2--dimensional RAAG is coherent if and only if its defining graph is a forest \cite{droms:graph}, which gives us the following.

\begin{corollary} \label{cor:coherence}
Let $A_\Gamma$ be a 2--dimensional RAAG. If some incoherent RAAG quasiisometrically embeds in $A_\Gamma$, then $A_\Gamma$ is incoherent.
\end{corollary}

\section{Stable quasiisometric embeddings} \label{sec:qie_stable}

As discussed in the introduction, Rull constructed quasiisometric embeddings of RAAGs with $n$--colourable defining graph into $(F_2)^n$ \cite[Thm~1.1]{rull:embedding}. The embeddings she constructed are non-rigid, in the sense that they do not take standard geodesics Hausdorff-close to standard geodesics. In \cite{baderbensaidpetyt:quasiisometric:products} it is shown that this is necessary if the domain is not itself a product of free groups.
In this section, we consider quasiisometric embeddings between more general $n$--dimensional RAAGs, and show that in many situations such non-rigid quasiisometric embeddings cannot exist. We formalise this with the following definition.

\begin{definition}\label{def:stable_QIE}
A quasiisometric embedding $\varphi\colon A_\Gamma\to A_\Lambda$ is said to be \emph{stable} if there is a constant $D$ such that $\varphi$ sends every standard geodesic of $A_\Gamma$ within Hausdorff distance at most $D$ of some standard geodesic of $A_\Lambda$. It is \emph{weakly stable} if $\varphi$ sends every standard geodesic of $A_\Gamma$ within finite Hausdorff distance of a standard geodesic of $A_\Lambda$. 
\end{definition}

Stability allows us to pass from the coarse setting of quasiisometric embeddings to the combinatorics of extension graphs (see \cref{def:extension_graph}). 

\begin{proposition} \label{prop:stable_implies_induced}
Let $\Gamma$ and $\Lambda$ be graphs, and let $\varphi\colon A_\Gamma\to A_\Lambda$ be a quasiisometric embedding. If $\varphi$ is weakly stable and the $\varphi$--image of every standard 2--flat in $X_\Gamma$ is at finite Hausdorff distance from a 2--flat in $X_\Lambda$, then $\varphi$ induces a combinatorial embedding $f\colon \Gamma^\ext\to \Lambda^\ext$.

In particular, this occurs when $\varphi$ is stable.
\end{proposition}

\begin{proof}
Recall that $A_\Gamma$ can be viewed as the 1-skeleton of the cube complex $X_\Gamma$. We view $\varphi$ as a map from $X_\Gamma$ to $X_\Lambda$. As $\varphi$ preserves parallelism (up to finite Hausdorff distance), it induces a map $f$ from parallelism classes of standard geodesics in $X_\Gamma$ to parallelism classes of standard geodesics in $X_\Lambda$, or in other words from the vertices of $\Gamma^\ext$ to the vertices of $\Lambda^\ext$ (see \cref{rem:extension_graph_definition}). This map $f$ is injective, because $\varphi$ is a quasiisometric embedding and non-parallel standard geodesics in a RAAG diverge. 

If vertices $v,w\in \Gamma^\ext$ are adjacent, then they correspond to parallelism classes of standard geodesics that can be represented by a pair of standard geodesics $\gamma_v$ and $\gamma_w$ that span a standard 2--flat $F\subset X_\Gamma$. We claim that $f(v)$ and $f(w)$ are adjacent in $\Lambda^\ext$.

By assumption, the quasiflat $\varphi(F)$ lies at finite Hausdorff distance from a 2--flat $F'\subset X_\Lambda$. 
Let $\beta_v$ and $\beta_w$ be standard geodesics in $X_\Lambda$ that lie at finite Hausdorff distance from $f(\gamma_v)$ and $f(\gamma_w)$, respectively.
According to \cref{lem:parallel_geod_in_convex}, both $\beta_v$ and $\beta_w$ are parallel into $F'$. Since $\gamma'_v$ and $\gamma'_w$ are not parallel, their parallels in $F'$ must be orthogonal, and $F'$ is a standard 2--flat. Hence $f$ sends $v$ and $w$ to adjacent vertices of $\Lambda^\ext$.

When $\varphi$ is stable, \cite[Prop.~4.3]{baderbensaidpetyt:from} tells us that the $\varphi$--image of every standard $2$--flat is at (uniform) finite Hausdorff distance from a $2$--flat in $X_\Lambda$ and so the conclusion follows.
\end{proof}

Our next goal is to show that under mild assumptions on $\Gamma$ and $\Lambda$, all quasiisometric embeddings $A_\Gamma\to A_\Lambda$ are stable. We will need the following lemma.

\begin{lemma}\label{lem:width_one_strip_outside_flat}
    Let $X$ be a CAT(0) cube complex, let $\alpha \subset X$ be a singular geodesic, and let $F \subset X$ be a singular flat containing a parallel copy of $\alpha$. Suppose that there exists a parallel copy of $\alpha$ not contained in $F$. Then there exist two singular geodesics $\alpha' \subset F$ and $\beta' \not\subset F$, both parallel to $\alpha$, which bound a strip subcomplex of width one.
\end{lemma}

\begin{proof}
    Let $n$ be the dimension of $F$, and let $\beta$ be a singular geodesic parallel to $\alpha$ such that $\beta \not\subset F$. Since $F$ contains a parallel copy of $\alpha$, $F$ is contained in the parallel set $P(\alpha)$. Since $X$ is a CAT(0) cube complex, $P(\alpha)$ is a convex subcomplex that splits as $\alpha\times C$, where $C$ is a CAT(0) cube complex; see \cite[Rem.~3.5]{huang:quasiisometric:1}. Under this splitting, $F$ corresponds to $\alpha\times H$ for some singular $(n-1)$--flat $H \subset C$, and $\beta$ corresponds to $\alpha\times \{v\}$ for some vertex $v \in C$. Since $\beta \not\subset F$, it follows that $v \notin H$. Since $P(\alpha)$ is connected, so is $C$. Thus, there exists a combinatorial path in $C$ from $v$ to a vertex of $H$. Let $h$ be the first vertex of this path that belongs to $H$, and let $c$ be the previous vertex. Then $c \notin H$ and $c$ is adjacent to $h$. Set
    $\alpha' := \alpha\times \{h\}$ and $\beta' := \alpha\times \{c\}$. Then $\alpha' \subset F$, while $\beta' \not\subset F$. Both are parallel to $\alpha$, and they bound the strip subcomplex $\alpha\times [h,c]$, which has width one.
\end{proof}

\begin{theorem} \label{thm:squarefree_stable}
Let $\varphi\colon  A_\Gamma\to A_\Lambda$ be a quasiisometric embedding between
$n$--dimensional RAAGs. Assume that $\Lambda$ does not contain, as a combinatorially embedded subgraph, the suspension of the disjoint union of a vertex and an $(n-1)$--clique. If $\Gamma$ is directionally branch-complemented, then $\varphi$ is stable.
\end{theorem} 

\begin{proof} 
Let $\gamma \subseteq X_\Gamma$ be a standard geodesic, associated with a vertex $v\in\Gamma$. Since $v$ is branching, there is an $n$--clique $K$ and a maximal clique $L\ne K$ such that $v\in K\cap L$. Since $\Gamma$ is directionally branch-complemented, all vertices of $K$ and $L$ lie in directionally branch-complemented $n$--cliques. Let $F,E\subset X_\Gamma$ be standard flats associated with $K$ and $L$, respectively, such that $\gamma\subset F\cap E$. 

Let $q$ be a quasiisometric embedding constant for $\varphi$, and let $D$ be a constant given by \cref{thm:fully_branching_raag_rephrased}. That theorem provides singular flats $\gamma',F',E'\subset X_\Lambda$ at Hausdorff distance at most $D$ from $\varphi(\gamma)$, $\varphi(F)$, and $\varphi(E)$, respectively. Our goal is to show that $\gamma'$ is standard.

Both $F'$ and $E'$ contain singular geodesics at finite Hausdorff distance from $\gamma'$, which are necessarily parallel to it. The flats $F$ and $E$ do not lie at finite Hausdorff distance, so $E$ must contain a parallel copy of $\gamma'$ that is not contained in $F'$.
%
By \cref{lem:width_one_strip_outside_flat}, there exist singular geodesics $\alpha_1\subset F'$ and $\beta\not\subset F'$ parallel to $\gamma'$ that bound a strip subcomplex of width one. Let $s\in A_\Lambda$ be the label of the edges traversing that strip subcomplex.

Let $\alpha_2,\dots,\alpha_n$ denote the singular geodesics in $F'$ such that $\{\alpha_1,\dots,\alpha_n\}$ form a family of pairwise-orthogonal geodesics that span $F'$. Let $L_j$ denote the set of labels of edges in $\alpha_j$. By \cref{lem:parallel_sing_geod_same_labels}, $L_1$ is also the set of labels of $\gamma'$, since $\alpha_1$ is parallel to $\gamma'$. We will show that $|L_1|=1$, which proves the theorem. 

Suppose towards a contradiction that $|L_1| \geq 2$, and let $t_1,t_2 \in L_1$ be distinct.
Note that the $L_j$ are pairwise disjoint, because the $\alpha_j$ span a flat. Moreover, if $i\ne j$ then every label in $L_i$ commutes with every label in $L_j$. In particular, for every $j \geq 2$, the labels of $L_j$ commute with both $t_1$ and $t_2$. This also implies that $t_1$ and $t_2$ do not commute: otherwise, together with one label from each of $L_2,\dots,L_n$, they would correspond to an $(n+1)$--clique in $\Lambda$. 

The label $s$ commutes with every label in $L_1$, since the strip between $\alpha_1$ and $\beta$ is a product subcomplex whose transverse edges have label $s$. For each $j \geq 2$, choose $a_j \in L_j \setminus \{s\}$. This is possible because if $L_j=\{s\}$ then the width-one strip labelled by $s$ and starting from $\alpha_1$ would already be contained in $F'$, contradicting the fact that $\beta\not\subset F'$.

Therefore, $a_2, \dots, a_n$ correspond to an $(n-1)$--clique in $\Lambda$ that is disjoint from the vertex $v_s$. Moreover, both $t_1$ and $t_2$ commute with all of $a_2, \dots, a_n,s$. Hence $\Lambda$ contains a subgraph isomorphic to the suspension of the disjoint union of an $(n-1)$--clique and a vertex, with suspension vertices $v_{t_1}$ and $v_{t_2}$. This contradicts the assumption on $\Lambda$. We conclude that $|L_1|=1$, and therefore $\gamma'$ is a standard geodesic. \qedhere

\end{proof}



\begin{remark}
    We cannot expect the assumption in \cref{thm:squarefree_stable} to be that $\Lambda$ does not contain an \emph{induced} subgraph that is the suspension of a disjoint union of a vertex and an $(n-1)$--clique, because of the example in \cref{fig:suspension_not_induced}: $A_\Gamma$ and $A_{D_v\Gamma}$ are commensurable, but $\Gamma$ does not contain an induced suspension of a 3--clique and a vertex while $D_v\Gamma$ does. 
\end{remark}

\begin{figure}[ht]
\includegraphics[height=4cm]{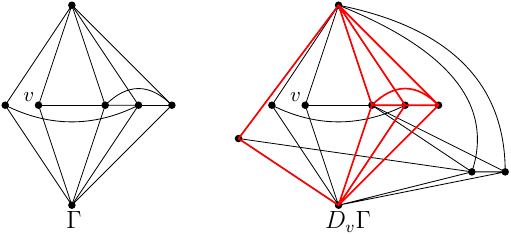}\centering
\caption{The graph $\Gamma$ does not contain an induced suspension of a $3$--clique and a vertex, but $D_v\Gamma$ does (in red).}\label{fig:suspension_not_induced}
\end{figure}

\begin{remark}\label{rem:square_free_implies_no_suspension}
    The assumption on $\Lambda$ in \Cref{thm:squarefree_stable} is satisfied if $\Lambda$ does not contain an induced square. Indeed, suppose that $\Lambda$ contains the suspension of the disjoint union of a vertex $z$ and an $(n-1)$--clique spanned by $a_2,\dots,a_n$, and let $x,y$ be the suspension vertices. Then $z$ is not adjacent to at least one of the $a_i$'s; otherwise $z,a_2,\dots,a_n,x$ would span an $(n+1)$--clique. Hence, for some $i$, the vertices $x,z,y,a_i$ span an induced square.

    When $n=2$, the two conditions are equivalent. Indeed, the suspension of the disjoint union of a vertex and a $1$--clique is a square. Since $\Lambda$ is triangle-free, every square subgraph is necessarily induced.
\end{remark}

\begin{corollary}\label{cor:squarefree_stable}
Let $A_\Gamma$ and $A_\Lambda$ be $n$--dimensional RAAGs such that $\Gamma$ is directionally branch-complemented and $\Lambda$ does not contain, as a subgraph, the suspension of the disjoint union of a vertex and an $(n-1)$--clique. If $\varphi\colon  A_\Gamma\to A_\Lambda$ is a quasiisometric embedding, then $\varphi$ induces a combinatorial embedding $\Gamma^\ext \to \Lambda^\ext$.
\end{corollary}

\begin{remark} \label{rem:girth}
\cref{cor:squarefree_stable} gives a method for obstructing the existence of quasiisometric embeddings between certain RAAGs: one has to rule out the existence of $\Gamma^\ext$ as a subgraph of $\Lambda^\ext$. For example, if $\Lambda$ is a graph of girth greater than $k\ge4$, then $A_{C_k}$ cannot be quasiisometrically embedded in $A_\Lambda$, because of \cite[Lem.~3.9(3)]{kimkoberda:embedability}. See \cref{sec:unconcealable} for expansion on this.
\end{remark}

Building on the idea of \cref{rem:girth}, \cite{baderbensaidpetyt:quasiisometric:flexibility} gives a complete classification of when $A_{C_m}$ can be quasiisometrically embedded in $A_{C_n}$.


\begin{remark}
As mentioned previously, Rull's embeddings of $n$--colourable RAAGs into $(F_2)^n=A_{K_{n\times2}}$ \cite[Thm~1.1]{rull:embedding} are essentially never stable (see \cite{baderbensaidpetyt:quasiisometric:products}), so the conclusion of \Cref{thm:squarefree_stable} does not hold in general. Note, though, that $K_{n\times2}$ contains, as a subgraph, the suspension of the disjoint union of a vertex and an $(n-1)$--clique. 

The theorem also requires some assumption on $\Gamma$. For example, \cite[Thm~5.3]{behrstockneumann:quasiisometric} shows that all RAAGs on trees of diameter at least 3 are quasiisometric. If there were a stable quasiisometry $A_{P_4}\to A_{P_5}$, where $P_k$ denotes the path graph on $k$ vertices, then \cref{prop:stable_implies_induced} would yield an isomorphism between $P_4^{\ext}$ and $P_5^{\ext}$. But $P_4^\ext$ contains $P_4$ subgraphs that cannot be extended to longer paths, whereas all maximal path subgraphs of $P_5^\ext$ have an odd number of vertices.
\end{remark}

\section{Quasiisometric embeddings of RAAGs in themselves} \label{sec:self_qie}

\cite{baderbensaidpetyt:quasiisometric:flexibility} is dedicated to constructing exotic quasiisometric embeddings between RAAGs defined on cycles of differing lengths. In this section we complement this by establishing rigidity results for self--quasiisometric-embeddings. More precisely, we will characterise all self--quasiisometric-embeddings for RAAGs defined on cycles of length at least five, and more generally for certain RAAGs with finite outer automorphism group.

Throughout this section, we will use notation as in Item~\ref{sh:vertex_notation}. In particular, whenever we consider a RAAG, its standard generating set will be denoted by $S$, and $v_s$ will be the vertex in the defining graph corresponding to a generator $s\in S$. We will always consider RAAGs with their word metric coming from the standard generating set.


\begin{example} \label{eg:self_qi}
Here are some examples of self-quasiisometries of RAAGs.
    \begin{enumerate}
    \item Any automorphism of a finitely generated group is a quasiisometry. By \cite{laurence:generating,servatius:automorphisms}, if $\Out A_\Gamma$ is finite, then it is generated by inner automorphisms, inversions of standard generators, and automorphisms induced from $\Aut\Gamma$. 
    \item For every group $G$ there is a map $L:G\to\QI(G)$, given by sending $g\in G$ to the left-multiplication operator $L_g$. The isometry $L_g$ is at a finite distance from the inner automorphism that is conjugation by $g$. Indeed, $\dist(L_g(x),gxg^{-1})=|g|$ for every $x\in G$. If $G=A_\Gamma$ is a RAAG, then the kernel of this map is exactly the centre $Z(A_\Gamma)$.
    \item Every non-abelian RAAG has self-quasiisometries that are not close to homomorphisms. For example, if $s\in A_\Gamma$ is a non-central generator then we can consider the map $\varphi\colon G\to G$ defined by setting $\varphi(g)=sg$ if $g=ss_2^{k_2}\cdots s_n^{k_n}$ is reduced and syllable-reduced, $\varphi(g)=s^{-1}g$ if $g=s^2s_2^{k_2}\cdots s_n^{k_n}$ is reduced and syllable-reduced, and $\varphi(g)=g$ otherwise. This is an example of a \emph{change-of-power route map}: see \cref{def:change-of-power}. \cref{lem:power_map_is_qie} and  \cref{lem:power_map_is_surjective} show that $\varphi$ is a quasiisometry of $A_\Gamma$, but one can easily check it directly in this case. Furthermore, $\varphi$ is not at finite distance from a homomorphism, as shown in \cref{prop:finite_dist_from_homo}.
    \end{enumerate}
\end{example}

We will show in \cref{thm:qi_of_finite_out_to_itself} that if $\Out A_\Gamma$ is finite, then, under an additional assumption on $\Gamma$, \cref{eg:self_qi} encompasses essentially all self-quasiisometries of $A_\Gamma$, in the sense that any such quasiisometry decomposes as a composition of an automorphism coming from $\Aut\Gamma$, a change-of-power map, and a left multiplication.



The self-quasiisometries in \cref{eg:self_qi} 
are stable and take standard geodesics of a fixed type to standard geodesics of a fixed type. This leads us to the following definition. Recall from \cref{prop:stable_implies_induced} that stable quasiisometric embeddings between RAAGs induce combinatorial embeddings of the associated extension graphs.

\begin{definition}[Type-mixing]
We say a combinatorial embedding $f\colon  \Gamma^{\ext}\to \Lambda^{\ext}$ is \emph{type-mixing} if there is a vertex $v\in\Gamma$ and a pair of elements $g_1,g_2\in A_\Gamma$ such that $\type(f(g_1\cdot v))\ne\type(f(g_2\cdot v))$.
    
We say a stable quasiisometric embedding $\psi\colon A_\Gamma\to A_\Lambda$ is \emph{type-mixing} if the induced combinatorial embedding between the extension graphs is type-mixing. 
%
\end{definition}


Type-mixing quasiisometric embeddings are fundamentally different from homomorphisms.

\begin{lemma} \label{lem:homos are not type-mixing}
If $\varphi\colon A_\Gamma\to A_\Lambda$ is a stable quasiisometric embedding and a homomorphism between right-angled Artin groups, then $\varphi$ is not type-mixing.
\end{lemma}

\begin{proof}
    Let $s$ be a standard generator of $A_\Gamma$ and let $g\in A_\Gamma$. Let $h\sgen t$ be a standard geodesic in $A_\Lambda$ that is at finite distance from $\varphi(\sgen s)$. 
    For every $g\in A_\Gamma$ we have that $\varphi(gs^k)=\varphi(g)\varphi(s^k)$, so $\varphi(g\sgen s)$ is at finite distance from $\varphi(g)h\sgen t$.
\end{proof}

In particular, type-mixing quasiisometric embeddings between right-angled Artin groups do not lie at finite distance from homomorphisms.

\subsection{Route maps}

Here we describe a way to reduce the problem of understanding all self--quasiisometric-embeddings of a RAAG $A_\Gamma$ to understanding a certain family of combinatorially defined embeddings. 

\begin{definition}[Route map]
Let $\Gamma$ be a graph. A \emph{route map} is a map $F\colon A_\Gamma/Z(A_\Gamma)\to A_\Gamma/Z(A_\Gamma)$ such that there exists a combinatorial embedding $f\colon \Gamma^\ext\to\Gamma^\ext$ with the property that $f(g\cdot\Gamma)=F(g)\cdot\Gamma$ for all $g\in A_\Gamma$. We also call $F$ the \emph{route map of $f$}. 
\end{definition}

We mod out by the centre, because if $z\in Z(A_\Gamma)$, then $gz\cdot \Gamma=g\cdot \Gamma$. We note that if $A_\Gamma$ has centre, then $A_\Gamma/Z(A_\Gamma)=A_{\Gamma'}$ where $\Gamma'$ is the induced subgraph of $\Gamma$ obtained by removing all vertices corresponding to central generators. By knowing the size of $\Gamma$, one can reconstruct $\Gamma$ and $\Gamma^{\ext}$ from $\Gamma'$. In $\Gamma^\ext$ there is only one vertex of a given central type and any combinatorial embedding $\Gamma^{\ext}\to \Gamma^\ext$ will map it to itself, up to an automorphism of $\Gamma$. Thus, in this section we will disregard the centre.

\bsh{Convention} \label{sh:convention_centreless}
All RAAGs considered in this section are centreless unless otherwise specified.
\esh

\begin{lemma} \label{lem:unique_route}
A combinatorial embedding $f\colon \Gamma^\ext\to\Gamma^\ext$ has at most one route map.
\end{lemma}

\begin{proof}
If $h$ and $h'$ satisfy that $h\cdot \Gamma=h'\cdot \Gamma$ then $h^{-1}h'$ commutes with all standard generators, hence it is in $Z(A_\Gamma)$. Thus, if $f(g\cdot \Gamma)=F(g)\cdot \Gamma$ for some $F(g)\in A_\Gamma$, then $F(g)$ is unique up to the centre.
\end{proof}

The next two lemmas are what allow us to reduce to studying route maps.

\begin{lemma} \label{lem:route_maps_exist}
If $n\ge5$, then every combinatorial embedding $C_n^\ext\to C_n^\ext$ has a route map. If $\Gamma$ is a graph with more than one vertex such that $\operatorname{Out}A_\Gamma$ is finite, then every automorphism of $\Gamma^\ext$ has a route map.
\end{lemma}

\begin{proof}
\cite[Lem.~3.11]{kimkoberda:embedability} shows that, for $n\ge5$, every $n$--cycle in $C_n^\ext$ has the form $g\cdot C_n$ for some $g\in A_{C_n}$. This gives the first statement. The second statement follows from \cite[Lem.~4.10, 4.12]{huang:quasiisometric:1}.
\end{proof}

Recall from \cite{laurence:generating,servatius:automorphisms} that the RAAG $A_\Gamma$ has finite outer automorphism group if and only if $\Gamma$ has no separating vertex star and no vertices $v,w$ with $\lk(w)\subset\star(v)$. 

\begin{lemma} \label{lem:route_map_close_to_qie}
Let $\varphi\colon A_\Gamma\to A_\Gamma$ be a stable quasiisometric embedding. Let $f\colon \Gamma^\ext\to \Gamma^\ext$ be the combinatorial embedding induced by $\varphi$. If $f$ has a route map $F$, then $F$ and $\varphi$ are at finite distance from one another. 
\end{lemma}

\begin{proof}
Let $S$ be the standard generating set of $A_\Gamma$.
Let $D$ be the stability  constant for $\varphi$, and let $g\in A_\Gamma$. For each standard generator $s\in S$, there is some $h_s\in A_\Gamma$ such that $\dist_{Haus}(\varphi(g\sgen s),h_s\sgen{t_s})\le D$, where $t_s=\type(f(g\cdot v_s))$. Note that since $\varphi$ has a route map, $\{t_s\mid s\in S\}=S$. Since $A_\Gamma$ is centreless, by \cref{lem:unique_route} the element $F(g)$ is the unique element in the intersection of the parallel sets of the $h_s\langle t_s\rangle$.
Hence $\varphi(g)$ is in the intersection of the $D$ neighbourhoods of the parallel sets of the $h_s\langle t_s\rangle$, which are the parallel sets of $F(g)\langle s\rangle$ for every $s\in S$.

We will show $\dist(\varphi(g),F(g))\leq |V(\Gamma)|\cdot D$. By left multiplication by $F(g)^{-1}$, we may assume that $F(g)=1$ and $\varphi(g)\in\bigcap_{v_s\in\Gamma}P_s^{+D}$, where $P_s$ is the parallel set of $\langle s\rangle$.

Choose a reduced word $s_1^{a_1}\cdots s_n^{a_n}$ representing $\varphi(g)$. The distance from $\varphi(g)$ to $P_s$ is at least $\sum_{s_i\notin C_{A_\Gamma}(s)} a_i$. Thus, as $A_\Gamma$ is centreless, for every $i$ there exists $s\in S$ that does not commute with $s_i$ and we get $\sum_{s_j=s_i} a_j\leq D$. In total, we must have $\dist(\varphi(g),1)=\sum_{i=1}^n a_i\leq |V(\Gamma)|\cdot D$.
\end{proof}

Even if a combinatorial embedding $f\colon \Gamma^\ext\to\Gamma^\ext$ is not induced by a quasiisometric embedding, it is often still the case that if it has a route map $F$, then $F$ is sufficiently well behaved that it induces the combinatorial embedding $f$. More precisely, we have the following.

\begin{lemma} \label{lem:route_maps_are_the_shit}
Suppose that no two vertices of $\Gamma$ have the same star. Let $f\colon \Gamma^{\ext}\to \Gamma^{\ext}$ be a combinatorial embedding with a route map $F$. For each $g\in A_\Gamma$ and each standard generator $s\in S$, there is some $m_k$ such that $F(gs^k)=F(g)\type(f(g\cdot v_s))^{m_k}$. The map $k\mapsto m_k$ is injective. In particular, $F$ induces $f$. 

If, furthermore, $f$ is not type-mixing and $F(1)=1$, then $F(g)$ has the same syllable structure as $g$, but perhaps with the powers changed, for every $g\in A_\Gamma$.
\end{lemma}

\begin{proof}
By definition, $f(g\cdot\Gamma)$ is the copy $F(g)\cdot\Gamma$ of $\Gamma$ inside $\Gamma^\ext$. Thus $f(g\cdot\star_\Gamma(v_s))=F(g)\cdot\star v_t$, where $f(g\cdot v_s)=F(g)\cdot v_t$. For every $k\ne0$, the intersection of $g\cdot\Gamma$ with $gs^k\cdot\Gamma$ is $g\cdot\star_\Gamma(v_s)$, and hence the intersection of $f(g\cdot\Gamma)$ with $f(gs^k\cdot\Gamma)$ is precisely $F(g)\cdot\star_\Gamma(v_t)$. Since $v_t$ is the unique vertex of $\Gamma$ whose star is $\star_\Gamma(w)$, it follows that $F(g)$ and $F(gs^k)$ differ by a suffix that is a power of $t$. The $m_k$ in the statement are pairwise distinct because no two $F(gs^k)\cdot\Gamma$ are equal, since $A_\Gamma$ is centreless.

We have shown that the image of the standard geodesic $g\sgen s$ under $F$ is the standard geodesic $F(g)\sgen{\type(f(g\cdot v_s))}$, and so $F$ induces $f$.

We prove the second statement by induction on $d$, the syllable-length of $g$. Let $g=s_1^{a_1}\dots s_d^{a_d}$ be a syllable-reduced word representing $g$. If $d=0$ then there is nothing to prove because $F(1)=1$, so suppose $d>0$. Let $h=s_1^{a_1}\dots s_{d-1}^{a_{d-1}}$. By induction, we know that $F(h)=t_1^{b_1}\dots t_{d-1}^{b_{d-1}}$ for some nonzero integers $b_1,\dots,b_{d-1}$. Since we started with a syllable-reduced expression for $g$, no reduced expression for $h$ can have $s_d^{\pm1}$ as a final letter. Since $f$ is not type-mixing, it follows that no reduced expression for $F(h)$ can have $\type(f(h\cdot v_{s_d}))^{\pm1}$ as a final letter. By the first part, we have $F(g)=F(h)t_d^{b_d}$ for some integer $b_d$, and $b_d$ is nonzero because of injectivity.
\end{proof}

Using \cref{lem:route_maps_are_the_shit}, we can see that in some situations, the existence of a route map prevents a combinatorial embedding $\Gamma^\ext\to\Gamma^\ext$ from being type-mixing.
The following definition was introduced by Huang in his work on quasiisometric rigidity \cite{huang:commensurability}. 

\begin{definition}[Star-rigid]
A graph $\Gamma$ is \emph{star-rigid} if for every $v\in\Gamma$, the only automorphism of $\Gamma$ that restricts to the identity on $\star(v)$ is the identity automorphism.
\end{definition}

Cyclic graphs are star-rigid. If $\Gamma$ is star-rigid, then no two vertices of $\Gamma$ have the same star, for transposing such vertices would be an automorphism.

\begin{proposition} \label{prop:no_exotic_map}
Suppose that $\Gamma$ is star-rigid. If a combinatorial embedding $f\colon \Gamma^\ext\to\Gamma^\ext$ has a route map $F$, then $f$ is not type-mixing.
\end{proposition}

\begin{proof}
We will show by induction on $d=|g|$ that $\type(f(g\cdot v_s))=\type(f(v_s))$ for every standard generator $s$. For $d=0$ there is nothing to prove. Assuming $d>0$, let $g=g't^{\pm 1}$ be a reduced expression for $g$. By \cref{lem:route_maps_are_the_shit}, there exists $m$ such that $F(g)=F(g')\type(f(g'\cdot v_t\cdot))^m$. Thus, the map $\Gamma\to\Gamma$ that sends the vertex of type $\type(f(g'\cdot v))$ to the vertex of type $\type(f(g\cdot v))$ for each $v$ is an automorphism of $\Gamma$ that restricts to the identity on $\star(v_t)$. Since $\Gamma$ is star-rigid, this implies that $\type(f(g\cdot v_s))=\type(f(g'\cdot v_s))$ for all $s\in S$. By induction, this is equal to $\type(f(v_s))$.
\end{proof}

\begin{corollary} \label{cor:power_change}
Let $\Gamma$ be a graph such that $A_\Gamma$ is centreless. If $\Gamma$ is star-rigid, then any route map $F\colon A_\Gamma\to A_\Gamma$ with $F(1)=1$ has the property that $F(g)$ has the same syllable structure as $g$, but perhaps with the powers changed.
\end{corollary}

\begin{proof}
Let $F$ be a route map of the combinatorial embedding $f\colon \Gamma^\ext\to \Gamma^\ext$. By \cref{prop:no_exotic_map}, if $\Gamma$ is star-rigid, then $f$ is not type-mixing. Star-rigidity implies that no two vertices have the same star, hence the conclusion holds by \cref{lem:route_maps_are_the_shit}. 
\end{proof}

\subsection{The structure of route maps}

Let $\Gamma$ be a star-rigid graph. Here we describe all the route maps of $A_\Gamma$. According to \cref{cor:power_change}, any map $A_\Gamma\to A_\Gamma$ that preserves the identity and is the route map of a combinatorial embedding $\Gamma^\ext\to\Gamma^\ext$ that is not type-mixing simply modifies the powers appearing in syllable-reduced expressions for elements of $A_\Gamma$. This motivates the following definition.

\begin{definition}[Change-of-power] \label{def:change-of-power}
We say that a map $F\colon A_\Gamma\to A_\Gamma$ is a \emph{change-of-power route map} if it satisfies the conclusion of \cref{cor:power_change}: it is the route map of some combinatorial embedding $\Gamma^\ext\to\Gamma^\ext$ and, for every $g\in A_\Gamma$, the element $F(g)$ has the same syllable structure as $g$, but perhaps with the powers changed. 
\end{definition}

Note that every change-of-power route map $F$ satisfies $F(1)=1$. The hypothesis that $F$ is a route map ensures that $F$ changes powers ``in an injective way''.

Given $\sigma\in\Aut\Gamma$, we denote by $\hat\sigma$ the automorphism of $A_\Gamma$ that, given an element $g=s_{v_1}^{a_1}\dots s_{v_k}^{a_k}\in A_\Gamma$, maps it to $\hat\sigma(g)=s_{\sigma(v_1)}^{a_1}\dots s_{\sigma(v_k)}^{a_k}$.

\begin{lemma} \label{lem:rooting_route_map}
Let $\Gamma$ be star-rigid, and let $F\colon A_\Gamma\to A_\Gamma$ be a route map. There exists $\sigma\in \operatorname{Aut}(\Gamma)$ such that $F'= L_{F(1)^{-1}} \circ F\circ\hat\sigma$ is a change-of-power route map. 
\end{lemma}

\begin{proof}
Clearly $F'(1)=1$. Let $f\colon \Gamma^\ext\to \Gamma^\ext$ be a combinatorial embedding for which $F$ is the route map. 
Since $F$ is a route map, there exists an automorphism $\sigma\in \Aut(\Gamma)$ such that $f(\sigma(v))=F(1)\cdot v$ for all $v\in\Gamma$. Since $\Gamma$ is star-rigid,  \cref{prop:no_exotic_map} tells us that $f$ is not type-mixing, so we have $f(g\cdot \sigma(v))=F(g)\cdot v$ for all $g\in A_{\Gamma}$ and all $v\in\Gamma$. 

Let $f':\Gamma^\ext\to \Gamma^\ext$ be the combinatorial embedding defined by $f'(g\cdot v_s)= (L_{F(1)^{-1}}\circ F\circ\hat\sigma(g))\cdot v_s$. Note that $F'=L_{F(1)^{-1}}\circ F\circ\hat\sigma$ is the route map of $f'$. By \cref{cor:power_change}, $F'$ is a change-of-power route map.
%
\end{proof}

Thus, in order to understand route maps of $A_{\Gamma}$ with $\Gamma$ star-rigid, we need only understand change-of-power route maps. 
Let $F\colon A_{\Gamma}\to A_{\Gamma}$ be a change-of-power route map.
By \cref{cor:power_change}, the map $F$ respects syllable-prefixes, in the sense that $F(g)$ is a prefix of $F(gg')$ whenever $g$ and $g'$ can be represented by syllable-reduced words whose concatenation is syllable-reduced. More precisely, for every $g\in A_{\Gamma}$ and $s\in S$ such that $\syl{gs}=\syl g +1$, we have an injective map $F_{g,s}\colon \mathbb{Z}\to \mathbb{Z}$, which is defined by setting $F_{g,s}(0)=0$ and $F_{g,s}(k)$ to be the integer for which $F(gs^k)=F(g)s^{F_{g,s}(k)}$ when $k\ne0$. We can thus recover $F$ from $\{F_{g,s}\mid s\in S,\, g\in A_{\Gamma}\}$.

\begin{definition}[Blocking set] \label{def:blocking}
Let $\Gamma$ be a graph. For a standard generator $s$ of $A_\Gamma$, the \emph{blocking set} of $s$ is the set $B_s$ consisting of all elements of $A_\Gamma$ that cannot be represented by a reduced word whose final letter commutes with $s$. 
\end{definition}

Note that $B_s$ is exactly the set of minimal labels of vertices of $\Gamma^\ext$ of type $s$.

In the remainder of the subsection, we use blocking sets to characterise change-of-power route maps. We will show that, in many situations, a change-of-power route map $F$ is determined by $\{F_{g,s}\mid s\in S,\,g\in B_s\}$; see \cref{lem:power_maps_determined_on_B}. Conversely, we will show that, for any $\Gamma$, if we make an arbitrary choice of injective map $\psi_{g,s}\colon \mathbb{Z}\to \mathbb{Z}$ for each $s\in S$ and each $g\in B_s$ such that $\psi_{g,s}(0)=0$, then this collection can arise as $\{F_{g,s}\}$ for some route map $F$; see \cref{lem:extending_power_maps}. 

\begin{proposition}
\label{lem:power_maps_determined_on_B}
Let $\Gamma$ be such that no two distinct vertices $w,v\in \Gamma$ have $\lk(w)\subseteq \str(v)$. 
Let $F$ be a change-of-power route map of $A_{\Gamma}$. The maps $\{F_{g,s}\mid s\in S,\,g\in B_s\}$ determine $F$. 
\end{proposition} 

\begin{proof}
We will recover $F_{g,s}$ for every $g\in A_{\Gamma}$ and every standard generator $s\in S$ by using induction on $d=\syl g$. For $d=0$ the maps are given for every $s$. 
    
Now assume $\syl g=d >0$ and let $s_2\in S$. If $g$ cannot end with a letter that commutes with $s_2$ then $g\in B_{s_2}$, so $F_{g,s_2}$ is given. Otherwise $g$ has a syllable-reduced expression $g=g's_1^{a_1}$, where $s_1$ commutes with $s_2$. We will show that $F_{g,s_2}=F_{g',s_2}$, which, by the induction hypothesis, is determined by the given data. 
    We will denote $b_1={F_{g',s_1}(a_1)}$, and for a nonzero integer $a_2$ we write $b_2={F_{g',s_2}(a_2)}$.

 See \cref{fig:determining comb embedding} for an illustration of the proof.

    Let $f\colon \Gamma^\ext\to\Gamma^\ext$ be the combinatorial embedding corresponding to $F$. Every generator $s\ne s_1$ that commutes with $s_1$ satisfies that 
    \[F(gs_2^{a_2})\cdot v_{s} \,=\, f(gs_2^{a_2}\cdot v_{s}) 
    \,=\, f(g's_2^{a_2}\cdot v_{s}) \,=\, F(g')s_2^{b_2}\cdot v_{s},\]
    and hence $F(gs_2^{a_2})=F(g')s_2^{b_2}h_0$, where $h_0$ commutes with all generators corresponding to $\lk(v_{s_1})$. This implies that there exists $m\neq 0$ such that $h_0=s_1^m$, because otherwise there exists a different generator, $s$ such that $\lk(v_{s_1})\subseteq \str(v_s)$, which cannot be by assumption on $\Gamma$.

    Similarly, every generator $s$ that commutes with $s_2$ satisfies that 
    \[F(gs_2^{a_2})\cdot v_{s} \,=\, f(gs_2^{a_2}\cdot v_{s})
    \,=\, f(g's_1^{a_1}\cdot v_{s}) \,=\, F(g')s_1^{b_1}\cdot v_{s},\]
    hence by similar reasoning there exists $n\neq 0$ such that $F(gs_2^{a_2})=F(g')s_1^{b_1}s_2^n$.

    We then conclude that $F(gs_2^{a_2})=F(g')s_1^{b_1}s_2^{b_2}$, which is equal to $F(g)s_2^{b_2}$. We have shown that $F_{g,s_2}(a_2)=F_{g',s_2}(a_2)$ for every $a_2\in\Z\ssm\{0\}$, as desired.
\end{proof}

\begin{figure}[ht]
\includegraphics{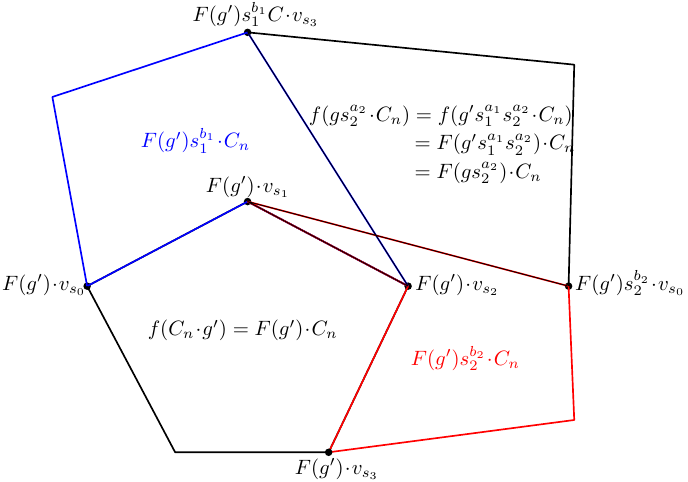}\centering 
\caption{The proof of \cref{lem:power_maps_determined_on_B} illustrated for $\Gamma=C_n$: the element $F(g)$ is determined by $F(g's_1^{a_1})$ and $F(g's_{2}^{a_2})$} 
\label{fig:determining comb embedding}
\end{figure}


Conversely, let $\Gamma$ be an arbitrary graph. For each $s\in S$ and each $g$ in the blocking set $B_s$, fix an injective map $\psi_{g,s}\colon \Z\to\Z$ with $\psi_{g,s}(0)=0$. We define a map $\hat\psi\colon A_\Gamma\to A_\Gamma$ inductively on syllable-length as follows. First set $\hat\psi(1)=1$. Now suppose that we have defined $\hat\psi$ on all elements of syllable-length less than $d>0$. Let $g\in A_\Gamma$ have $\syl g=d-1$, and suppose that $\syl {gs^a}=d$. If $g\in B_s$, then we define $\hat\psi(gs^a)=\hat\psi(g)s^{\psi_{g,s}(k)}$. Otherwise there is a syllable-reduced expression $g=g's_1^{a_1}$ where $s_1$ commutes with $s$. In this case, although we do not know that $g'\in B_s$, by induction there is some nonzero $b$ such that $\hat\psi(g's^a)=\hat\psi(g')s^b$, and we declare $\hat\psi(gs^a)=\hat\psi(g)s^b$.

\begin{proposition}\label{lem:extending_power_maps}
The map $\hat\psi$ is a well-defined change-of-power route map.
\end{proposition} 

\begin{proof}
    We first show that if $\gamma$ and $\gamma'$ are syllable-reduced words representing the same element in $A_\Gamma$, then $\hat\psi(\gamma)=\hat\psi(\gamma')$.
    Let $\gamma$ be syllable-reduced. By \cref{prop:shuffling}, any other syllable-reduced word representing it can be obtained by interchanging commuting adjacent syllables. Hence it is enough to show that if $\gamma=s_1^{a_1}\dots s_d^{a_d}$ and $\gamma'$ is obtained from $\gamma$ by interchanging commuting syllables $s_c^{a_c}$ and $s_{c+1}^{a_{c+1}}$, then $\hat\psi(\gamma)=\hat\psi(\gamma')$.
    
    We prove this by induction on $d$. For $d\le1$ there is nothing to prove. By induction we need only consider the case where $c=d-1$. Let $g'=s_1^{a_1}\dots s_{d-2}^{a_{d-2}}$
    and let $b_1,b_2$ be such that 
    \[
    \hat\psi(g's_{d-1}^{a_{d-1}}) \,=\, \hat\psi(g')s_{d-1}^{b_1} \quad \text{and} \quad 
    \hat\psi(g's_{d}^{a_{d}}) \,=\, \hat\psi(g')s_{d}^{b_2}.
    \]
    As $g's_{d-1}^{a_{d-1}}\notin B_{s_d}$ and $g's_{d}^{a_d}\notin B_{s_{d-1}}$, we by definition have that $\hat\psi(\gamma)=\hat\psi(g's_{d-1}^{a_{d-1}})s_d^{b_2}$ and $\hat\psi(\gamma')=\hat\psi(g's_{d}^{a_{d}})s_{d-1}^{b_1}$. By definition of $b_1$ and $b_2$, these are equal as elements of $A_{\Gamma}$. 
    
    

    Note that the map $\hat\psi$ simply changes powers in syllable-reduced expressions. To complete the proof, we show that it is the route map of some combinatorial embedding $\Gamma^\ext\to\Gamma^\ext$. By definition, $\hat\psi(g.\bar s)$ is a subset of $\hat\psi(g)\sgen s$ for every $g\in A_\Gamma$ and every $s\in S$, so $\hat\psi$ induces a map of vertex sets $f\colon \Gamma^{\ext}\to \Gamma^{\ext}$ by setting $f(g\cdot v)=\hat\psi(g)\cdot v$. This ensures that if $f$ is a combinatorial embedding, then $\hat\psi$ is the route map of $f$. 

Since $\hat\psi$ merely changes powers, if $v_1$ and $v_2$ are adjacent vertices of $\Gamma$, then $f(g\cdot v_1)$ and $f(g\cdot v_2)$ are adjacent or equal. Thus, to finish the proof we need only show that $f$ is injective. 


    Assume $f(g_1\cdot v_1)=f(g_2\cdot v_2)$ and assume that $g_1,g_2$ are minimal representatives, which in particular means that $g_isg_i^{-1}$ is reduced and syllable-reduced for each $i\in\{1,2\}$, where $s=\type(f(g_1\cdot v_1))$. From the definition of $f$ we see that $\hat\psi(g_1)\cdot v_1=\hat\psi(g_2)\cdot v_2$, so $v_1=v_2$ and this vertex has type $s$. 
    It follows that $\hat\psi(g_1)s\hat\psi(g_1)^{-1}= \hat\psi(g_2)s\hat\psi(g_2)^{-1}$. We note that these expressions are reduced and syllable reduced. By \cref{prop:shuffling}, we must have $\hat\psi(g_1)=\hat\psi(g_2)$. As the $\psi_{g,s}$ are injective, we see by induction that $g_1=g_2$, hence $g_1\cdot v_1=g_2\cdot v_2$.
\end{proof}
    
    




The following summarises what we know about route maps in the star-rigid case.

\begin{corollary} \label{lem:route_rooting_summary}
Let $\Gamma$ be a star-rigid graph such that no vertex link is a subset of the star of another vertex. Every route map is determined by an automorphism of $\Gamma$, a left multiplication map, and a collection of injections of $\Z$ defined on blocking sets of standard generators of $A_\Gamma$; and vice versa.
\end{corollary}

\begin{proof}
By \cref{lem:rooting_route_map}, every route map is determined by an automorphism of $\Gamma$, a left multiplication, and a change-of-power route map. By \cref{lem:power_maps_determined_on_B}, the latter is determined by a collection of injections of $\Z$, one for each pair $(g,s)$ with $g$ in the blocking set of $s$. The converse is given by \cref{lem:extending_power_maps}.
\end{proof}

\begin{proposition}\label{prop:finite_dist_from_homo}
    A change-of-power route map is at finite distance from a homomorphism if and only if it is a homomorphism defined by $s_i\mapsto s_i^{n_i}$ for $s_i\in S$ and $n_i\neq 0$.
\end{proposition}

\begin{proof}
It is clear that a homomorphism of the form $s_i\mapsto s_i^{n_i}$ is a change-of-power map.

For the other direction, let $F$ be a change-of-power map. Assume that $F$ is at distance at most $C$ from a homomorphism $f$.
Let $s$ be a standard generator. By the definition of $C$, for every $n$ there exists $g_{n}\in A_\Gamma$ satisfying $|g_{n}|\leq C$ and such that $f(s^n)=F(s^n)g_{n}=s^{F_{1,s}(n)}g_{n}$.

As $f$ is a homomorphism, we have $f(s^n)=f(s)^n$. In other words 
\[s^{F_{1,s}(n)}g_n=(s^{F_{1,s}(1)}g_1)^n.\]
In particular, the distance of $f(s)^n$ from $\sgen s$ is bounded by $|g_n|\leq C$. 

We will first show that $g_1=s^l$ for some $l\in \mathbb{Z}$.

Write $g_1=hwh^{-1}$ where $hwh^{-1}$ is reduced and $w\neq 1$ is cyclically reduced. We will show that $h=1$ and $w=s^l$ for some $l\in\mathbb{Z}$.

If $w$ contains a letter that is not $s^{\pm 1}$ then $d((s^{F_{1,s}(1)}g_1)^n, \sgen s)\geq n$.
Hence, $w=s^l$ for some $l\neq 0$. If $h\neq 1$, then $h$ does not commute with $s$. In particular, $h$ contains a letter that does not commute with $s$ and so does not get cancelled in the expression $(s^{F_{1,s}(1)}g_1)^n$. We then get again, 
$d((s^{F_{1,s}(1)}g_1)^n, \sgen s)\geq n$.
In conclusion, $g_1=s^{l_s}$. 

Consequently, the homomorphism $f$ satisfies $f(s)=s^{n_s}$ for some $n_s\in\Z$. We note that $n_s\neq 0$, since $F$ is assumed to be injective and $\dist(F(s^n),f(s^n))\le C$ for all $n$. Indeed, there are only finitely many elements of word length at most $C$. Applying the same argument for each standard generator, we deduce that $f$ has the form $t\mapsto t^{n_t}$ with $n_t\ne0$ for all standard generators $t$. In particular, $f$ is a change-of-power route map.

At this point, one could argue that $l_s=0$, but that is not necessary. Indeed, we now know that $f$ and $F$ induce the same combinatorial embedding of extension graphs, and hence $F$ is the route map of $f$. Since $f$ is a route map, this shows that $F=f$.
\end{proof}
\subsection{When is a route map a quasiisometric embedding?}

Here we characterise when a change-of-power route map is a quasiisometric embedding and when it is coarsely onto, in terms of conditions on the maps $F_{g,s}$ given by \cref{lem:power_maps_determined_on_B}.

\begin{definition}
For $C\ge\eps>0$, a map $\zeta\colon \Z\to\Z$ is said to be $(\eps,C)$--bilipschitz if for every $m,n\in\Z$ we have 
\[
\eps|n-m| \,\le\, |\zeta(n)-\zeta(m)| \,\le\, C|n-m|.
\]
A collection of maps $\zeta_i\colon \Z\to\Z$ is said to be \emph{uniformly bilipschitz} if there exist $C\ge\eps>0$ such that every $\zeta_i$ is $(\eps,C)$--bilipschitz. 

Let $F\colon A_\Gamma\to A_\Gamma$ be a change-of-power route map. We say that $F$ is \emph{power-bilipschitz} if its collection $\{F_{g,s}\,:\,g\in A_\Gamma,\,s\in S\}$ of associated maps is uniformly bilipschitz.
\end{definition}

Note that we are not requiring $(\eps,C)$--bilipschitz maps to be bijections of $\Z$, merely injections. It was shown in \cite{benjaminishamov:bilipschitz} that every bilipschitz bijection of $\Z$ is at finite $\ell^\infty$-distance from either the identity or a reflection.




\begin{proposition} \label{lem:power_map_is_qie}
A change-of-power route map of a RAAG $A_\Gamma$ is a quasiisometric embedding if and only if it is power-bilipschitz.
\end{proposition}

\begin{proof}
Let $F\colon A_\Gamma\to A_\Gamma$ be a change-of-power route map. First suppose that $F$ is power-bilipschitz, with constants $C\ge\eps>0$. We will show that $F$ is $C$--Lipschitz and $\eps$--colipschitz.

Let $gs$ be syllable-reduced expression, with $s\in S$ a standard generator. For every $a,b\in\Z$ we have
\[
\dist(F(gs^a),F(gs^b)) \,\le\, |F_{g,s}(a)-F_{g,s}(b)| \,\le\, C|a-b|,
\]
and it follows from the triangle inequality that $F$ is $C$--Lipschitz.

For the other inequality, let $g_1,g_2\in A_\Gamma$, and let $h$ be their maximal common prefix. Since $F$ is a change-of-power map, the maximal common prefix of $F(g_1)$ and $F(g_2)$ has the same syllable structure as $h$, but may not be equal to $F(h)$.

Let $\gamma=s_1^{a_1}\dots s_d^{a_d}$ be a syllable-reduced expression for $h$. Let $\gamma_i$ be a syllable-reduced expression for $g_i$ such that the first $d$ syllables of $\gamma_i$ use the same letters as $\gamma$ in the same order, though possibly with different exponents. Let $b_{ij}$ be the exponent on the $j^\mathrm{th}$ syllable of $\gamma_i$.

Let $c\le d$ be maximal such that the first $c$ syllables of both $\gamma_i$ are precisely $s_1^{a_1}\dots s_c^{a_c}$. Any syllable for which this is not the case must commute with every syllable that appears after it in $\gamma$. After changing the representative of $h$ by shuffling commuting syllables, and changing the $\gamma_i$ correspondingly, we may assume that if $k>c$, then there is at most one $i\in\{1,2\}$ such that $b_{ik}=a_k$. Since $h$ is the maximal common prefix of $g_1$ and $g_2$, there must be exactly one such $i$, and $s_k^{a_k}$ is a prefix of the $k^\mathrm{th}$ syllable of the other $\gamma_i$. 

For each $k\le d$, let $h_k$ be the prefix of $h$ represented by the word $s_1^{a_1}\dots s_k^{a_k}$. Since the $s_k$ with $k>c$ commute pairwise, by the proof of \cref{lem:power_maps_determined_on_B} we have that if $k_1,k_2>c$ then $F_{h_{k_1},s_{k_2}}=F_{h_c,s_{k_2}}$. From the fact that $F$ is power-bilipschitz, we have that if $k\in\{c+1,\dots,d\}$, then $|F_{h_c,s_k}(b_{1k})-F_{h_c,s_k}(b_{2k})|\ge\eps|b_{1k}-b_{2k}|$. 


Now let $\bar g_i\in A_\Gamma$ be the element represented by $s_1^{b_{i1}}\dots s_d^{b_{id}}$. Note that $h$ is also the maximal common prefix of $\bar g_1$ and $\bar g_2$. Additionally, $F(\bar g_i)$ is a prefix of $F(g_i)$, and $F(\bar g_1)$ and $F(\bar g_2)$ have the same maximal common prefix as $F(g_1)$ and $F(g_2)$. These latter two imply that
\[
\dist(F(g_1),F(g_2)) \,=\, \dist(F(g_1),F(\bar g_1))+\dist(F(\bar g_1),F(\bar g_2))+\dist(F(\bar g_2),F(g_2)).
\]
For $k>d$, the fact that $F$ is power-bilipschitz implies that the absolute value of the exponent of the $k^\mathrm{th}$ syllable of $F(g_i)$ is at least $\eps|b_{ik}|$. This shows that $\dist(F(g_i),F(\bar g_i))\ge\eps\dist(g_i,\bar g_i)$. 

We now combine these various estimates to show that $F$ is $\eps$--colipschitz. Indeed, 
\begin{align*}
\dist(F(g_1),F(g_2)) 
    \,&=\, \dist(F(g_1),F(\bar g_1)) + \dist(F(\bar g_1),F(\bar g_2)) + \dist(F(\bar g_2),F(g_2)) \\
    &\ge\, \eps\dist(g_1,\bar g_1) + \eps\dist(\bar g_2,g_2) 
        + \sum_{k=c+1}^d|F_{h_c,s_k}(b_{1k})-F_{h_c,s_k}(b_{2k})| \\
    &\ge\, \eps\dist(g_1,\bar g_1) + \eps\dist(\bar g_2,g_2) + \eps\sum_{k=c+1}^d|b_{1k}-b_{2k}| \\
    &=\, \eps\dist(g_1,g_2),
\end{align*}
as desired.

\medskip

For the other direction, suppose that the change-of-power route map $F$ is an $(A,K)$--quasiisometric embedding, with $A\ge1$. Let $C=A+K$ and let $\eps=\frac1{A(1+K)}$. For every $s\in S$, $g\in B_{s}$, and $a,b\in\mathbb{Z}$, we have 
\[
\frac1A|a-b|-K \,\le\, \dist(F(gs^a),F(gs^{b})) \,\le\, A|a-b|+K.
\] 
By definition, $F(gs^n)=F(g)s^{F_{g,s}(n)}$, so the right-hand side shows that every $F_{g,s}$ with $s\in S$ and $g\in B_s$ is $C$--Lipschitz. Moreover, route maps are injective, so the left-hand side implies that every such $F_{g,s}$ is $\eps$--colipschitz. We conclude from the proof of \cref{lem:power_maps_determined_on_B} that $F$ is power-bilipschitz, with constants $(\eps,C)$.
\end{proof}

Note that for any graph $\Gamma$, if $s\in S$ is a central element of $A_\Gamma$, then $B_s=\{1\}$. If $\Gamma$ is infinite then there may be infinitely many such $s$.

\begin{lemma} \label{lem:power_map_is_surjective}
Let $\Gamma$ be a graph, and let $F\colon A_\Gamma\to A_\Gamma$ be a change-of-power route map, with associated collection of maps $\{F_{g,s}\,:\,s\in S,\,g\in B_s\}$. The map $F$ is coarsely onto if and only if the following holds: if $s$ is non-central then $F_{g,s}$ is onto for every $g\in B_s$, and the maps $F_{1,s}$ with $s$ central are uniformly coarsely onto. If all $F_{g,s}$ are onto, then $F$ is onto.
\end{lemma}

\begin{proof}
First suppose that $F_{g,s}$ is onto for every non-central $s\in S$ and every $g\in B_s$. For such $s$ and $g$ there exists an inverse function $F_{g,s}^{-1}$. Let $\Lambda$ denote the subgraph of $\Gamma$ induced by vertices corresponding to such generators. Note that $A_\Gamma$ is the direct product of $A_\Lambda$ with a free-abelian group. By \cref{lem:extending_power_maps}, these inverse maps define a change-of-power route map $\hat\psi\colon A_\Lambda\to A_\Lambda$. By construction, $F\hat\psi(h)=h$ for all $h\in A_\Lambda$. Since $F$ is a change-of-power map, it follows that if the $F_{1,s}$ with $s\in S$ central are uniformly coarsely onto, then $F$ is coarsely onto, and if every such $F_{1,s}$ is onto, then $F$ is a bijection.
        
We now consider the converse. If the $F_{1,s}$ with $s$ central are not uniformly coarsely onto, then clearly $F$ is not coarsely onto. Suppose instead that there is a non-central $s\in S$ and some $g\in B_s$ for which $F_{g,s}$ is not onto. Let $n\in \mathbb{Z}\ssm\set{0}$ lie outside the image of $F_{g,s}$. Since $s$ is not central, there is some standard generator $t$ that does not commute with it. The map $F$ just changes powers, so it follows that if $h$ is any syllable-reduced word for which $t$ is the unique first letter, then $F(g)s^nh\notin F(A_\Gamma)$. In particular, $F(g)s^nt^k$ is at distance at least $k$ from $F(A_\Gamma)$ for each $k\in\N$. Thus, $F$ is not coarsely onto.
\end{proof}

\subsection{Rigidity results}

Here we combine the results of the previous subsections to prove the main results of the section. 

For a metric space $X$, let $\QIE(X)$ denote the \emph{quasiisometric embedding monoid} of $X$. That is, elements of $\QIE(X)$ are equivalence classes of quasiisometric embeddings $f\colon X\to X$, where $f_1$ and $f_2$ are equivalent if and only if they are at finite distance, in the sense that $\dist(f_1(x),f_2(x))$ is uniformly bounded for all $x\in X$. 

\begin{theorem} \label{thm:qie_of_cycle_to_itself}
If $n>4$, then every element of $\QIE(A_{C_n})$ is represented by a map of the form $F=L_g\circ\hat\psi\circ\hat\sigma$, for some $g\in A_{C_n}$, some $\sigma\in\Aut(C_n)$, and some power-bilipschitz change-of-power route map $\hat\psi$. In turn, $\hat\psi$ is determined by a set $\{F_{g,s}\,:\,s\in S,\,g\in B_s\}$ of uniformly bilipschitz maps of $\Z$ that preserve 0. The correspondence is one-to-one.

A quasiisometric embedding $F\in\QIE(A_{C_n})$ is a quasiisometry if and only if the maps $F_{g,s}$ are all onto.
%
\end{theorem}

\begin{proof}
Let $\varphi\colon A_{C_n}\to A_{C_n}$ be a quasiisometric embedding. By \cref{thm:squarefree_stable}, $\varphi$ is stable, and so by \cref{prop:stable_implies_induced} it induces a combinatorial embedding $f\colon C_n^\ext\to C_n^\ext$. 

By \cref{lem:route_maps_exist}, $f$ has a route map $F$, which is unique by \cref{lem:unique_route}, and $F$ is at finite distance from $\varphi$ by \cref{lem:route_map_close_to_qie}. As $F$ is a route map of $A_{C_n}$, the desired decomposition is given by \cref{lem:route_rooting_summary}, and the power-bilipschitz property of $F$ is given by \cref{lem:power_map_is_qie}. The correspondence is one-to-one because distinct maps with the form of $F$ are not at finite distance, by \cref{lem:route_map_close_to_qie}. The statement about when $F$ is a quasiisometry is given by \cref{lem:power_map_is_surjective}. 
%
%
        %
\end{proof}

Recall from \cite{laurence:generating,servatius:automorphisms}, that a RAAG $A_\Gamma$ has finite outer automorphism group if and only if $\Gamma$ has no separating vertex stars and there are no $v,w\in\Gamma$ with $\lk(w)\subset\star(v)$. In particular, $A_\Gamma$ is centreless.

\begin{theorem} \label{thm:qi_of_finite_out_to_itself}
Let $\Gamma$ be a finite, star-rigid graph. If $\Out A_\Gamma$ is finite and $A_\Gamma\ne\Z$, then every element of $\QI(A_\Gamma)$ is represented by a map of the form $F=L_g\circ\hat\psi\circ\hat\sigma$, for some $g\in A_\Gamma$, some $\sigma\in\Aut\Gamma$, and some power-bilipschitz change-of-power route map $\hat\psi$ whose associated maps of $\Z$ are all bijections preserving 0. The correspondence is one-to-one.
\end{theorem}

\begin{proof}
By \cite[Thm~3.20]{huang:quasiisometric:1}, if $\varphi$ is a self-quasiisometry of a RAAG whose defining graph has no $v,w$ with $\lk(w)\subset\star(v)$, then $\varphi$ is stable. Moreover, $\varphi$ induces an isomorphism of extension graphs by \cite[Lem.~4.5]{huang:quasiisometric:1} (this also follows from \cref{prop:stable_implies_induced} and the fact that $\varphi$ has a quasiinverse).

The proof now follows as in the second paragraph of \cref{thm:qie_of_cycle_to_itself}: note that $\Out A_\Gamma$ being finite implies that it satisfies the conditions of \cref{lem:route_rooting_summary} and \cref{lem:route_map_close_to_qie}.
\end{proof}

Based on the logic of this section and the proof of \cref{thm:qie_of_cycle_to_itself}, it is natural to ask the following.

\begin{question}
For which graphs $\Gamma$ does every combinatorial embedding $\Gamma^\ext\to\Gamma^\ext$ have a route map? For which RAAGs $A_\Gamma$ does every self--quasiisometric-embedding induce a combinatorial embedding $\Gamma^\ext\to\Gamma^\ext$ with a route map?
\end{question}

A natural candidate in the 2--dimensional case would be the \emph{atomic} graphs of \cite{bestvinakleinersageev:asymptotic}, which satisfy the necessary assumptions in \cref{thm:squarefree_stable} to induce a combinatorial embedding, and also meet the conditions of \cref{lem:route_maps_are_the_shit} and \cref{lem:power_maps_determined_on_B}. Note, though, that the analogue of \cite[Lem.~3.11]{kimkoberda:embedability} can fail for such graphs, as shown by the example in \cref{fig:counter_to_3.11}. It should also be observed that if $\Gamma$ is not star-rigid, there exists an isomorphism $\Gamma^{\ext}\to \Gamma^{\ext}$ that is type-mixing. 



\begin{figure}[ht]
\includegraphics{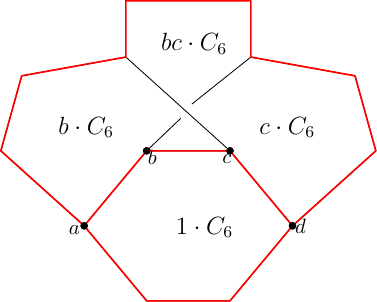}\centering
\caption{The graph $\Gamma$ obtained by gluing a 6-cycle to a 12-cycle along a $P_4$, embedded in $C_6^\ext\hookrightarrow \Gamma^\ext$. Thus $\Gamma$ fails \cite[Lem.~3.11]{kimkoberda:embedability}.} \label{fig:counter_to_3.11}
\end{figure}

\subsection{The quasiisometry group of a RAAG is big} \label{sec:qi_big}

In this last subsection we work with general RAAGs, which we do not require to be centreless. 

Recall that for a group $G$, the \emph{abstract commensurator} $\Comm(G)$ is the group of equivalence classes of isomorphisms between finite-index subgroups of $G$, where two such isomorphisms $\phi$ and $\psi$ are deemed equivalent whenever their domains have a common finite-index subgroup $H$ such that $\phi|_H=\psi|_H$. If $G$ is finitely generated, then $G$ has countably many finite-index subgroups, and since there are only countably many isomorphisms between two finitely generated groups, it follows that $\Comm(G)$ is countable.

\cite[Thm~1.10]{bestvinakleinersageev:asymptotic} shows that for an atomic RAAG $A_\Gamma$ the canonical maps $\mathrm{Aut}(A_\Gamma)\to \mathrm{Comm}(A_\Gamma)$ and $\mathrm{Comm}(A_\Gamma)\to \mathrm{QI}(A_\Gamma)$ are injective and have infinite-index images. Here we show the same for general RAAGs. Furthermore, we show that the image of the latter embedding has \emph{uncountable} index. 

\begin{lemma} \label{lem:aut_inj_in_comm}
    Let $\Gamma$ be a graph. The homomorphism $\mathrm{Aut}(A_\Gamma)\to \mathrm{Comm}(A_\Gamma)$ is injective.
\end{lemma}

\begin{proof}
Let $\alpha\in \ker(\mathrm{Aut}(A_\Gamma)\to\mathrm{Comm}(A_\Gamma))$. That is, $\alpha\in \mathrm{Aut}(A_\Gamma)$ restricts to the identity on some finite-index subgroup $H<A_\Gamma$. By passing to a further finite-index subgroup, we may assume $H$ is normal in $A_\Gamma$.

Let $g\in A_\Gamma$ and note that $g^{-1}\alpha(g)$ commutes with every element of $H$. Indeed, let $h\in H$. By assumption, $\alpha(h)=h$ and since $ghg^{-1}\in H$ also $\alpha(ghg^{-1})=ghg^{-1}$, hence 
\[g^{-1}\alpha(g)h=g^{-1}\alpha(g)h\alpha(g^{-1})\alpha(g)=g^{-1}\alpha(ghg^{-1})\alpha(g)=g^{-1}ghg^{-1}\alpha(g)=hg^{-1}\alpha(g),\]
and so $g^{-1}\alpha(g)$ commutes with $h$.

Since for every standard generator $s$ of $A_\Gamma$, there exists $k\neq 0$ such that $s^k\in H$, by \cite{servatius:automorphisms}, the centraliser of $H$ commutes with every standard generator and so it is exactly the centre of $A_\Gamma$. 

Thus $\varphi\colon A_\Gamma\to Z(A_\Gamma)$, defined by  $\varphi(g)= g^{-1}\alpha(g)$ is a homomorphism, since for every $g,h\in A_\Gamma$, we have 
\[
\varphi(gh^{-1}) = (gh^{-1})^{-1}\alpha(gh^{-1}) = hg^{-1}\alpha(g)\alpha(h^{-1}) = g^{-1}\alpha(g)h\alpha(h^{-1})=\varphi(g)\varphi(h^{-1})
\]
where in the penultimate inequality, we used the fact that $g^{-1}\alpha(g)$ is in the centre of $A_\Gamma$.

Since $\varphi|_H=1$, the image of $\varphi$ is finite. Since $Z(A_\Gamma)$ is torsion-free, $\varphi=1$ and so $\alpha=\operatorname{Id}$.
\end{proof}

The proof of the following is similar to the proof of \cref{prop:finite_dist_from_homo}, where we replace $F$ with $\operatorname{Id}$. The difficulty arises because $f$ is only defined on a finite-index subgroup.

\begin{lemma}\label{lem:comm_injective_in_qi}
    Let $\Gamma$ be a graph. The homomorphism $\mathrm{Comm}(A_\Gamma) \to \mathrm{QI}(A_\Gamma)$ is injective.
\end{lemma}

\begin{proof}
Let $f\in \ker(\mathrm{Comm}(A_\Gamma)\to \mathrm{QI}(A_\Gamma))$. Let $H_1,H_2<A_\Gamma$ be finite-index subgroups such that (a representative) $f\colon H_1\to H_2$ is an isomorphism.

By assumption, there exists $C>0$ such that $d_{A_\Gamma}(f(h),h)\leq C$ for every $h\in H_1$.

Let $h\in H_1$. By the definition of $C$, for every $n$ there exists $g_{n}\in A_\Gamma$ satisfying $|g_{n}|\leq C$ and such that $f(h^n)=h^ng_{n}$.

As $f$ is a homomorphism, we have $f(h^n)=f(h)^n$. In other words 
\[h^{n}g_n=(hg_1)^n.\]
In particular, the distance of $f(h)^n$ from $\sgen h$ is bounded by $|g_n|\leq C$. 

By \cite[Cor.~1.5]{carr:two}, the group $\sgen{h,g_1}$ is either free or free abelian, and it is undistorted. That is, there exists $q\geq 1$ and $A\geq 0$ such that $\sgen{h,g_1}$ is $(q,A)$-quasiisometrically embedded  and so in both cases $d((hg_1)^n,h^n)\geq qn-A$. For a big enough $n$, this leads to a contradiction.
In conclusion, $g_1=1$ and $f=\operatorname{Id}$.
\end{proof}

The proof of the next lemma is similar to proof of \cite[Sec.~11, Part~3]{bestvinakleinersageev:asymptotic}. We write it here for completeness, without relying on the notation of \cite{bestvinakleinersageev:asymptotic}. We also emphasise that we are not requiring the RAAG to be atomic. 

\begin{lemma}\label{lem:aut_infinite_index}
Let $\Gamma$ be a graph that is not a clique. The group $\mathrm{Aut}(A_\Gamma)$ has infinite index inside $\mathrm{Comm}(A_\Gamma)$.
\end{lemma}


\begin{proof}
Fix a non-central standard generator $s_0$. 
    Consider the map $\varphi\colon A_\Gamma\to \sgen{t\mid t^{n+1}}$ defined by $\varphi(s)=1$ if $s\neq s_0$ and $\varphi(s_0)=t$. Let $G_n=\ker\varphi$. The group $G_n$ is generated by $\{s_0^{n+1}\}\cup\{s_0^{k}ss_0^{-k}\,:\, s\in S,\,1\le k\leq n\}$. 
    
    Let $\psi_n:G_n\to G_n$ be the automorphism defined by permuting the generators of $G_n$ as follows: $\psi(s_0^{n+1})=s_0^{n+1}$ and $\psi(s_0^kss_0^{-k})=s_0^{k+1}ss_0^{-(k+1)}$ for $k=0,\dots, n$ where the powers are taken mod $n+1$.

    The map $\psi_n$ does not come from an automorphism of $A_\Gamma$, since that automorphism would be forced to map $s\mapsto s_0ss_0^{-1}$ and $s_0\mapsto s_0$, because right-angled Artin groups are biorderable \cite{duchampkrob:lower} and hence elements have unique roots. But then the automorphism would have to map $s_0^nss_0^{-n}\mapsto s_0^{n+1}ss_0^{-(n+1)}$, and that is not the case.

    Our goal is to show that for $m\neq n$, the maps $\psi_n$ and $\psi_m$ represent different cosets of $\mathrm{Aut}(A_\Gamma)$ in $\mathrm{Comm}(A_\Gamma)$. Let $m<n$. If $\psi_n$ and $\psi_m$ represent the same coset of $\Aut(A_\Gamma)$ inside $\Comm(A_\Gamma)$, then there is a finite-index subgroup $G<A_\Gamma$ and an automorphism $\phi\in\Aut(A_\Gamma)$ such that $(\psi_n\psi_m^{-1})|_G=\phi|_G$. Let $k=|A_\Gamma\colon G|$, so that $g^{k!}\in G$ for all $g\in A_\Gamma$. We have $\psi_n\psi_m^{-1}(s_0^{k!})=s_0^{k!}$, and hence $\phi(s_0)=s_0$ by the unique roots property. If $s\in S$ is a standard generator that does not commute with $s_0$, then
\[
\psi_n\psi_m^{-1}(s_0s^{k!}s_0^{-1}) \,=\, s_0s^{k!}s_0^{-1}.
\]
Using the unique roots property, this gives $\phi(s_0ss_0^{-1})=s_0ss_0^{-1}$. As $\phi$ is a homomorphism, we therefore have $\phi(s)=s$. But we can also compute
\[
\psi_n\psi_m^{-1}(s^{k!}) \,=\, \psi_n(s_0^ms^{k!}s_0^{-m}) \,=\, s_0^{m+1}s^{k!}s_0^{-(m+1)},
\]
which gives a contradiction.
\end{proof}

See \cref{def:blocking} for the definition of the blocking set $B_s$.

\begin{theorem} \label{thm:qi_uncountable}
If $\Gamma$ is not a complete graph, then $\QI(A_\Gamma)$ is uncountable. In particular, if $\Gamma$ is finite then $\Comm(A_\Gamma)$ has uncountable index in $\QI(A_\Gamma)$.
\end{theorem}

\begin{proof}
Let $s$ and $t$ be a pair of non-commuting standard generators of $A_\Gamma$. Note that $t^n\in B_s$ for all $n\in\Z$. Let $\sigma\colon \Z\to\Z$ be the transposition $(1,2)$. Given a subset $P\subset\N$, let $F^P$ be the change-of-power route map obtained from \cref{lem:extending_power_maps} by setting all maps $\psi_{g,s'}\colon \Z\to\Z$ to be the identity except $\psi_{t^n,s}=\sigma$ when $n\in P$. By \cref{lem:power_map_is_qie} and \cref{lem:power_map_is_surjective}, each of the uncountably many maps $F^P$ is a quasiisometry, and as they induce pairwise different route maps of $A_\Gamma/Z(A_\Gamma)$, by \cref{lem:route_map_close_to_qie} no two lie at finite distance from one other.

When $\Gamma$ is finite, $\mathrm{Comm}(A_\Gamma)$ is countable, and so it has uncountable index in $\mathrm{QI}(A_\Gamma)$.
\end{proof}

Note that all the quasiisometries $F^P$ produced in the above proof are surjective and have quasiisometry constant 2.

\begin{remark}
One could alternatively prove \cref{thm:qi_uncountable} by noting that there are uncountably many $(\frac13,3)$--bilipschitz bijections of $\Z$ that preserve zero. Indeed, given $P\subset2\N$, one can consider the map $\psi^P$ that transposes $n$ and $n+1$ if $n\in P$ and fixes all other $n$. Taking $F^P_{1,s}=\psi^P$ and all other $F^P_{g,s'}$ to be the identity again gives uncountably many quasiisometries $F^P$.
\end{remark}

\section{Beyond directionally branch-complemented}\label{sec:beyond_FB}

This section gives an extension of \Cref{thm:fully_branching_raag_rephrased}, which is about quasiisometric embeddings of directionally branch-complemented cliques, to cliques that are not directionally branch-complemented. 

More precisely, we will give a weaker branching condition for top-dimensional standard flats, or, equivalently, cliques in the defining graph. The condition is meant to capture cliques that are ``attached to the directionally branch-complemented part of $\Gamma$''; see the blue cliques in \Cref{fig:leaf_flats_3D}. In particular, it includes all edges containing leaves in connected, triangle-free graphs that are not stars. We show in \cref{thm:weakly_FB_cliques} that top-dimensional flats satisfying the condition are mapped Hausdorff-close to flats.
In the triangle-free connected case, when $\Gamma$ is not a star, we upgrade the result to \emph{all} 2--flats in \cref{thm:qi_embedding_every_2flats_2D_RAAGs}.

\begin{definition}[Weak branching]\label{def:weakly_fully_branching_clique}
Let $\Gamma$ be a graph with clique number $n \geq 2$, and let $K \subset \Gamma$ be an $n$--clique. We say that $K$ is \emph{weakly directionally branch-complemented} if
it has vertices $v_1,\dots,v_{n-1}$ such that each $v_i$ lies in some directionally branch-complemented $n$--clique.
\end{definition}

\begin{figure}[htbp]
    \centering
    \begin{tikzpicture}[baseline=-.1cm, scale=0.9]

\draw[red!30, line width=3pt]
(0,0) -- (1,0) -- (2,0) -- (3,0)
(0,-1) -- (0,0);

\draw[blue!30, line width=3pt]
(-1,0) -- (0,0)
(3,0) -- (4,0)
(0,-1) -- (0,-2)
(2,0) -- (2,1);

\draw[thick]
(-1,0) -- (4,0)
(0,-2) -- (0,0)
(2,0) -- (2,1);

\fill (-1,0) circle(.08);
\fill[red] (0,0) circle(.08);
\fill[red] (1,0) circle(.08);
\fill[red] (2,0) circle(.08);
\fill[red] (3,0) circle(.08);
\fill (4,0) circle(.08);
\fill[red] (0,-1) circle(.08);
\fill (0,-2) circle(.08);
\fill (2,1) circle(.08);
\end{tikzpicture}
\qquad
\begin{tikzpicture}[baseline=-.1cm, scale=1.15]

\fill[fill=red!30]
(-.4,.35) -- (0,-.35) -- (.8,-.35) -- (.4,.35) -- cycle;

\fill[fill=blue!30]
(-.8,-.35) -- (-.4,.35) -- (0,-.35) -- cycle;
\fill[fill=blue!30]
(0,-.35) -- (.4,.35) -- (.8,-.35) -- cycle;
\fill[fill=blue!30]
(-.4,.35) -- (0,1.05) -- (.4,.35) -- cycle;

\draw[thick]
(-.8,-.35) -- (.8,-.35)
(-.4,.35) -- (.4,.35);

\draw[thick]
(-.8,-.35) to (-.4,.35) to (0,-.35) to (.4,.35) to (.8,-.35)
(-.4,.35) to (0,1.05) to (.4,.35);

\fill (-.8,-.35) circle(.08);
\fill[red] (-.4,.35) circle(.08);
\fill[red] (0,-.35) circle(.08);
\fill[red] (.4,.35) circle(.08);
\fill (.8,-.35) circle(.08);
\fill (0,1.05) circle(.08);

\end{tikzpicture}
\qquad
\begin{tikzpicture}[baseline=-.1cm, scale=1.15]
\fill[fill=red!30]
(-.4,-.35) -- (0,.35) -- (.4,-.35) -- cycle;

\fill[fill=blue!30]
(-.8,.35) -- (-.4,-.35) -- (0,.35) -- cycle;
\fill[fill=blue!30]
(0,.35) -- (.4,-.35) -- (.8,.35) -- cycle;

\draw[thick]
(-.8,.35) -- (.8,.35)
(-1.2,-.35) -- (1.2,-.35);

\draw[thick]
(-1.2,-.35) to (-.8,.35) to (-.4,-.35) to (0,.35) to (.4,-.35) to
(.8,.35) to (1.2,-.35);

\fill (-1.2,-.35) circle(.08);
\fill (-.8,.35) circle(.08);
\fill[red] (-.4,-.35) circle(.08);
\fill[red] (0,.35) circle(.08);
\fill[red] (.4,-.35) circle(.08);
\fill (.8,.35) circle(.08);
\fill (1.2,-.35) circle(.08);
\end{tikzpicture}
\qquad
\begin{tikzpicture}[baseline=-.1cm, scale=1.15]
\fill[fill=red!30]
(-.4,-.35) -- (0,.35) -- (.4,-.35) -- cycle;

\fill[fill=blue!30]
(-.8,.35) -- (-.4,-.35) -- (0,.35) -- cycle;
\fill[fill=blue!30]
(0,.35) -- (.4,-.35) -- (.8,.35) -- cycle;

\draw[thick]
(-.8,.35) -- (.8,.35)
(-1.2,-.35) -- (1.2,-.35);

\draw[thick]
(-1.2,-.35) to (-.8,.35) to (-.4,-.35) to (0,.35) to (.4,-.35) to (.8,.35)
(1.2,-.35) to (.4,-.35);

\fill (-1.2,-.35) circle(.08);
\fill (-.8,.35) circle(.08);
\fill[red] (-.4,-.35) circle(.08);
\fill[red] (0,.35) circle(.08);
\fill[red] (.4,-.35) circle(.08);
\fill (.8,.35) circle(.08);
\fill (1.2,-.35) circle(.08);
\end{tikzpicture}
\par\vspace{-.6cm}
\begin{tikzpicture}[baseline=-.1cm, scale=1.6]

\coordinate (t1) at (-.85,1);
\coordinate (b1) at (-.85,-1);

\coordinate (v1) at (-1.35,.08);
\coordinate (v2) at (-1.10,-.28);
\coordinate (v3) at (-.60,-.28);
\coordinate (v4) at (-.35,.08);
\coordinate (v5) at (-.85,.35);

\coordinate (a1) at (.55,-.30);
\coordinate (a2) at (1.35,-.30);
\coordinate (a3) at (1.75,.30);
\coordinate (a4) at (.95,.30);
\coordinate (t2) at (1.15,1.00);
\coordinate (b2) at (1.15,-1.00);

\coordinate (u) at (.43,.67);
\coordinate (w) at (2.06,1.17);

\draw[thick]
(v1) -- (v2) -- (v3) -- (v4);
\draw[thick,dotted]
(v4) -- (v5) -- (v1);

\draw[thick]
(t1) -- (v1) (t1) -- (v2) (t1) -- (v3) (t1) -- (v4)
(b1) -- (v1) (b1) -- (v2) (b1) -- (v3) (b1) -- (v4);
\draw[thick,dotted]
(t1) -- (v5)
(b1) -- (v5);

\draw[thick]
(a1) -- (a2) -- (a3);
\draw[thick,dotted]
(a1) -- (a4) -- (a3);
\draw[thick]
(a1) -- (t2)
(a2) -- (t2)
(a3) -- (t2)
(a1) -- (b2)
(a2) -- (b2)
(a3) -- (b2);
\draw[thick,dotted]
(a4) -- (t2)
(a4) -- (b2);

\fill[blue!30] (v4) -- (a1) -- (u) -- cycle;
\draw[thick] (v4) -- (a1) -- (u) -- cycle;

\fill[blue!30] (a3) -- (t2) -- (w) -- cycle;
\draw[thick] (a3) -- (t2) -- (w) -- cycle;

\fill[red] (v1) circle(.06);
\fill[red] (v2) circle(.06);
\fill[red] (v3) circle(.06);
\fill[red] (v4) circle(.06);
\fill[red] (v5) circle(.06);
\fill[red] (t1) circle(.06);
\fill[red] (b1) circle(.06);

\fill[red] (a1) circle(.06);
\fill[red] (a2) circle(.06);
\fill[red] (a3) circle(.06);
\fill[red] (a4) circle(.06);
\fill[red] (t2) circle(.06);
\fill[red] (b2) circle(.06);

\fill (u) circle(.06);
\fill (w) circle(.06);

\end{tikzpicture}
    \caption{Weakly directionally branch-complemented cliques (in blue) when $n=2,3$. \\ Red vertices are branch-complemented.}
    \label{fig:leaf_flats_3D}
\end{figure}


Recall that a graph is a \emph{star} if it is connected and it has a vertex that lies in every edge.

\begin{observation}\label{rem:weakly_FB_cliques_in_2D}
Let $\Gamma$ be a connected, triangle-free graph that is not a star.
\begin{itemize}
\item   Every non-leaf vertex $v\in\Gamma$ is directionally branch-complemented. Indeed, otherwise all neighbours of $v$ would be leaves, so $\Gamma$ would be a star.
\item Every edge of $\Gamma$ is either directionally branch-complemented or weakly directionally branch-complemented. Indeed, let $e$ be an edge of $\Gamma$. If $e$ is not directionally branch-complemented, then its two vertices cannot both have degree $1$, since $\Gamma$ is connected and not a star. Hence one vertex $v$ of $e$ has degree at least $2$. Again, the same argument implies that one neighbour of $v$ must also have degree at least $2$. Therefore, $e$ intersects a directionally branch-complemented edge at $v$, so $e$ is weakly directionally branch-complemented.
\end{itemize}
\end{observation}

Before moving to the main theorem of this section, we state the following specialisation of \cite[Thm~7.8]{baderbensaidpetyt:from}, which is the main ingredient in the proof. We refer to \cite[\S10]{drutukapovich:geometric} for background on asymptotic cones. Recall that the asymptotic rank of a CAT(0) cube complex is the supremal $n$ such that it contains isometrically embedded boxes $[0,m]^n$ for all $m$. 

\begin{theorem}\label{thm:top_quasiflats_union_parallel_codim1_flats_paper1}
Let $Y$ be a finite-dimensional CAT(0) cube complex with asymptotic rank $n$, and set $E=\R^n$. Suppose that $f\colon  E \to Y$ is a quasiisometric embedding. Let $E_\omega$ be an asymptotic cone of $E$ with respect to some fixed basepoint, and let $f_\omega\colon  E_\omega \to Y_\omega$ be the induced bilipschitz embedding. 

If $H\subset E$ is an $(n-1)$--flat such that $f_\omega(H_\omega')$ is an $(n-1)$--flat in $Y_\omega$ for every $H_\omega' \subseteq E_\omega$ parallel to $H_\omega$, then $f(E)$ lies at finite Hausdorff distance from some $n$--flat in $Y$.
\end{theorem}

We shall use this theorem to control the images of weakly directionally branch-complemented flats. For the purposes of the proof, we refer to \cite[Def.~9.5]{baderbensaidpetyt:from} for the more precise notion of \emph{directionally 1--branch-complemented}.

\begin{theorem}\label{thm:weakly_FB_cliques}
Let $\Gamma$ be a graph with clique number $n \geq 2$ and let $Y$ be a finite-dimensional CAT(0) cube complex with asymptotic rank $n$. Let $f\colon  X_\Gamma \to Y$ be a quasiisometric embedding. 

If $K \subset \Gamma$ is a weakly directionally branch-complemented $n$--clique, then for every standard $n$--flat $F \subset X_\Gamma$ associated with $K$, the image $f(F)$ lies at finite Hausdorff distance from some $n$--flat in $Y$.
\end{theorem}

\begin{proof}
Let $K'\subset K$ be an $(n-1)$--subclique each of whose vertices lie in some directionally branch-complemented $n$--clique. By the version of \cref{thm:fully_branching_raag_rephrased} discussed in \cref{rem:rephrased_asrk}, there is a constant $D=D(n,q,Y)$ such that if $H\subset F$ is a flat associated with $K'$, then $f(H)$ lies at Hausdorff distance at most $D$ from an $(n-1)$--flat in $Y$. It follows that $f_\omega(H'_\omega)$ is an $(n-1)$--flat for every $H'_\omega\subset F_\omega$ parallel to $H_\omega$. The conclusion now follows from \cref{thm:top_quasiflats_union_parallel_codim1_flats_paper1}. \qedhere


\end{proof}

It follows from \Cref{rem:weakly_FB_cliques_in_2D} and \Cref{thm:weakly_FB_cliques} that if $\Gamma$ is a connected, triangle-free, non-star graph, if $Y$ is a CAT(0) square complex, and if $f\colon  X_\Gamma \to Y$ is a quasiisometric embedding, then for every standard $2$--flat $F \subset X_\Gamma$, the image $f(F)$ lies at finite Hausdorff distance from some $2$--flat in $Y$. In fact, by combining \cref{thm:weakly_FB_cliques} with \cref{thm:fully_branching_raag_to_sing_boundary} we obtain the following, which can be viewed as generalising \cref{cor:squarefree_stable}. For a graph $\Gamma$ with clique number $n$ and an induced subgraph $\Lambda\subset\Gamma$, say that a subset $V\subset\Gamma\ssm\Lambda$ is \emph{peripheral with respect to $\Lambda$} if no two elements of $V$ are adjacent and for each $v\in V$, the neighbours of $v$ in $\Lambda$ are exactly a single $(n-1)$--clique.


\begin{corollary} \label{cor:peripheral_boundary_map}
    Let $\Gamma=\Lambda\sqcup V$ be a graph with clique number $n \geq 2$, where $\Lambda$ is a directionally branch-complemented subgraph and $V$ is peripheral with respect to $\Lambda$. Let $Y$ be an $n$--dimensional CAT(0) cube complex. If $f : X_\Gamma \to Y$ is a quasiisometric embedding, then there is a combinatorial embedding $\Gamma^\ext\to\partial_{\mathrm{sing}}Y$.
\end{corollary}

\begin{proof}
The same argument as in the proof of \cref{thm:fully_branching_raag_to_sing_boundary} gives a combinatorial embedding $\phi$ of the induced subgraph of $\Gamma^\ext$ spanned by vertices coming from $\Lambda$ into $\partial_{\mathrm{sing}}Y$. Every other vertex of $\Gamma^\ext$ comes from $V$. 

In particular, for each vertex $u \in \Gamma^\ext$ not coming from $V$, \cref{thm:fully_branching_raag_to_sing_boundary} provides a singular geodesic $\gamma_u' \subset Y$ that is at finite Hausdorff distance from the image of the standard geodesic representing $u$. Moreover, for each standard $n$--flat $F\subset X_\Gamma$ associated with an $n$--clique $K\cup\{v\}$ of $\Gamma^\ext$, where $v$ comes from $V$, \cref{thm:weakly_FB_cliques} provides an $n$--flat $F'\subset Y$ at finite Hausdorff distance from $f(F)$. Since $Y$ is $n$--dimensional, $F'$ is necessarily singular, by \cite[Lem.~3.2]{baderbensaidpetyt:from}. The flat $F'$ contains a singular $(n-1)$--flat that is spanned by parallels of the $\gamma_u'$ for $u\in K$. We assign to the vertex $v$ a singular geodesic $\gamma_v' \subseteq F'$, orthogonal to all $\gamma_u'$, and define $\phi(v)$ to be either one of its endpoints. It is immediate that $\phi(u)$ and $\phi(l)$ are adjacent in $\partial_{\mathrm{sing}}Y$ for all $u\in K$. 

For every maximal clique of $\Gamma^\ext$ not equal to $K\cup\{v\}$, every neighbourhood of the corresponding flat $E$ in $X_\Gamma$ intersects $F$ in a neighbourhood of a (possibly trivial) standard subflat of a standard flat corresponding to $K$. Hence for each singular flat $E'\subset Y$ at finite Hausdorff distance from $f(E)$, the coarse intersection of $E'$ with $F'$ is contained in a neighbourhood of a singular subflat of $F'$ corresponding to $K$. By definition of $\gamma'_v$, this means that $E'$ has bounded coarse intersection with $\gamma'_v$. It follows that if $w\notin K\cup\{v\}$, then $\gamma'_w$ cannot lie at finite Hausdorff distance from $\gamma'_v\subset F'$. This shows that there is no $w\in\Gamma\ssm\{v\}$ with $\phi(w)=\phi(v)$. Since $v\in V$ was arbitrary and $\phi|_{\Lambda}$ is injective, we conclude that $\phi$ is injective. \qedhere


\end{proof}

Corollary~\ref{cor:peripheral_boundary_map} applies in particular to the final example in Figure~\ref{fig:leaf_flats_3D}. In the two-dimensional case, it reduces to the following.

\begin{corollary} \label{cor:2dim_boundary_map}
Let $\Gamma$ be a triangle-free simplicial graph such that no connected component of $\Gamma$ is a star. Let $Y$ be a 2--dimensional CAT(0) cube complex. If there is a quasiisometric embedding $f:A_\Gamma\to Y$, then there is a combinatorial embedding $\Gamma^\ext\to\partial_{\mathrm{sing}}Y$.
\end{corollary}

In the particular setting of \cref{cor:2dim_boundary_map}, we can actually show that the conclusion holds for \emph{every} $2$--flat in $X_\Gamma$. This generalises \cite[Thm~1.1]{bestvinakleinersageev:asymptotic}.

\begin{theorem}\label{thm:qi_embedding_every_2flats_2D_RAAGs}
Let $\Gamma$ be a triangle-free simplicial graph such that no connected component is a star. Let $Y$ be a finite-dimensional CAT(0) cube complex of asymptotic rank two. If $f\colon  X_\Gamma \to Y$ is a quasiisometric embedding, then for every $2$--flat $F \subset X_\Gamma$, the image $f(F)$ lies at finite Hausdorff distance from some $2$--flat in $Y$.
\end{theorem}

The proof will follow from the next proposition. We refer to \cite[Defs.~9.3,9.5]{baderbensaidpetyt:from} for the metric definitions of branch-complemented and directionally branch-complemented flats.

\begin{proposition}\label{prop:2flats_in_2D_RAAGs}
If $\Gamma$ is a triangle-free graph such that no connected component is a star, then every $2$--flat $F \subset X_\Gamma$ satisfies one of the following:
\begin{itemize}
    \item $F$ is directionally branch-complemented, or
    \item $F$ contains a branch-complemented standard geodesic.
\end{itemize}
\end{proposition}

\begin{proof}
Let $F$ be 2--flat. 
Since it is top-dimensional, it is a subcomplex. Let $\alpha$ and $\beta$ be two orthogonal geodesics contained in the $1$--skeleton of $F$. Let $L_\alpha$ and $L_\beta$ denote the sets of labels of edges of $\alpha$ and $\beta$, respectively, that is, the corresponding sets of generators of $A_\Gamma$. First, note that if $\alpha' \subset F$ is a geodesic parallel to $\alpha$ and contained in the $1$--skeleton of $X_\Gamma$, then $\alpha$ and $\alpha'$ have the same set of labels of edges by \Cref{lem:parallel_sing_geod_same_labels}. The same holds for the parallels of $\beta$. Moreover, since $\alpha$ and $\beta$ span the $2$--flat $F$, every label in $L_\alpha$ commutes with every label in $L_\beta$ and the sets are disjoint.

If $\alpha$ is standard, then either $\beta$ is standard or $|L_\beta|\geq 2$. If $\beta$ is standard, then by the assumption on $\Gamma$, as observed in \cref{rem:weakly_FB_cliques_in_2D}, $F$ satisfies the conclusion of the proposition. If $ |L_\beta|\geq 2$, then $\alpha$ corresponds to a non-leaf vertex, hence it is branching. Since $\Gamma$ has no star components, $\alpha$ is branch-complemented, as observed in  \cref{rem:weakly_FB_cliques_in_2D}, and the proposition follows. 

Thus, by symmetry, we are left with the case that both $\alpha$ and $\beta$ are not standard.
That means that there exists distinct elements $a_1,a_2\in L_\alpha$ and distinct elements $b_1,b_2\in L_\beta$. All four elements are pairwise distinct as $L_\alpha$ and $L_\beta$ are disjoint.

By left multiplication by an element of $A_\Gamma$, we can assume that $\alpha$ and $\beta$ intersect at $1_{A_\Gamma}$. 
As $b_1$ commutes with every element of $L_\alpha$, the geodesics $\{b_1^k\alpha\}_{k\in\mathbb{Z}}$ are parallel, and hence, form a 2--flat $F_1$. Similarly  the geodesics $\{b_2^k\alpha\}_{k\in\mathbb{Z}}$ form a 2--flat, $F_2$. We have that $F_1\cap F_2=\alpha$, hence $\alpha$ is branching. 
Similarly, using $\bigcup_{k\in\Z}a_1^k\beta$ and $\bigcup_{k\in\Z}a_2^k\beta$ we get that $\beta$ is branching. 

We have shown that both $\alpha$ and $\beta$ are branching. By definition, they are both branch-complemented and hence $F$ is directionally branch-complemented.
\end{proof}

\begin{proof}[Proof of \Cref{thm:qi_embedding_every_2flats_2D_RAAGs}]
By \Cref{prop:2flats_in_2D_RAAGs}, $F$ is either directionally branch-complemented or contains a branch-complemented standard geodesic. If $F$ is directionally branch-complemented, then the conclusion follows from \cite[Thm~10.1]{baderbensaidpetyt:from}. 

Suppose instead that $F$ contains a branch-complemented standard geodesic $\gamma$. By the version of \cref{thm:fully_branching_raag_rephrased} described in \cref{rem:rephrased_asrk}, $f(\gamma')$ lies at Hausdorff distance at most $D=D(q,Y)$ from a geodesic in $Y$ for every parallel $\gamma'$ of $\gamma$ inside $F$. It follows that $f_\omega(\gamma'_\omega)$ is a geodesic for every parallel $\gamma'_\omega$ of $\gamma_\omega$ in $F_\omega$. The conclusion follows from \cref{thm:top_quasiflats_union_parallel_codim1_flats_paper1}.
\end{proof}

\begin{remark} \label{rem:so_star_neccesary}
The condition that the connected components are not stars is necessary: if $\Gamma=P_3$ is the path on three vertices, then $X_\Gamma \cong T \times \mathbb R$, and there exists a quasiisometric embedding $T \times \mathbb R \to T \times T$ for which some flat $L \times \mathbb R$ is not sent at finite Hausdorff distance from any $2$--flat in $T \times T$, see \cite[\S2]{bowditch:quasiisometric}.    
\end{remark}

\bibliographystyle{alpha}
\footnotesize{\bibliography{biblio}}
\Addresses
\end{document}

%% file: preamble.tex
\usepackage[leqno]{amsmath}
\usepackage{amssymb,xfrac}
\usepackage{amsthm}
\usepackage{ifthen} 
\usepackage[pagebackref,breaklinks,unicode]{hyperref} 
    \renewcommand*{\backrefalt}[4]{\ifcase #1 (Not cited).\or (Cited p.~#2).\else (Cited pp.~#2).\fi} 
\usepackage[capitalise,noabbrev]{cleveref}
\usepackage{enumitem}
\usepackage{graphicx}
\usepackage{tikz-cd, tikz}
\usepackage{float}
\usepackage[justification=centering]{caption}
\usepackage{subcaption}
\usepackage{wasysym}

\numberwithin{equation}{section}

\newtheorem{theorem}{Theorem}[section]
\newtheorem{proposition}[theorem]{Proposition}
\newtheorem{lemma}[theorem]{Lemma}
\newtheorem{corollary}[theorem]{Corollary}

\newtheorem*{theorem*}{Theorem}
\newtheorem*{proposition*}{Proposition}
\newtheorem*{lemma*}{Lemma}
\newtheorem*{corollary*}{Corollary}

\newtheorem{mthm}{Theorem} 
\newtheorem{mcor}[mthm]{Corollary}

\theoremstyle{definition}
\newtheorem{definition}[theorem]{Definition}
\newtheorem{observation}[theorem]{Observation}
\newtheorem{remark}[theorem]{Remark}
\newtheorem{example}[theorem]{Example}
\newtheorem{question}[theorem]{Question}

\newtheorem*{definition*}{Definition}
\newtheorem*{observation*}{Observation}
\newtheorem*{remark*}{Remark}
\newtheorem*{example*}{Example}
\newtheorem*{question*}{Question}
\newtheorem*{exercise*}{Exercise}
\newtheorem*{fact*}{Fact}
\newtheorem*{notation*}{Notation}

\newcommand{\Z}{\mathbb{Z}}

\DeclareMathOperator{\QI}{QI}
\DeclareMathOperator{\QIE}{QIE}

\DeclareMathOperator{\Aut}{Aut}

\DeclareMathOperator{\Out}{Out}

\DeclareMathOperator{\Comm}{Comm}

\DeclareMathOperator{\lk}{Link}

\DeclareMathOperator{\str}{Star}

\DeclareMathOperator{\starr}{Star}
\renewcommand*{\star}{\starr}

\newcommand{\set}[1]{\left\{ #1 \right\}}

\newcounter{claimref}
\newcounter{claimcount}

\newenvironment{claim*}{\par\medskip\noindent\textbf{Claim:}\hspace{0.5mm}}{}

\newenvironment{claim*proof}{\medskip\noindent\emph{Proof of Claim.}\hspace{0.5mm}}
    {\leavevmode\unskip\penalty9999\hbox{}\nobreak\hfill\quad\hbox{$\diamondsuit$}\medskip}

\newcounter{subclaimcount}

\newenvironment{subclaim*}{\par\medskip\textbf{Subclaim:}\hspace{0.5mm}}{}
\newenvironment{subclaimproof*}{\medskip\noindent\emph{Proof of Subclaim.}\hspace{0.5mm}}
    {\leavevmode\unskip\penalty9999\hbox{}\nobreak\hfill\quad\hbox{$\fullmoon$}\medskip}

\crefname{cond}{condition}{conditions}
\creflabelformat{cond}{#2#1\@#3}
\crefname{obs}{observation}{observations}
\creflabelformat{obs}{#2#1\@#3}

\newcommand*{\eps}{\varepsilon}

\newcommand*{\R}{\mathbb{R}}

\newcommand*{\dist}{d}

\DeclareMathOperator{\type}{Typ}
\newcommand*{\sgen}[1]{\langle#1\rangle}
\newcommand*{\syl}[1]{|#1|_{\mathrm{syl}}}
\newcommand*{\bigast}{\mathop{\scalebox{1.5}{\raisebox{-0.2ex}{$*$}}}}
\renewcommand*{\subset}{\subseteq}

\newcounter{shcount}
\newcounter{thmcount}

\swapnumbers
\newcommand*{\bsh}[1]{\theoremstyle{definition}\newtheorem{subhead\theshcount}[theorem]{#1}
    \begin{subhead\theshcount}} 
\newcommand*{\esh}{\end{subhead\theshcount}\stepcounter{shcount}} 
 
\newcommand*{\numberedtheorem}[3]{\theoremstyle{plain}\newtheorem*{makethm\thethmcount}{#1}
    \ifthenelse{\equal{#2}{}}{\begin{makethm\thethmcount}#3\end{makethm\thethmcount}\stepcounter{thmcount}}
    {\begin{makethm\thethmcount}[#2]#3\end{makethm\thethmcount}\stepcounter{thmcount}}} 

\makeatletter\newcommand{\linkdest}[1]{\Hy@raisedlink{\hypertarget{#1}{}}}\makeatother 

\usepackage{lmodern,anyfontsize} 
\usepackage{setspace} 
\usepackage{tocloft} 
\usepackage[nottoc]{tocbibind}

%% file: main.bbl
\begin{thebibliography}{BKMM12}

\bibitem[Bal95]{ballmann:lectures}
Werner Ballmann.
\newblock {\em Lectures on spaces of nonpositive curvature}, volume~25 of {\em DMV Seminar}.
\newblock Birkh\"{a}user Verlag, Basel, 1995.
\newblock With an appendix by Misha Brin.

\bibitem[Bas72]{bass:degree}
H.~Bass.
\newblock The degree of polynomial growth of finitely generated nilpotent groups.
\newblock {\em Proc. London Math. Soc. (3)}, 25:603--614, 1972.

\bibitem[BB97]{bestvinabrady:morse}
Mladen Bestvina and Noel Brady.
\newblock Morse theory and finiteness properties of groups.
\newblock {\em Invent. Math.}, 129(3):445--470, 1997.

\bibitem[BBP26a]{baderbensaidpetyt:from}
Shaked Bader, Oussama Bensaid, and Harry Petyt.
\newblock From branching quasiflats to flats in {CAT}(0) cube complexes.
\newblock {\em arXiv:2605.10248}, 2026.

\bibitem[BBP26b]{baderbensaidpetyt:quasiisometric:flexibility}
Shaked Bader, Oussama Bensaid, and Harry Petyt.
\newblock Quasiisometric embeddings between right-angled {A}rtin groups: flexibility.
\newblock {\em arXiv:2605.13314}, 2026.

\bibitem[BBP26c]{baderbensaidpetyt:quasiisometric:products}
Shaked Bader, Oussama Bensaid, and Harry Petyt.
\newblock Quasiisometric embeddings of right-angled {A}rtin groups into products.
\newblock {\em Preprint}, 2026.

\bibitem[BDS07]{buyalodranishnikovschroeder:embedding}
Sergei Buyalo, Alexander Dranishnikov, and Viktor Schroeder.
\newblock Embedding of hyperbolic groups into products of binary trees.
\newblock {\em Invent. Math.}, 169(1):153--192, 2007.

\bibitem[BH99]{bridsonhaefliger:metric}
Martin~R. Bridson and Andr\'{e} Haefliger.
\newblock {\em Metric spaces of non-positive curvature}, volume 319 of {\em Grundlehren der Mathematischen Wissenschaften}.
\newblock Springer-Verlag, Berlin, 1999.

\bibitem[BHS21]{behrstockhagensisto:quasiflats}
Jason Behrstock, Mark Hagen, and Alessandro Sisto.
\newblock Quasiflats in hierarchically hyperbolic spaces.
\newblock {\em Duke Math. J.}, 170(5):909--996, 2021.

\bibitem[BJN10]{behrstockjanuszkiewiczneumann:quasiisometric}
Jason~A. Behrstock, Tadeusz Januszkiewicz, and Walter~D. Neumann.
\newblock Quasi-isometric classification of some high dimensional right-angled {A}rtin groups.
\newblock {\em Groups Geom. Dyn.}, 4(4):681--692, 2010.

\bibitem[BKMM12]{behrstockkleinerminskymosher:geometry}
Jason Behrstock, Bruce Kleiner, Yair~N. Minsky, and Lee Mosher.
\newblock Geometry and rigidity of mapping class groups.
\newblock {\em Geom. Topol.}, 16(2):781--888, 2012.

\bibitem[BKS08]{bestvinakleinersageev:asymptotic}
Mladen Bestvina, Bruce Kleiner, and Michah Sageev.
\newblock The asymptotic geometry of right-angled {A}rtin groups. {I}.
\newblock {\em Geom. Topol.}, 12(3):1653--1699, 2008.

\bibitem[BN12]{behrstockneumann:quasiisometric}
Jason~A. Behrstock and Walter~D. Neumann.
\newblock Quasi-isometric classification of non-geometric 3-manifold groups.
\newblock {\em J. Reine Angew. Math.}, 669:101--120, 2012.

\bibitem[Bow16]{bowditch:quasiisometric}
Brian~H. Bowditch.
\newblock Quasi-isometric maps between direct products of hyperbolic spaces.
\newblock {\em Internat. J. Algebra Comput.}, 26(4):619--633, 2016.

\bibitem[Bow18]{bowditch:large:mapping}
Brian~H. Bowditch.
\newblock Large-scale rigidity properties of the mapping class groups.
\newblock {\em Pacific J. Math.}, 293(1):1--73, 2018.

\bibitem[Bow22]{bowditch:median:book}
Brian~H. Bowditch.
\newblock Median algebras.
\newblock {\em Preprint available at \mbox{bhbowditch.com/papers/median-algebras.pdf}}, 2022.

\bibitem[BS15]{benjaminishamov:bilipschitz}
Itai Benjamini and Alexander Shamov.
\newblock Bi-{L}ipschitz bijections of {$\mathbb Z$}.
\newblock {\em Anal. Geom. Metr. Spaces}, 3(1):313--316, 2015.

\bibitem[Bur65]{burling:oncolouring}
James~P. Burling.
\newblock {\em On colouring problems of families of polytopes}.
\newblock PhD thesis, University of Colorado, Boulder, 1965.

\bibitem[But19]{button:groups}
J~Button.
\newblock Groups acting faithfully on trees and properly on products of trees.
\newblock {\em arXiv:1910.04614}, 2019.

\bibitem[Car15]{carr:two}
Michael Carr.
\newblock {\em Two-generator subgroups of right-angled {A}rtin groups are quasi-isometrically embedded}.
\newblock PhD thesis, Brandeis University. www.proquest.com/docview/1687761198, 2015.

\bibitem[Cha07]{charney:introduction}
Ruth Charney.
\newblock An introduction to right-angled {A}rtin groups.
\newblock {\em Geom. Dedicata}, 125:141--158, 2007.

\bibitem[Cho96]{chow:groups}
Richard Chow.
\newblock Groups quasi-isometric to complex hyperbolic space.
\newblock {\em Trans. Amer. Math. Soc.}, 348(5):1757--1769, 1996.

\bibitem[DJ00]{davisjanuszkiewicz:right}
Michael~W. Davis and Tadeusz Januszkiewicz.
\newblock Right-angled {A}rtin groups are commensurable with right-angled {C}oxeter groups.
\newblock {\em J. Pure Appl. Algebra}, 153(3):229--235, 2000.

\bibitem[DK92]{duchampkrob:lower}
G.~Duchamp and D.~Krob.
\newblock The lower central series of the free partially commutative group.
\newblock {\em Semigroup Forum}, 45(3):385--394, 1992.

\bibitem[DK18]{drutukapovich:geometric}
Cornelia Dru\c{t}u and Michael Kapovich.
\newblock {\em Geometric group theory}, volume~63 of {\em American Mathematical Society Colloquium Publications}.
\newblock American Mathematical Society, Providence, RI, 2018.
\newblock With an appendix by Bogdan Nica.

\bibitem[Dro87a]{droms:graph}
Carl Droms.
\newblock Graph groups, coherence, and three-manifolds.
\newblock {\em J. Algebra}, 106(2):484--489, 1987.

\bibitem[Dro87b]{droms:isomorphisms}
Carl Droms.
\newblock Isomorphisms of graph groups.
\newblock {\em Proc. Amer. Math. Soc.}, 100(3):407--408, 1987.

\bibitem[Dun85]{dunwoody:accessibility}
M.~J. Dunwoody.
\newblock The accessibility of finitely presented groups.
\newblock {\em Invent. Math.}, 81(3):449--457, 1985.

\bibitem[Esk98]{eskin:quasiisometric}
Alex Eskin.
\newblock Quasi-isometric rigidity of nonuniform lattices in higher rank symmetric spaces.
\newblock {\em J. Amer. Math. Soc.}, 11(2):321--361, 1998.

\bibitem[FM99]{farbmosher:quasiisometric}
Benson Farb and Lee Mosher.
\newblock Quasi-isometric rigidity for the solvable {B}aumslag-{S}olitar groups. {II}.
\newblock {\em Invent. Math.}, 137(3):613--649, 1999.

\bibitem[FN20]{fishernguyen:quasiisometric}
David Fisher and Thang Nguyen.
\newblock Quasi-isometric embeddings of non-uniform lattices.
\newblock {\em Comment. Math. Helv.}, 95(1):37--78, 2020.
\newblock With an appendix by S. Garibaldi, D. B. McReynolds, N. Miller and D. Witte Morris.

\bibitem[FW18]{fisherwhyte:quasiisometric}
David Fisher and Kevin Whyte.
\newblock Quasi-isometric embeddings of symmetric spaces.
\newblock {\em Geom. Topol.}, 22(5):3049--3082, 2018.

\bibitem[Gen23]{genevois:algebraic}
Anthony Genevois.
\newblock Algebraic properties of groups acting on median graphs.
\newblock {\em Preprint available at \mbox{sites.google.com/view/agenevois/books}}, 2023.

\bibitem[Gen26]{genevois:polynomial}
Anthony Genevois.
\newblock Polynomial hyperbolicity and products of free groups.
\newblock {\em arXiv:2605.20419}, 2026.

\bibitem[Gro81]{gromov:groups}
Mikhael Gromov.
\newblock Groups of polynomial growth and expanding maps.
\newblock {\em Inst. Hautes \'{E}tudes Sci. Publ. Math.}, 53:53--73, 1981.

\bibitem[Gro93]{gromov:asymptotic}
M.~Gromov.
\newblock Asymptotic invariants of infinite groups.
\newblock In {\em Geometric group theory, {V}ol. 2 ({S}ussex, 1991)}, volume 182 of {\em London Math. Soc. Lecture Note Ser.}, pages 1--295. Cambridge Univ. Press, Cambridge, 1993.

\bibitem[Gui73]{guivarch:croissance}
Yves Guivarc'h.
\newblock Croissance polynomiale et p\'eriodes des fonctions harmoniques.
\newblock {\em Bull. Soc. Math. France}, 101:333--379, 1973.

\bibitem[HK18]{huangkleiner:groups}
Jingyin Huang and Bruce Kleiner.
\newblock Groups quasi-isometric to right-angled {A}rtin groups.
\newblock {\em Duke Math. J.}, 167(3):537--602, 2018.

\bibitem[HM95]{hermillermeier:algorithms}
Susan Hermiller and John Meier.
\newblock Algorithms and geometry for graph products of groups.
\newblock {\em J. Algebra}, 171(1):230--257, 1995.

\bibitem[HO17]{huangosajda:quasieuclidean}
Jingyin Huang and Damian Osajda.
\newblock Quasi-{E}uclidean tilings over 2-dimensional {A}rtin groups and their applications.
\newblock {\em arXiv:1711.00122}, 2017.

\bibitem[HOV24]{huangosajdavaskou:rigidity}
Jingyin Huang, Damian Osajda, and Nicolas Vaskou.
\newblock Rigidity and classification results for large-type {A}rtin groups.
\newblock {\em arXiv:2407.19940}, 2024.

\bibitem[Hua17a]{huang:quasiisometric:1}
Jingyin Huang.
\newblock Quasi-isometric classification of right-angled {A}rtin groups {I}: the finite {O}ut case.
\newblock {\em Geom. Topol.}, 21(6):3467--3537, 2017.

\bibitem[Hua17b]{huang:top}
Jingyin Huang.
\newblock Top-dimensional quasiflats in {CAT}(0) cube complexes.
\newblock {\em Geom. Topol.}, 21(4):2281--2352, 2017.

\bibitem[Hua18]{huang:commensurability}
Jingyin Huang.
\newblock Commensurability of groups quasi-isometric to {RAAG}s.
\newblock {\em Invent. Math.}, 213(3):1179--1247, 2018.

\bibitem[Hua25]{huang:quasiisometric:2}
Jingyin Huang.
\newblock Quasi-isometric classification of right-angled {A}rtin groups {II}: {S}everal infinite {O}ut cases.
\newblock {\em Groups Geom. Dyn.}, 19(4):1165--1261, 2025.

\bibitem[HW99]{hsuwise:onlinear}
Tim Hsu and Daniel~T. Wise.
\newblock On linear and residual properties of graph products.
\newblock {\em Michigan Math. J.}, 46(2):251--259, 1999.

\bibitem[HW08]{haglundwise:special}
Fr\'{e}d\'{e}ric Haglund and Daniel~T. Wise.
\newblock Special cube complexes.
\newblock {\em Geom. Funct. Anal.}, 17(5):1551--1620, 2008.

\bibitem[Kam09]{kambites:oncommuting}
Mark Kambites.
\newblock On commuting elements and embeddings of graph groups and monoids.
\newblock {\em Proc. Edinb. Math. Soc. (2)}, 52(1):155--170, 2009.

\bibitem[KK13]{kimkoberda:embedability}
Sang-Hyun Kim and Thomas Koberda.
\newblock Embedability between right-angled {A}rtin groups.
\newblock {\em Geom. Topol.}, 17(1):493--530, 2013.

\bibitem[KK14a]{kimkoberda:geometry}
Sang-Hyun Kim and Thomas Koberda.
\newblock The geometry of the curve graph of a right-angled {A}rtin group.
\newblock {\em Internat. J. Algebra Comput.}, 24(2):121--169, 2014.

\bibitem[KK14b]{kimkoberda:obstruction}
Sang-Hyun Kim and Thomas Koberda.
\newblock An obstruction to embedding right-angled {A}rtin groups in mapping class groups.
\newblock {\em Int. Math. Res. Not. IMRN}, 14:3912--3918, 2014.

\bibitem[KKL17]{kimkoberdalee:finite}
Sang-Hyun Kim, Thomas Koberda, and Juyoung Lee.
\newblock Finite subgraphs of an extension graph.
\newblock {\em arXiv:1708.02088}, 2017.

\bibitem[KL97]{kleinerleeb:rigidity}
Bruce Kleiner and Bernhard Leeb.
\newblock Rigidity of quasi-isometries for symmetric spaces and {E}uclidean buildings.
\newblock {\em Inst. Hautes \'{E}tudes Sci. Publ. Math.}, 86:115--197, 1997.

\bibitem[Lau95]{laurence:generating}
Michael~R. Laurence.
\newblock A generating set for the automorphism group of a graph group.
\newblock {\em J. London Math. Soc. (2)}, 52(2):318--334, 1995.

\bibitem[Man23]{mangioni:random}
Giorgio Mangioni.
\newblock Random quotients of mapping class groups are quasi-isometrically rigid.
\newblock {\em arXiv:2312.00701}, 2023.

\bibitem[Mar74]{marden:geometry}
Albert Marden.
\newblock The geometry of finitely generated kleinian groups.
\newblock {\em Ann. of Math. (2)}, 99:383--462, 1974.

\bibitem[Mar20]{margolis:quasiisometry}
Alexander Margolis.
\newblock Quasi-isometry classification of right-angled {A}rtin groups that split over cyclic subgroups.
\newblock {\em Groups Geom. Dyn.}, 14(4):1351--1417, 2020.

\bibitem[Mos80]{mostow:onremarkable}
G.~D. Mostow.
\newblock On a remarkable class of polyhedra in complex hyperbolic space.
\newblock {\em Pacific J. Math.}, 86(1):171--276, 1980.

\bibitem[Nai23]{nairne:embeddings}
Patrick~S. Nairne.
\newblock Embeddings of trees, {C}antor sets and solvable {B}aumslag-{S}olitar groups.
\newblock {\em Geom. Dedicata}, 217(2.13):1--26, 2023.

\bibitem[Pan89]{pansu:metriques}
Pierre Pansu.
\newblock M\'{e}triques de {C}arnot-{C}arath\'{e}odory et quasiisom\'{e}tries des espaces sym\'{e}triques de rang un.
\newblock {\em Ann. of Math. (2)}, 129(1):1--60, 1989.

\bibitem[Pet21]{petyt:mapping}
Harry Petyt.
\newblock Mapping class groups are quasicubical.
\newblock {\em arXiv:2112.10681, to appear in Amer. J. Math.}, 2021.

\bibitem[Pra73]{prasad:strong}
Gopal Prasad.
\newblock Strong rigidity of {${\bf Q}$}-rank {$1$} lattices.
\newblock {\em Invent. Math.}, 21:255--286, 1973.

\bibitem[Rul08]{rull:embedding}
Alina Rull.
\newblock {\em Embedding of right-angled {A}rtin and {C}oxeter groups into products of trees}.
\newblock PhD thesis, University of Zurich. doi.org/10.5167/uzh-163804, 2008.

\bibitem[Sch95]{schwartz:quasiisometry}
Richard~Evan Schwartz.
\newblock The quasi-isometry classification of rank one lattices.
\newblock {\em Inst. Hautes \'{E}tudes Sci. Publ. Math.}, 82:133--168 (1996), 1995.

\bibitem[Ser89]{servatius:automorphisms}
Herman Servatius.
\newblock Automorphisms of graph groups.
\newblock {\em J. Algebra}, 126(1):34--60, 1989.

\bibitem[Sta68]{stallings:ontorsionfree}
John~R. Stallings.
\newblock On torsion-free groups with infinitely many ends.
\newblock {\em Ann. of Math. (2)}, 88:312--334, 1968.

\bibitem[Tuk86]{tukia:onquasiconformal}
Pekka Tukia.
\newblock On quasiconformal groups.
\newblock {\em J. Analyse Math.}, 46:318--346, 1986.

\bibitem[Tuk88]{tukia:homeomorphic}
Pekka Tukia.
\newblock Homeomorphic conjugates of {F}uchsian groups.
\newblock {\em J. Reine Angew. Math.}, 391:1--54, 1988.

\bibitem[Wen11]{wenger:asymptotic}
Stefan Wenger.
\newblock The asymptotic rank of metric spaces.
\newblock {\em Comment. Math. Helv.}, 86(2):247--275, 2011.

\bibitem[Wis21]{wise:structure}
Daniel~T. Wise.
\newblock {\em The structure of groups with a quasiconvex hierarchy}, volume 209 of {\em Annals of Mathematics Studies}.
\newblock Princeton University Press, Princeton, NJ, 2021.

\end{thebibliography}
